\documentclass[a4paper]{amsart}
\usepackage[utf8]{inputenc}
\usepackage[T1]{fontenc}

\usepackage{amsmath,amsthm,amssymb,bbm,comment, mathrsfs, mathabx}
\usepackage{thmtools,mathtools}
\usepackage{graphicx,enumitem,stmaryrd}
\usepackage{multicol}
\usepackage{xparse,hyperref}
\usepackage{anyfontsize}
\SetSymbolFont{stmry}{bold}{U}{stmry}{m}{n}
\allowdisplaybreaks[4]
\usepackage[all]{xy}
\usepackage[table]{xcolor}
\usepackage[normalem]{ulem}
\usepackage{xcolor, bm, tensor}
\usepackage{tikz,tikz-cd}
\usetikzlibrary{matrix}
\usetikzlibrary{fit}
\usetikzlibrary{calc}

\newcommand{\xtwoheadrightarrow}[2][]{%
  \xrightarrow[#1]{#2}\mathrel{\mkern-14mu}\rightarrow
}

\makeatletter
\def\slashedarrowfill@#1#2#3#4#5{%
  $\m@th\thickmuskip0mu\medmuskip\thickmuskip\thinmuskip\thickmuskip
   \relax#5#1\mkern-7mu%
   \cleaders\hbox{$#5\mkern-2mu#2\mkern-2mu$}\hfill
   \mathclap{#3}\mathclap{#2}%
   \cleaders\hbox{$#5\mkern-2mu#2\mkern-2mu$}\hfill
   \mkern-7mu#4$%
}
\def\rightslashedarrowfill@{%
  \slashedarrowfill@\relbar\relbar\mapstochar\rightarrow}
\newcommand\xslashedrightarrow[2][]{%
  \ext@arrow 0055{\rightslashedarrowfill@}{#1}{#2}}
\makeatother

\makeatletter
\newcommand{\ostar}{\mathbin{\mathpalette\make@circled\star}}
\newcommand{\make@circled}[2]{%
  \ooalign{$\m@th#1\smallbigcirc{#1}$\cr\hidewidth$\m@th#1#2$\hidewidth\cr}%
}
\newcommand{\smallbigcirc}[1]{%
  \vcenter{\hbox{\scalebox{0.77778}{$\m@th#1\bigcirc$}}}%
}
\makeatother

\usepackage{adjustbox}

\newcommand{\yo}{\text{\usefont{U}{min}{m}{n}\symbol{'110}}}

\DeclareFontFamily{U}{min}{}
\DeclareFontShape{U}{min}{m}{n}{<-> dmjhira}{}

\tikzcdset{scale cd/.style={every label/.append style={scale=#1},
    cells={nodes={scale=#1}}}}

\calclayout

\usepackage[bb=dsserif]{mathalpha}

\usetikzlibrary{calc}
\usetikzlibrary{decorations.pathmorphing}
\tikzset{curve/.style={settings={#1},to path={(\tikztostart)
    .. controls ($(\tikztostart)!\pv{pos}!(\tikztotarget)!\pv{height}!270:(\tikztotarget)$)
    and ($(\tikztostart)!1-\pv{pos}!(\tikztotarget)!\pv{height}!270:(\tikztotarget)$)
    .. (\tikztotarget)\tikztonodes}},
    settings/.code={\tikzset{quiver/.cd,#1}
        \def\pv##1{\pgfkeysvalueof{/tikz/quiver/##1}}},
    quiver/.cd,pos/.initial=0.35,height/.initial=0}

\tikzset{tail reversed/.code={\pgfsetarrowsstart{tikzcd to}}}
\tikzset{2tail/.code={\pgfsetarrowsstart{Implies[reversed]}}}
\tikzset{2tail reversed/.code={\pgfsetarrowsstart{Implies}}}
\tikzset{no body/.style={/tikz/dash pattern=on 0 off 1mm}}

\definecolor{blue(pigment)}{rgb}{0.2, 0.2, 0.6}
\definecolor{americanrose}{rgb}{1.0, 0.01, 0.24}
\definecolor{nicegreen}{rgb}{0.0, 0.5, 0.0}
\definecolor{deepmagenta}{rgb}{0.8, 0.0, 0.8}
\definecolor{deepcarrotorange}{rgb}{0.91, 0.41, 0.17}
\definecolor{cadetgrey}{rgb}{0.57, 0.64, 0.69}

\newtheorem{theoremm}{Theorem}[section]
\newtheorem{theoremmm}{Theorem}
\newtheorem{sketches}{Definition-sketch}
\declaretheorem[style=plain,name=Theorem,numberlike=theoremmm]{theoremat}

\declaretheorem[style=plain,name=Theorem,numberlike=theoremm]{theorem}
\declaretheorem[style=plain,name=Theorem,numbered=no]{theorem*}

\declaretheorem[style=plain,name=Lemma,numberlike=theoremm]{lemma}
\declaretheorem[style=plain,name=Proposition,numberlike=theoremm]{proposition}
\declaretheorem[style=plain,name=Corollary,numberlike=theoremm]{corollary}
\declaretheorem[style=plain,name=Observation,numberlike=theoremm]{observation}
\declaretheorem[style=plain,name=Conjecture,numberlike=theoremm]{conjecture}
\declaretheorem[style=definition,name=Definition,numberlike=theorem]{definition}
\declaretheorem[style=definition,name=Definition-sketch,numberlike=sketches]{defsketch}
\declaretheorem[style=remark,name=Example,numberlike=theorem]{example}
\declaretheorem[style=remark,name=Remark,numberlike=theorem]{remark}
\declaretheorem[style=remark,name=Notation,numberlike=theorem]{notation}
\declaretheorem[style=remark,name=Terminology,numberlike=theorem]{terminology}

\font\sc=rsfs10
\newcommand{\csym}[1]{\sc\mbox{#1}\hspace{1.0pt}}

\font\scc=rsfs7
\newcommand{\ccf}[1]{\scc\mbox{#1}\hspace{0.5pt}}

\font\sccc=rsfs5
\newcommand{\cccsym}[1]{\sccc\mbox{#1}\hspace{0.2pt}}

\newcommand{\on}[1]{\operatorname{#1}}
\newcommand{\setj}[1]{\left\{ #1 \right\}}
\newcommand{\hcomp}{\circ_{h}}

\newcommand{\xiso}{\xrightarrow{\sim}}

\DeclareMathSymbol{\blackdiamond}{\mathbin}{mathb}{"0C}

\DeclareMathAlphabet\EuRoman{U}{eur}{m}{n}
\SetMathAlphabet\EuRoman{bold}{U}{eur}{b}{n}
\newcommand{\euler}{\EuRoman}

\newcommand{\Hom}[1]{\left\langle #1 \right\rangle}
\newcommand{\Cint}{\left\llbracket \csym{C}\, \right\rrbracket}
\newcommand{\Cintjr}{\left\llbracket \ccf{C}\, \right\rrbracket}
\newcommand{\Sint}{\left\llbracket \csym{S}\, \right\rrbracket}

\newcommand{\cotimes}{\overset{\cccsym{C}}{\otimes}}
\newcommand{\kotimes}{\otimes_{\Bbbk}}

\newcommand{\zotimes}{\otimes_{\mathbb{Z}}}

\newcommand{\got}[1]{\left\llbracket #1 \right\rrbracket}
\newcommand{\gotr}[1]{\left\llbracket #1 \right\rrbracket_{\mathbb{R}}}

\usepackage{tikzit}
\tikzstyle{bieski}=[-, draw={rgb,255: red,41; green,28; blue,218}, dashed, thick]
\tikzstyle{thickstrand}=[-, thick]

\begin{document}
\title{Identity in the presence of adjunction}
\author{Mateusz Stroi{\' n}ski}
\begin{abstract}
 We develop a theory of adjunctions in semigroup categories, i.e. monoidal categories without a unit object. We show that a rigid semigroup category is promonoidal, and thus one can naturally adjoin a unit object to it. This extends the previous results of Houston in the symmetric case, and addresses a question of his. It also extends the results in the non-symmetric case with additional finiteness assumptions, obtained by Benson-Etingof-Ostrik, Coulembier and Ko-Mazorchuk-Zhang. We give an interpretation of these results using comonad cohomology, and, in the absence of finiteness conditions, using enriched traces of monoidal categories. As an application of our results, we give a characterization of finite tensor categories in terms of the finitary $2$-representation theory of Mazorchuk-Miemietz.
\end{abstract}

\maketitle

\tableofcontents

\section{Introduction}

One of the most important notions in higher category theory is that of an adjunction internal to a given $n$-category (or $(\infty,n)$-category).
The most familiar, primordial instance of adjunctions is that of an adjunction internal to the bicategory $\mathbf{Cat}$, which gives nothing else but an adjunction of (small) categories, since the bicategorical structures internal to $\mathbf{Cat}$ realize the theory of ordinary (small) categories.

Viewing monoidal categories as special bicategories (under delooping), monoidally dual objects are realized as adjoint $1$-morphisms. This leads to another elementary example: a pair consisting of a finite-dimensional $\Bbbk$-vector space and its $\Bbbk$-linear dual forms an (ambidextrously) adjoint pair.

Another fundamentally important instance of adjunctions is the study of monoidal $(\infty,n)$-categories with duals in the context of homotopy hypothesis (see \cite{BaDo}, \cite{Lu1}, \cite{GrPa}), which in that language states that the category $\on{Bord}_{n}^{\on{fr}}$ of $n$-dimensional extended framed cobordisms is the free symmetric monoidal $(\infty, n)$-category with duals. Here, following \cite{Lu1}, we say that a monoidal $(\infty,n)$-category has duals if its homotopy category, has duals.

We thus see that in the above case (and many others), it suffices to consider duals in monoidal categories and adjunctions in bicategories.
This is also the scope of generality of this paper. We remark that the algebraic notions of a weak $n$-adjunction are also of great interest - see e.g. \cite{Gu1},\cite{Gu2} for their applications to coherence problems in tricategories.

Recall that an adjunction
$\begin{tikzcd}
	{\mathtt{i}} & {\mathtt{j}}
	\arrow[""{name=0, anchor=center, inner sep=0}, "{\mathrm{F}}", shift left, from=1-1, to=1-2]
	\arrow[""{name=1, anchor=center, inner sep=0}, "{\mathrm{G}}", shift left=2, from=1-2, to=1-1]
	\arrow["\dashv"{anchor=center, rotate=-90}, draw=none, from=0, to=1]
\end{tikzcd}$
in a bicategory $\csym{B}$ requires a {\it unit $2$-morphism} $\mathbb{1}_{\mathtt{i}} \Rightarrow \mathrm{G}\circ \mathrm{F}$ and a {\it counit $2$-morphism} $\mathrm{F \circ G} \Rightarrow \mathbb{1}_{\mathtt{j}}$. In particular, this axiomatization of adjointness explicitly requires the presence of unit $1$-morphisms.

In contrast to this, from \cite[Chapter~2.10]{EGNO}, we know that unitality of a monoidal categority is a property, and not an additional structure: the unit object of a monoidal category is unique up to a unique isomorphism. The latter fact indicates that we should be able to axiomatize adjointness in a monoidal category (or a bicategory) without reference to the unit object.

The main aim of this paper is to provide a unit-free axiomatization of adjointness in a semigroup category (i.e. a category endowed with a coherently associative tensor product functor but not necessarily admitting a unit object) which is equivalent to the usual notion of an internal adjunction in the monoidal case, and then study the problem of adjoining units to rigid (i.e. such that every object admits both a left and a right adjoint in the non-unital sense) semigroup categories.
We now sketch the definition mentioned above:

\begin{defsketch}[{Definition~\ref{AdjunctionSemigroup}}]\label{Sketchy}
 Let $(\csym{S},\boxtimes, \mathbf{a})$ be a semigroup category. A {\it non-unital adjunction} in $\csym{S}$ is a tuple $(\mathrm{F},\mathrm{F}^{\blackdiamond}, \eta^{l},\eta^{r}, \varepsilon^{l},\varepsilon^{r})$ consisting of objects $\mathrm{F},\mathrm{F}^{\blackdiamond}$ and natural transformations: units $\eta^{l}: \on{Id}_{\ccf{S}} \Rightarrow - \boxtimes (\mathrm{F}^{\blackdiamond}\boxtimes \mathrm{F})$ and $\eta^{r}: \on{Id}_{\ccf{S}} \Rightarrow (\mathrm{F}^{\blackdiamond}\boxtimes \mathrm{F}) \boxtimes -$, and similarly defined counits, such that:
  \begin{itemize}
     \item $(\mathrm{F} \boxtimes - ,\mathrm{F}^{\blackdiamond} \boxtimes -, \eta^{l}, \varepsilon^{l})$ is an (ordinary) adjunction of left $\csym{S}$-module endofunctors
     \item $- \boxtimes (\mathrm{F}^{\blackdiamond}, -\boxtimes \mathrm{F}, \eta^{r}, \varepsilon^{r})$ is an (ordinary) adjunction of right $\csym{S}$-module endofunctors,
  \end{itemize}
  satisfying further compatibility conditions between the two adjunctions.
\end{defsketch}
Three points should be emphasized: first, our definition explicitly refers to ordinary adjunctions in auxiliary monoidal categories. However, the unit object in both of them (the identity functor) is a much more ``natural'' choice of a unit than one abstractly specified; in some sense, we delegate unitality considerations to ``more canonical'' units living in other categories.

Second, similarly to ordinary adjoints, non-unital adjoints are unique up to isomorphism, see Corollary~\ref{propertyNotStructure}. Thus, rigidity (i.e. admitting left and right duals for all objects) is a condition for a semigroup category, rather than additional structure.

 Third, the adjunctions and transformations involved are not merely adjunctions in $\mathbf{Cat}$, but adjunctions of $\csym{S}$-module endofunctors. Importantly, we cannot reduce to the former case, see \cite{HZ}. The main results of the first part of the paper can be summarized as follows:
\begin{theoremat}[Theorem~\ref{PromonoidalUnital}, Theorem~\ref{Ansatzes}]\label{ThmA}
    Assume $\csym{S}$ is rigid. The category $\widehat{\csym{S}}$ of presheaves on $\csym{S}$ is monoidal, the Yoneda embedding $\csym{S} \hookrightarrow \widehat{\csym{S}}$ is a semigroup embedding with image in $\widehat{\csym{S}}_{\on{rig}}$, the category of rigid objects in $\widehat{\csym{S}}$. The monoidal unit of $\csym{S}$ is given by $\on{Hom}_{\ccf{S}\!\on{-Mod}(\ccf{S},\ccf{S})}(-,\on{Id}_{\ccf{S}})$, which is isomorphic to $\on{Hom}_{\on{Mod-}\ccf{S}(\ccf{S},\ccf{S})}(-,\on{Id}_{\ccf{S}})$
\end{theoremat}
The last statement illustrates an important difference between our axiomatization and that found in \cite{Hou}, where the problem of unit-free adjunctions also is considered, but only in the case $\csym{S}$ is braided or symmetric. It is remarked there that there is no obvious isomorphism between the copresheaf $\on{Hom}_{\ccf{S}\!\on{-Mod}(\ccf{S},\ccf{S})}(\on{Id}_{\ccf{S}},-)$ and the copresheaf $\on{Hom}_{\on{Mod-}\ccf{S}(\ccf{S},\ccf{S})}(\on{Id}_{\ccf{S}},-)$. We explicitly exhibit an isomorphism under our axiomatization.

Let $\csym{T}$ be a cocomplete monoidal category. By the results of \cite{Day}, every semigroup functor $\euler{G}: \csym{S} \rightarrow \csym{T}$ extends to a semigroup functor $\widehat{\euler{G}}:\widehat{\csym{S}} \rightarrow \csym{T}$.
The problem we consider in the second part of this paper is that of finding sufficient conditions for $\widehat{\euler{G}}$ to be monoidal (hence, unital), in the case $\csym{T} = \mathbf{Cat}(\widehat{\mathbf{M}},\widehat{\mathbf{M}})$, where $\mathbf{M}$ is an $\csym{S}$-module category.

In other words, letting $\csym{S}^{\circ} = \csym{S} \sqcup \setj{\mathbb{1}_{\widehat{\ccf{S}}}}$, we study the problem of determining when an $\csym{S}$-module category $\mathbf{M}$ can be extended to a unital $\csym{S}^{\circ}$-module category. This problem (using the language of bicategories rather than monoidal categories) was considered in \cite{KMZ}, in the finitary setting - our approach allows to remove these finiteness restrictions.

We now briefly describe the approach of \cite{KMZ}. There one instead starts by requiring from $\csym{S}$ a weak unitality structure - the existence of so-called lax and oplax units.
A {\it lax unit} in \cite{KMZ} is an object $\mathrm{I}$ of $\csym{S}$ together with natural transformations $l_{\mathrm{I}}: \on{Id}_{\ccf{S}} \Rightarrow \mathrm{I} \boxtimes -$ and $r_{\mathrm{I}}: \on{Id}_{\ccf{S}} \Rightarrow - \boxtimes \mathrm{I}$ satisfying certain axioms. Similarly one defines an oplax unit.

Two objects $\mathrm{F,G} \in \csym{S}$ are said to be adjoint relative to a choice of a lax unit $\mathrm{I}$ and an oplax unit $\mathrm{I'}$ if they satisfy zig-zag equations modified by appropriately replacing the unit object of a monoidal category with the lax and oplax units.
If $\csym{R}$ has a (bilax) unitality structure with respect to which it admits left and right adjoints, it is said to be {\it fiax}. It is then shown (\cite[Proposition~3.13]{KMZ}) that one can adjoin a unit to a fiax category by constructing a bar complex resolving the identity functor on $\csym{S}$.
Taking instead the approach of writing down the bar complex first and deducing the unit from it, one sees that the finiteness assumptions in \cite{KMZ} are used to establish the representability of the bilax unit - this is not necessary when solving the problem of unitality of $\widehat{\csym{S}}$, and thus the finiteness assumptions are not necessary.

There is a direct connection between our setup and that of \cite{KMZ}. For a rigid $\csym{S}$ and any object $\mathrm{F} \in \csym{S}$, the tuple $(\mathrm{F}^{\blackdiamond}\boxtimes \mathrm{F}, \varepsilon^{l}, \varepsilon^{r})$, where $\varepsilon^{l},\varepsilon^{r}$ are as in Definition-sketch~\ref{Sketchy}, gives a lax unit in $\csym{S}$. Further, two objects $(\mathrm{F,G}) \in \csym{S}$ form an adjoint pair in our sense precisely if $\mathrm{F}\boxtimes \mathrm{G}$ is a lax unit, $\mathrm{G}\boxtimes \mathrm{F}$ is an oplax unit, and $(\mathrm{F,G})$ are trivially (unit and counit both being identity) adjoint in the sense of \cite{KMZ}, with respect to that lax unit and oplax unit.

Our results on unital lifts of module categories parallel those of \cite{KMZ}, but, again, the finiteness assumptions can be removed. In both cases, we need to assume that $\csym{S}\star \mathbf{M} = \mathbf{M}$, i.e. that every object of $\mathbf{M}$ is a retract of an object (in the $\Bbbk$-linear case, a finite direct sum of objects) of the form $\mathrm{F}\star X$, for some $\mathrm{F} \in \csym{S}$ and $X \in \mathbf{M}$.
\begin{theorem*}[\cite{KMZ}]
 Let $\csym{S}$ be a fiax semigroup category. A finitary bilax $\csym{S}$-module category $\mathbf{M}$ such that $\csym{S} \star \mathbf{M} = \mathbf{M}$ extends to a (unital) $\overline{\csym{S}}$-module category $\overline{\mathbf{M}}$.
\end{theorem*}
\begin{theoremat}[Theorem~\ref{Extension}] \label{ThmB}
 Let $\csym{S}$ be a rigid semigroup category. An adjointness-preserving $\csym{S}$-module category $\mathbf{M}$ such that $\csym{S} \star \mathbf{M} = \mathbf{M}$ lifts to a unital $\widehat{\csym{S}}$-module category.
\end{theoremat}

Beyond the removal of finiteness assumptions, we should observe that the unital extensions of \cite{KMZ} use {\it abelianizations}, i.e. finite cocompletions, whereas in the potentially infinite setting, we use the significantly bigger categories of presheaves.
Second, the requirement of \cite{KMZ} that the module category be {\it bilax} asks for an additional structure on $\mathbf{M}$, namely, (op)lax unitality transformations $\on{Id}_{\mathbf{M}} \rightarrow \mathbf{M}(\mathrm{I})$ and $\mathbf{M}(\mathrm{I}) \rightarrow \on{Id}_{\mathbf{M}}$.
This requirement is mirrored by our requirement of $\mathbf{M}$ being {\it adjointness-preserving}: if $(\mathrm{F},\mathrm{F}^{\blackdiamond})$ is an adjoint pair in the sense of Definition-sketch~\ref{Sketchy}, then the pair of functors $(\mathbf{M}(\mathrm{F}),\mathbf{M}(\mathrm{F}^{\blackdiamond}))$ should be an adjoint pair.
This is clearly a property of $\mathbf{M}$, and not additional structure on it.
In fact, the proof of Theorem~\ref{ThmB} requires the existence of an adjunction $(\mathbf{M}(\mathrm{F}),\mathbf{M}(\mathrm{F}^{\blackdiamond}),\eta,\varepsilon)$ for which the unit $\eta$ and the counit $\varepsilon$ have additional properties, which would require additional structure, but in Proposition~\ref{Djunctions}, we show that if $(\mathbf{M}(\mathrm{F}),\mathbf{M}(\mathrm{F}^{\blackdiamond}))$ is an adjoint pair, then $\eta,\varepsilon$ having the additional properties can be found - in that sense, Proposition~\ref{Djunctions} can be viewed as a coherence result.

The problem of adjoining a unit (or rather, associating a monoidal category to a semigroup category) in a $\Bbbk$-linear setting with additional finiteness conditions was considered also in \cite{BEO} and \cite{Co}. There, the object of interest is not an abstract semigroup category, but a tensor ideal in a tensor category. The unit is adjoined in a similar way to that in \cite{KMZ}, by constructing a suitable bar complex.
The setup there is slightly less finite than finitary, so more precise technical conditions are identified than those above (however, we remark that these conditions can also be found in \cite{KMZ}).

We interpret the bar complex of \cite{BEO},\cite{Co}, \cite{KMZ} as one coming from the comonad cohomology associated to the coalgebra object $\mathrm{F}^{\blackdiamond} \boxtimes \mathrm{F} \in \csym{S}$, for an object $\mathrm{F}$ satisfying the aforementioned techinical assumptions (in the terminology of \cite{Co}, a {\it strongly faithful object}).
The assumptions about $\mathrm{F}$ may then be identified with those found in \cite{AEBSV}, characterizing the faithfulness of the ``ìnverse image'' functor $\euler{G}^{\ast}: \widehat{\mathcal{D}} \rightarrow \widehat{\mathcal{C}}$ associated to a functor $\euler{G}: \mathcal{C} \rightarrow \mathcal{D}$.
In the presence of such an object, Theorem~\ref{ThmB} is implied by the results of \cite{BEO}, \cite{Co} and \cite{KMZ}; however, our use of comonad cohomology gives more conceptual proofs, involving fewer calculations.

In the absence of finiteness conditions, Theorem~\ref{ThmB} still applies, but we may no longer be able to find a representing coalgebra object $\mathrm{F}^{\blackdiamond} \boxtimes \mathrm{F}$ in $\csym{S}$ for the bar complex we consider.
However, if $\csym{S}$ is symmetric, we provide a different interpretation of the bar complex: we show that the projective presentation obtained by truncating the bar complex is an $\csym{S}$-enriched analogue of the trace of $\csym{S}$ in the sense of \cite{BGHL}.

In the final part of the paper, combining our results with some of the results of \cite{BEO} and \cite{MM5}, we give a description of finite tensor categories in terms of the properties of their semigroup categories of projective objects.
More precisely, given a finite tensor category $\csym{C}$, the category $\csym{C}\!\on{-proj}$ is a finitary rigid semigroup category (it is not monoidal unless $\csym{C}$ is semisimple). We may recover $\csym{C}$ as $[\csym{C}\!\on{-proj}, \mathbf{vec}_{\Bbbk}]$, the category of finite-dimensional presheaves on $\csym{C}\!\on{-proj}$, endowed with the monoidal structure coming from Day convolution. We characterize the finitary rigid semigroup categories $\csym{S}$, such that the finite abelian monoidal category $[\csym{S},\mathbf{vec}_{\Bbbk}]$ is a tensor category:
\begin{theoremat}\label{ThmC}
    $[\csym{S}^{\on{op}},\mathbf{vec}_{\Bbbk}]$ is a finite tensor category if and only if $\csym{S}$ is simple transitive (in the sense of \cite{MM5}) as both a left and a right module category over itself.
\end{theoremat}

For the reader's convenience, we give a self-contained proof of  \cite[Proposition~18(i)]{KM}, which is extremely useful to us, and formulated as Theorem~\ref{MazorchukMiemietzLemma}.

We remark that Theorem~\ref{ThmC} bears some resemblence to the following elementary algebraic statement:
\begin{equation}\label{Claim}
 \text{Let $R$ be a ring with no non-trivial left or right ideals. Then $R$ is a division ring.}
\end{equation}
Indeed, adjoints are sometimes interpreted as ``weak inverses'': for instance, viewing a monoid as a discrete monoidal category, we find that it is a group if and only if the corresponding monoidal category is rigid.
In fact, \cite{NCat} describes the idea of an adjoint as approximating a weak inverse. Theorem~\ref{ThmC} guarantees the existence of {\it adjoints}, provided the absence of non-trivial ideals; similarly, provided the absence of non-trivial ideals, Observation~\ref{Claim} guarantees the existence of {\it inverses}.

The paper is organized as follows:
\begin{itemize}
 \item In Section~\ref{s2}, we first introduce the elementary semigroup categorical notions, illustrate the similarities and contrasts with monoidal categories via various examples, and, crucially, introduce our notion of an adjunction in a semigroup category, in Definition~\ref{AdjunctionSemigroup}. We then establish some properties of adjunctions in semigroup categories, in order to prove Theorem~\ref{ThmA}. Finally, in Subsection~\ref{s21}, we show that the two choices of Ansatz for a unit in $\widehat{\csym{S}}$ given in \cite{Hou} are, in fact, isomorphic.
 \item
In Section~\ref{s3}, we first recall the definition and some of the properties of comonad cohomology, largely following \cite{FGS}; we then recall the notion of a liberal functor, and show that in the presence of an object $\mathrm{F} \in \csym{S}$ such that $\mathrm{F} \otimes -$ is liberal, one can interpret Theorem~\ref{ThmA} in terms of comonad cohomology. In Subsection~\ref{s32}, Theorem~\ref{ThmB} is proved in the presence of the object $\mathrm{F}$ above. Our coherence result for adjunctions, Proposition~\ref{Djunctions}, is also stated and proven therein. In Subsection~\ref{s33}, we extend Theorem~\ref{ThmB} to the case where the object $\mathrm{F}$ does not necessarily exist. In Subsection~\ref{s34}, we show that in the symmetric case our results can be interpreted as calculating the enriched vertical trace of the semigroup category $\csym{S}$.
\item In Subsection~\ref{41}, we give a self-contained proof of \cite[Proposition~18(i)]{KM}. In Subsection~\ref{42}, we prove Theorem~\ref{ThmC}.
\end{itemize}

\vspace{5mm}
\textbf{Acknowledgements.}
The author would like to thank his advisor Volodymyr Mazorchuk for very helpful comments, and for introducing the author to the problem of defining adjunctions in semigroup categories.
Majority of this paper was written during the visit of the author to the University of Oregon. The author would like to thank his host professor at the University of Oregon, Victor Ostrik, for numerous interesting and helpful discussions.
The hospitality of the University of Oregon is gratefully acknowledged.
The visit was partially financed by Stiftelsen
f{\" o}r Vetenskaplig Forskning och Utbildning i Matematik (SVeFUM), Thuns resestipendium, Svenska Matematikersamfundet (SMS) and Kungliga Vetenskapsakademien, to whom the sincere gratitude of the author is extended.

\section{Semigroup categories}\label{s2}

We assume $\Bbbk$ to be a field, and throughout assume all categories and functors to be $\Bbbk$-linear. Throughout the document, we will often denote Hom-spaces in categories using bracket notation: for example, $\Hom{X,Y}_{\mathcal{C}}$ is synonymous to $\on{Hom}_{\mathcal{C}}(X,Y)$.

\begin{definition}
  A {\it semigroup category} is a category $\csym{S}$ together with a functor $\boxtimes: \csym{S} \kotimes \csym{S} \rightarrow \csym{S}$ and a natural isomorphism $\mathsf{a}: \boxtimes \circ (\boxtimes \kotimes \on{Id}_{\ccf{S}}) \xiso \boxtimes \circ (\on{Id}_{\ccf{S}} \kotimes  \boxtimes)$ satisfying the familiar pentagon axiom for monoidal categories - see \cite{EK}, \cite[Remark~2.2.9]{EGNO}. Similarly, we define braided, symmetric and closed semigroup categories. Observe that these notions are easy to define, since none of them explicitly uses the existence of the unit object.
\end{definition}

 From now on, let $\csym{S}$ be a semigroup category. We denote the semigroup category obtained by reversing the tensor product by $\csym{Ś}^{\on{rev}}$.

 Similarly to semigroup categories, one obtains the notion of a semigroup functor and a semigroup transformation, by removing unitality conditions.

\begin{definition}\label{SemigroupModules}
 For a semigroup category $\csym{S}$, a left $\csym{S}$-module category is a category $\mathbf{M}$ together with a functor $\euler{LA}^{\mathbf{M}}: \csym{S} \kotimes \mathbf{M} \rightarrow \mathbf{M}$. For $\mathrm{F} \in \csym{S}$ and $X \in \mathbf{M}$, we denote $\euler{LA}^{\mathbf{M}}((\mathrm{F},X))$ by $\mathrm{F} \star X$. We also require isomorphisms
 \[
  \mathbf{m}_{\mathrm{G,F};X}: (\mathrm{G} \boxtimes \mathrm{F})\star X \xiso \mathrm{G} \star (\mathrm{F} \star X), \text{ natural in } \mathrm{G,F},X,
 \]
 satisfying the multiplicative coherence axiom for module categories over a monoidal category.

 Similarly to the case of monoidal categories, a right $\csym{S}$-module category is defined as a left $\csym{S}^{\on{rev}}$-module category, and, given a further semigroup category $\csym{T}$, a $\csym{S}$-$\csym{T}$-bimodule category is a category which admits a left $\csym{S}$-module category structure and a right $\csym{T}$-module category structure, with the two commuting up to a coherent isomorphism.
\begin{example}
 The functor $\boxtimes$ makes $\csym{S}$ into a $\csym{S}$-$\csym{S}$-bimodule category.
\end{example}

 Again analogously to monoidal categories, by removing unitality axioms we obtain the suitable notions of $\csym{S}$-module functors and transformations, for $\csym{S}$ a semigroup category. We denote the $2$-category of small left $\csym{S}$-module categories, functors and transformations by $\csym{S}\!\on{-Mod}$, and the similarly defined $2$-category of right $\csym{S}$-module categories by $\on{Mod-}\!\csym{S}$.
\end{definition}

 Similarly to semigroups versus monoids, semigroup categories are not as well-behaved as monoidal categories. The following example can be used to easily exhibit semigroup-categorical phenomena that do not occur in the monoidal setting.

\begin{example}\label{ZeroSemigroup}
  Let $\mathcal{A}$ be a category admitting a zero object $\euler{0}_{\mathcal{A}}$. It is easy to verify that setting
  \[
\mathrm{G} \boxtimes \mathrm{F} := \euler{0}_{\mathcal{A}} \text{ and } \mathsf{a}_{\mathrm{H,G,F}} := \on{id}_{\euler{0}_{\mathcal{A}}} \text{ for all } \mathrm{H,G,F} \in \mathcal{A}
  \]
  endows $\mathcal{A}$ with the structure of a semigroup category.
\end{example}

\begin{example}\label{NaiveCat}
 Let $(S,\blackdiamond_{S})$ be a semigroup. The semigroup category $\mathsf{C}^{s}(S)$ is the discrete category on the object set $\on{Ob}\mathsf{C}^{s}(S) := S$, with semigroup structure given by the unique functor $\mathsf{C}^{s}(S) \times \mathsf{C}^{s}(S) \rightarrow \mathsf{C}^{s}(S)$ which equals $\blackdiamond_{S}$ on objects. The free $\Bbbk$-linear category $\mathsf{C}(S)$ on $\mathsf{C}^{s}(S)$ inherits a semigroup category structure.

 In particular, $\mathsf{C}^{s}(S)$ and $\mathsf{C}(S)$ are monoidal categories if and only if $S$ is a monoid.
\end{example}

Recall the following characterization of unit objects in a monoidal category:
\begin{proposition}[{\cite[Chapter~2.2]{EGNO}}]\label{EGNOUnit}
The unit object in a monoidal category $\csym{C}$ is unique up to a unique isomorphism; it is characterized by the combination of the following properties:
\begin{enumerate}
 \item $\mathbb{1} \boxtimes \mathbb{1} \simeq \mathbb{1}$;
 \item $\mathbb{1} \boxtimes -$ and $- \boxtimes \mathbb{1}$ are autoequivalences of $\csym{C}$.
\end{enumerate}

In particular, it is sufficient for $\mathbb{1} \boxtimes -$ and $- \boxtimes \mathbb{1}$ to both be naturally isomorphic to the identity functor.
\end{proposition}

\begin{proposition}\label{2Yoneda}
 There is a semigroup functor
 \[
 \begin{aligned}
 \yo_{\mathbf{B}(\ccf{S})}: \csym{S} &\rightarrow \on{Mod-}\csym{S}(\csym{S},\csym{S}). \\
 \mathrm{F} &\mapsto \mathrm{F} \boxtimes - \text{ for } \mathrm{F} \in \csym{S},\\
 \mathrm{f} &\mapsto \mathrm{f} \boxtimes - \text{ for } \mathrm{f} \in \csym{S}(\mathrm{F,F'}).
 \end{aligned}
 \]
 $\yo_{\mathbf{B}(\ccf{S})}$ is a semigroup equivalence if and only if $\csym{S}$ is monoidal. If $\csym{S}$ is not monoidal, then $\yo_{\mathbf{B}(\ccf{S})}$ is not essentially surjective. It may also fail to be full and faithful.
\end{proposition}

\begin{proof}

 If $\csym{S}$ is monoidal, this fact follows from the bicategorical Yoneda lemma (see e.g. \cite[Section~8.3]{JY}): we consider the delooping bicategory $\mathbf{B}(\csym{S})$ of $\csym{S}$, consisting of a unique object $\mathsf{i}$ whose endomorphism (monoidal) category is $\csym{S}$, and the unique representable pseudofunctor $\mathbf{B}(\csym{S}) \rightarrow \mathbf{Cat}$ is given by $\yo_{\mathbf{B}(\csym{S})}$.

 From the bicategorical Yoneda lemma  it then follows that for any pseudofunctor $\mathbf{M}: \mathbf{B}(\csym{S}) \rightarrow \mathbf{Cat}$ we have $\mathbf{PsFun}(\mathbf{B}(\csym{S})(\mathsf{i},-), \mathbf{M}) \simeq \mathbf{M}(\mathsf{i})$ and the monoidal equivalence in the lemma is the special case $\mathbf{M} = \mathbf{B}(\csym{S})(\mathsf{i},-)$.

 Assume that $\yo_{\mathbf{B}(\ccf{S})}$ is a semigroup equivalence. Choose a quasi-inverse $\yo_{\mathbf{B}(\ccf{S})}^{-1}$. We show that $\yo_{\mathbf{B}(\ccf{S})}^{-1}(\on{Id}_{\ccf{S}})$ is a unit object for $\csym{S}$. Since $\yo_{\mathbf{B}(\ccf{S})}\circ \yo_{\mathbf{B}(\ccf{S})}^{-1} \simeq \on{Id}_{\on{Mod-}\!\ccf{S}(\ccf{S},\ccf{S})}$, we have
 \[
 \yo_{\mathbf{B}(\ccf{S})}^{-1}(\on{Id}_{\ccf{S}}) \boxtimes - \simeq \on{Id}_{\ccf{S}}.
 \]
 Composing the isomorphisms
 \[
  \mathrm{F} \xiso \yo_{\mathbf{B}(\ccf{S})}^{-1}(\mathrm{F} \boxtimes -) = \yo_{\mathbf{B}(\ccf{S})}^{-1}((\mathrm{F} \boxtimes -)\circ \on{Id}_{\ccf{S}}) \simeq \yo_{\mathbf{B}(\ccf{S})}^{-1}(\mathrm{F} \boxtimes -) \circ \yo_{\mathbf{B}(\ccf{S})}^{-1}(\on{Id}_{\ccf{S}}) \simeq \mathrm{F} \boxtimes \yo_{\mathbf{B}(\ccf{S})}^{-1}(\on{Id}_{\ccf{S}}),
 \]
 and observing that they are all natural in $\mathrm{F}$, shows that
 \[
 - \boxtimes \yo_{\mathbf{B}(\ccf{S})}^{-1}(\on{Id}_{\ccf{S}})  \simeq \on{Id}_{\ccf{S}}.
 \]
  Using Proposition~\ref{EGNOUnit}, we conclude that $\csym{S}$ is monoidal.
  \end{proof}

\begin{example}
  $\yo_{\mathbf{B}(\ccf{S})}$ may fail to be faithful.
\end{example}

  \begin{proof}
   Let $\mathcal{A}$ be a category admitting a zero object and a non-zero object and endow it with the zero semigroup category structure given in Example~\ref{ZeroSemigroup}. We have $(\yo_{\mathbf{B}(\mathcal{A})})_{A,A'} = 0$ for all $A,A' \in \on{Ob}\mathcal{A}$, so the functor $\yo_{\mathbf{B}(\mathcal{A})}$ is not faithful.
  \end{proof}

  \begin{example}
  $\yo_{\mathbf{B}(\ccf{S})}$ may fail to be full.
  \end{example}

  \begin{proof}
   Consider the semigroup $S = \setj{y,0}$ where $0$ is a zero element and $y^{2} = y$. Denote by $\csym{S}$ the $\Bbbk$-linear category $\mathsf{C}(S)$ described in Example~\ref{NaiveCat}.

   Then $\on{Mod-}\csym{S}(\csym{S},\csym{S})(-\boxtimes y,-\boxtimes y) \simeq \csym{S}(0,0) \oplus \csym{S}(y,y)$. Indeed, $\csym{S}$ is discrete $\Bbbk$-linear, so any pair of morphisms $(\sigma_{y}: y\boxtimes y \rightarrow y \boxtimes y, \sigma_{0}: 0 \boxtimes y \rightarrow 0 \boxtimes y)$ gives a natural transformation, which is easily checked to be a right $\csym{S}$-module transformation. On the other hand, $\csym{S}(y,y)$ is one-dimensional, and so $\yo_{\mathbf{B}(\ccf{S})}$ cannot be full since $(\yo_{\mathbf{B}(\ccf{S})})_{y}$ cannot be surjective due to the dimension of its codomain being strictly greater than that of its domain.

   To obtain a similar example where the $0$ element of the semigroup corresponds to the zero object, extend of $S$ to $S' = \setj{x,y,0}$, where $xa = ax = 0$ for all $a \in S'$; setting $\on{Hom}(0,0) = \setj{0}$ and $\on{Hom}(x,x) = \Bbbk$, one obtains a similar dimension inequality.
  \end{proof}

  We remark that the following definition agrees with the one given in \cite[Definition~2.10.1]{EGNO}, however, we decorate the right dual of an object on the right, writing $\mathrm{F}^{\blackdiamond}$ for the right dual of $\mathrm{F}$; \cite{EGNO} would instead denote the right dual by ${}^{\ast}\mathrm{F}$, decorating on the left.
\begin{definition}{\cite[Definition~2.10.1]{EGNO}}
 The {\it right dual} of an object $\mathrm{F}$ in a monoidal category $\csym{C}$ is an object $\mathrm{F}^{\blackdiamond}$ such that there are morphisms $\eta^{\mathrm{F}}: \mathbb{1}_{\ccf{C}} \rightarrow \mathrm{F}^{\blackdiamond}\mathrm{F}$ and $\varepsilon^{\mathrm{F}}: \mathrm{F}\mathrm{F}^{\blackdiamond} \rightarrow \mathbb{1}_{\ccf{C}}$, satisfying the zig-zag equations analogous to those for the unit and counit of an adjunction. The {\it left dual} of $\mathrm{F}$ is a right dual ${}^{\blackdiamond}\mathrm{F}$ for $\mathrm{F}$ in $\csym{C}^{\on{rev}}$. The object $\mathrm{F}$ is said to be {\it rigid} if it admits both a left and a right dual.

 The monoidal category $\csym{C}$ is said to be {\it rigid} if all of its objects are rigid. Taking right duals defines an anti-autoequivalence $(-)^{\blackdiamond}: \csym{C}^{\on{rev,opp}} \xiso \csym{C}$, a quasi-inverse for which is given by taking left duals.
\end{definition}

Duals (and more generally, adjunctions in bicategories) are preserved by all monoidal functors (pseudofunctors in the bicategorical case). As an immediate consequence we find the following:
\begin{observation}\label{RigidMonoidal}
 Let $\csym{C}$ be a monoidal category and let $\mathrm{F}$ be an object of $\csym{C}$ admitting a right dual $\mathrm{F}^{\blackdiamond}$. The pair of functors $(\mathrm{F}\boxtimes -,\mathrm{F}^{\blackdiamond}\boxtimes -)$ is an adjoint pair. Conversely, if $\mathrm{F}\boxtimes -$ admits a right dual $\Phi$ inside $\on{Mod-}\csym{C}(\csym{C},\csym{C})$, then, by Lemma~\ref{2Yoneda}, $\Phi(\mathbb{1}_{\ccf{C}})$ is a right dual to $\mathrm{F}$.
\end{observation}

\begin{remark}
 It is crucial that $\Phi$ in Observation~\ref{RigidMonoidal} lies in $\on{Mod-}\csym{C}(\csym{C},\csym{C})$. Requiring just an adjunction in $\mathbf{Cat}(\csym{C},\csym{C})$ gives us the notion of internal homs, and even if such a left adjoint were to be of the form $\mathrm{F}^{\circ} \boxtimes -$, for some $\mathrm{F}^{\circ} \in \csym{C}$, this would not suffice to guarantee $\mathrm{F}^{\circ}$ to be left dual to $\mathrm{F}$ in $\csym{C}$; see \cite{HZ}.
\end{remark}

\begin{remark}\label{CounterExamples}
 In view of Observation~\ref{RigidMonoidal}, we could try to naïvely define duals in a semigroup category $\csym{S}$ as duals in $\on{Mod-}\!\csym{S}(\csym{S},\csym{S})$. However, elementary properties of duals, which can be deduced from Observation~\ref{RigidMonoidal} and Lemma~\ref{2Yoneda}, such as uniqueness of duals, do not necessarily hold under this definition. For a counter-example, again consider a category $\mathcal{A}$ admitting a zero object and a non-zero object and endow it with the zero semigroup category structure of Example~\ref{ZeroSemigroup}. The zero functor is self-adjoint, and so, under this naïve definition, every object in $\mathcal{A}$ is the two-sided dual of every other object.

 Observe also that in this case, for any object $X$ and any choice of a right dual $X^{\blackdiamond}$, the components of the unit and counit of the adjunction $(X \boxtimes -, X^{\blackdiamond} \boxtimes -, \eta^{X}, \varepsilon^{X})$ can be written as images of morphisms in $\mathcal{A}$ under $\yo_{\mathbf{B}(\mathcal{A})}$. Thus, modifying the definition so that it also requires this condition still does not suffice for a well-behaved definition of a dual in a semigroup category.
\end{remark}

\begin{definition}\label{AdjunctionSemigroup}
 Let $\csym{S}$ be a semigroup category. An {\it adjunction} in $\csym{S}$ is a tuple $(\mathrm{F},\mathrm{F}^{\blackdiamond}, \eta^{l},\eta^{r}, \varepsilon^{l},\varepsilon^{r})$, where:
 \begin{enumerate}
  \item $\mathrm{F}$ and $\mathrm{F}^{\blackdiamond}$ are objects of $\csym{S}$;
  \item $\eta^{l}: \on{Id}_{\ccf{S}} \rightarrow - \boxtimes (\mathrm{F}^{\blackdiamond} \boxtimes \mathrm{F})$ and $\varepsilon^{l}: - \boxtimes (\mathrm{F} \boxtimes \mathrm{F}^{\blackdiamond}) \rightarrow \on{Id}_{\ccf{S}}$ are left $\csym{S}$-module transformations;
  \item $\eta^{r}: \on{Id}_{\ccf{S}} \rightarrow  (\mathrm{F}^{\blackdiamond} \boxtimes \mathrm{F}) \boxtimes -$ and $\varepsilon^{r}:  (\mathrm{F} \boxtimes \mathrm{F}^{\blackdiamond}) \boxtimes - \rightarrow \on{Id}_{\ccf{S}}$ are right $\csym{S}$-module transformations.
 \end{enumerate}
 This data is subject to the following axioms:
 \begin{enumerate}[label = (\Roman*)]
  \item \label{LeftAdj} The tuple $(-\boxtimes \mathrm{F}^{\blackdiamond},-\boxtimes \mathrm{F},\eta^{l},\varepsilon^{l})$ is an adjunction in $\csym{S}\!\on{-Mod}(\csym{S},\csym{S})$;
  \item \label{RightAdj} The tuple $(\mathrm{F} \boxtimes -,\mathrm{F}^{\blackdiamond}\boxtimes -,\eta^{l},\varepsilon^{l})$ is an adjunction in $\on{Mod-}\!\csym{S}(\csym{S},\csym{S})$;
  \item \label{ZigZag} The following diagrams commute:
\[\begin{tikzcd}[ampersand replacement=\&]
	{\mathrm{F}} \& {\mathrm{F} \boxtimes (\mathrm{F}^{\blackdiamond}\boxtimes \mathrm{F})} \& {(\mathrm{F} \boxtimes \mathrm{F}^{\blackdiamond})\boxtimes \mathrm{F}} \\
	\&\& {\mathrm{F}}
	\arrow["{\eta^{l}_{\mathrm{F}}}", from=1-1, to=1-2]
	\arrow["{\mathsf{a}^{-1}_{\mathrm{F,F^{\blackdiamond},F}}}", from=1-2, to=1-3]
	\arrow["{\varepsilon^{r}_{\mathrm{F}}}", from=1-3, to=2-3]
	\arrow["{\on{id}_{\mathrm{F}}}"', from=1-1, to=2-3]
\end{tikzcd}\]
and
\[\begin{tikzcd}[ampersand replacement=\&]
	{\mathrm{F}^{\blackdiamond}} \& {(\mathrm{F}^{\blackdiamond}\boxtimes \mathrm{F}) \boxtimes \mathrm{F}^{\blackdiamond}} \& {\mathrm{F}^{\blackdiamond}\boxtimes (\mathrm{F} \boxtimes \mathrm{F}^{\blackdiamond})} \\
	\&\& {\mathrm{F}^{\blackdiamond}}
	\arrow["{\eta^{r}_{\mathrm{F}}}", from=1-1, to=1-2]
	\arrow["{\mathsf{a}_{\mathrm{F^{\blackdiamond},F,F^{\blackdiamond}}}}", from=1-2, to=1-3]
	\arrow["{\varepsilon^{l}_{\mathrm{F}}}", from=1-3, to=2-3]
	\arrow["{\on{id}_{\mathrm{F}^{\blackdiamond}}}"', from=1-1, to=2-3]
\end{tikzcd}\]
\item \label{Involution} For all $\mathrm{H,K} \in \csym{S}$, the following diagrams commute:
\[\begin{tikzcd}[ampersand replacement=\&]
	{\mathrm{H \boxtimes K}} \&\& {\mathrm{H \boxtimes ((F^{\blackdiamond} \boxtimes F) \boxtimes K)}} \\
	\&\& {\mathrm{(H \boxtimes (F^{\blackdiamond} \boxtimes F))\boxtimes K}}
	\arrow["{\mathrm{H \boxtimes \eta^{r}_{\mathrm{K}}}}", from=1-1, to=1-3]
	\arrow["{\mathsf{a}_{\mathrm{H,F^{\blackdiamond}\boxtimes F, K}}}", from=1-3, to=2-3]
	\arrow["{\eta^{l}_{\mathrm{H}}\boxtimes \mathrm{K}}"', from=1-1, to=2-3]
\end{tikzcd}\]
and
\[\begin{tikzcd}[ampersand replacement=\&]
	{\mathrm{(H \boxtimes (F\boxtimes F^{\blackdiamond}))\boxtimes K}} \&\& {\mathrm{H \boxtimes K}} \\
	{\mathrm{H \boxtimes ((F \boxtimes F^{\blackdiamond}) \boxtimes K)}}
	\arrow["{\mathrm{H \boxtimes \varepsilon^{r}_{\mathrm{K}}}}"', from=2-1, to=1-3]
	\arrow["{\mathsf{a}_{\mathrm{H,F^{\blackdiamond}\boxtimes F, K}}}"', from=1-1, to=2-1]
	\arrow["{\varepsilon^{l}_{\mathrm{H}}\boxtimes \mathrm{K}}", from=1-1, to=1-3]
\end{tikzcd}\]
 \end{enumerate}
\end{definition}

\begin{remark}\label{SplitMono}
 As a consequence of Axiom~\ref{ZigZag}, we see that the morphisms $\eta^{l,\mathrm{F}}_{\mathrm{F}}$ and $\eta^{r,\mathrm{F}}_{\mathrm{F}}$ are split monomorphisms, and that the morphisms $\varepsilon^{l,\mathrm{F}}_{\mathrm{F}},\varepsilon^{r,\mathrm{F}}_{\mathrm{F}}$ are split epimorphisms.
\end{remark}

\begin{remark}
  We have stated our axioms in the setting of a non-strict semigroup category, however, we will omit associators and unitors in our proofs to make them more readable - they can be easily reinsterted, and in fact all our results follow generally from the strict case as soon as we have established Proposition~\ref{StrictifyingEmbedding}, which gives a strictification.
\end{remark}

\begin{proposition}\label{StrictifyingEmbedding}
 If $\csym{S}$ is a semigroup category such that all objects of $\csym{S}$ admit right duals, then the semigroup functor $\yo_{\mathbf{B}(\ccf{S})}: \csym{S} \rightarrow \on{Mod-}\!\csym{S}(\csym{S},\csym{S})$ is full and faithful.

  If the objects of $\csym{S}$ instead admit left duals, then the semigroup functor $\yo_{\mathbf{B}(\ccf{S})}^{\on{rev}}: \csym{S}^{\on{rev}} \rightarrow \csym{S}\!\on{-Mod}(\csym{S},\csym{S})$ is full and faithful.
\end{proposition}

\begin{proof}
 We only show the first claim; the second is dual to the first.
 We first show that $\yo_{\mathbf{B}(\ccf{S})}$ is faithful. Let $\mathrm{F,G} \in \csym{S}$ and let $\mathrm{f,f'} \in \Hom{\mathrm{F,G}}_{\ccf{S}}$. Assume that $\mathrm{f} \boxtimes - = \mathrm{f'} \boxtimes -$. In particular $\mathrm{f} \boxtimes (\mathrm{G}^{\blackdiamond} \boxtimes \mathrm{G}) = \mathrm{f'} \boxtimes (\mathrm{G}^{\blackdiamond} \boxtimes \mathrm{G})$. It now follows that
 \[
  \begin{aligned}
   \mathrm{f} = \varepsilon^{r,\mathrm{G}}_{\mathrm{G}} \circ \eta^{l,\mathrm{G}}_{\mathrm{G}} \circ \mathrm{f} = \varepsilon^{l,\mathrm{G}}_{\mathrm{G}} \circ (\mathrm{f \boxtimes G^{\blackdiamond} \boxtimes G}) \circ \eta^{l,\mathrm{G}}_{\mathrm{F}} = \varepsilon^{l,\mathrm{G}}_{\mathrm{G}} \circ (\mathrm{f' \boxtimes G^{\blackdiamond} \boxtimes G}) \circ \eta^{l,\mathrm{G}}_{\mathrm{F}} = \varepsilon^{r,\mathrm{G}}_{\mathrm{G}} \circ \eta^{l,\mathrm{G}}_{\mathrm{G}} \circ \mathrm{f'} = \mathrm{f'},
  \end{aligned}
 \]
 where the first and last equality follow by Axiom~\ref{ZigZag}, the second and the penultimate ones follow by the naturality of $\eta^{l,\mathrm{G}}$, and the middle one follows from the assumption.

 Now we show that $\yo_{\mathbf{B}(\ccf{S})}$ is full. Let $\alpha: \mathrm{F} \boxtimes - \rightarrow \mathrm{G} \boxtimes -$. We claim that
 \begin{equation}\label{Fullness}
\alpha_{\mathrm{X}} = (\varepsilon^{r,\mathrm{G}}_{\mathrm{G}} \circ \alpha_{\mathrm{G^{\blackdiamond}\boxtimes \mathrm{G}}} \circ \eta^{l,\mathrm{G}}_{\mathrm{F}}) \boxtimes \mathrm{X}.
 \end{equation}
 This shows fullness, since then $\alpha = (\varepsilon^{r,\mathrm{G}}_{\mathrm{G}} \circ \alpha_{\mathrm{G^{\blackdiamond}\boxtimes \mathrm{G}}} \circ \eta^{l,\mathrm{G}}_{\mathrm{F}}) \boxtimes -$. To show the claim, consider the diagram
\begin{equation}\label{DiagramFullness}
\begin{tikzcd}[ampersand replacement=\&]
	{\mathrm{F}\boxtimes \mathrm{X}} \&\& {\mathrm{G}\boxtimes \mathrm{X}} \& {\mathrm{G}\boxtimes \mathrm{X}} \\
	{\mathrm{F}\boxtimes \mathrm{G}^{\blackdiamond}\boxtimes \mathrm{G} \boxtimes \mathrm{X}} \&\& {\mathrm{G}\boxtimes \mathrm{G}^{\blackdiamond}\boxtimes \mathrm{G} \boxtimes \mathrm{X}}
	\arrow[""{name=0, anchor=center, inner sep=0}, "{\alpha_{X}}", from=1-1, to=1-3]
	\arrow["{\mathrm{F}\boxtimes \eta^{r,\mathrm{G}}_{X}}"', from=1-1, to=2-1]
	\arrow[""{name=1, anchor=center, inner sep=0}, "{\alpha_{\mathrm{G}^{\blackdiamond}\boxtimes \mathrm{G} \boxtimes X}}"', from=2-1, to=2-3]
	\arrow["{\mathrm{G}\boxtimes \eta^{r,\mathrm{G}}_{\mathrm{X}}}"', from=1-3, to=2-3]
	\arrow["{\on{id}_{\mathrm{G\boxtimes X}}}", from=1-3, to=1-4]
	\arrow[""{name=2, anchor=center, inner sep=0}, "{\varepsilon^{r,\mathrm{G}}_{\mathrm{G \boxtimes X}}}"', from=2-3, to=1-4]
	\arrow["{(2)}"{description}, draw=none, from=1-3, to=2]
	\arrow["{(1)}"{description}, draw=none, from=0, to=1]
\end{tikzcd}
\end{equation}
Face (1) commutes due to the naturality of $\alpha$ and face (2) commutes by the unit-counit identities for $\eta^{r,\mathrm{G}}, \varepsilon^{r,\mathrm{G}}$.

Now we observe that, by Axiom~\ref{Involution}, we have $\mathrm{F} \boxtimes \eta^{r,\mathrm{G}}_{X} = \eta^{l,\mathrm{G}}_{\mathrm{F}} \boxtimes \mathrm{X}$, and that, since $\alpha$ and $\varepsilon^{r,\mathrm{G}}$ are $\csym{S}$-module transformations, we have $\alpha_{\mathrm{G^{\blackdiamond} \boxtimes G \boxtimes X}} = \alpha_{\mathrm{G^{\blackdiamond} \boxtimes G}} \boxtimes \mathrm{X}$ and $\varepsilon^{r,\mathrm{G}}_{\mathrm{G}\boxtimes \mathrm{X}} = \varepsilon^{r,\mathrm{G}}_{\mathrm{G}}\boxtimes \mathrm{X}$. Combining these with the equation given by the outer face of Diagram~\eqref{DiagramFullness} gives Equation~\eqref{Fullness}.
\end{proof}

\begin{proposition}\label{AdjunctionsCompose}
 Given adjunctions $(\mathrm{F},\mathrm{F}^{\blackdiamond}, \eta^{l,\mathrm{F}},\eta^{r,\mathrm{F}}, \varepsilon^{l,\mathrm{F}},\varepsilon^{r,\mathrm{F}})$ and $(\mathrm{G},\mathrm{G}^{\blackdiamond}, \eta^{l,\mathrm{G}},\eta^{r,\mathrm{G}}, \varepsilon^{l,\mathrm{G}},\varepsilon^{r,\mathrm{G}})$ in $\csym{S}$, the tuple $(\mathrm{G}\boxtimes \mathrm{F},\mathrm{F}^{\blackdiamond} \boxtimes \mathrm{G}^{\blackdiamond}, \eta^{l,\mathrm{G}\boxtimes \mathrm{F}},\eta^{r,\mathrm{G}\boxtimes \mathrm{F}}, \varepsilon^{l,\mathrm{G}\boxtimes\mathrm{F}},\varepsilon^{r,\mathrm{G}\boxtimes \mathrm{F}})$, defined by
 \begin{multicols}{2}
 \begin{enumerate}
  \item $\eta^{l,\mathrm{G}\boxtimes \mathrm{F}}_{\mathrm{X}} = (\eta^{l,\mathrm{G}}_{\mathrm{X} \boxtimes \mathrm{F}^{\blackdiamond}} \boxtimes \mathrm{F}) \circ \eta^{l,\mathrm{F}}_{\mathrm{X}}$;
  \item $\varepsilon^{l,\mathrm{G}\boxtimes \mathrm{F}}_{\mathrm{X}} = \varepsilon^{l,\mathrm{G}}_{X} \circ (\varepsilon^{l,\mathrm{F}}_{\mathrm{X \boxtimes G}} \boxtimes \mathrm{G}^{\blackdiamond})$;
  \item $\eta^{r,\mathrm{G}\boxtimes \mathrm{F}}_{\mathrm{X}} = (\mathrm{F}^{\blackdiamond}\boxtimes \eta^{r,\mathrm{G}}_{\mathrm{F} \boxtimes \mathrm{X}}) \circ \eta^{r,\mathrm{F}}_{\mathrm{X}}$;
  \item $\varepsilon^{r,\mathrm{G}\boxtimes \mathrm{F}}_{\mathrm{X}} = \varepsilon^{r,\mathrm{G}}_{X} \circ (\mathrm{G} \boxtimes \varepsilon^{r,\mathrm{F}}_{\mathrm{G^{\blackdiamond} \boxtimes X}}) $,
 \end{enumerate}
\end{multicols}
 is an adjunction in $\csym{S}$.
\end{proposition}

\begin{proof}
 Axioms~\ref{LeftAdj} and \eqref{RightAdj} follow immediately from the fact that $\eta^{l,\mathrm{G}\boxtimes \mathrm{F}}, \varepsilon^{l,\mathrm{G}\boxtimes \mathrm{F}}$ are given by composition of adjunctions in the monoidal category $\csym{S}\!\on{-Mod}(\csym{S},\csym{S})$, and similarly for $\eta^{r,\mathrm{G}\boxtimes \mathrm{F}}, \varepsilon^{r,\mathrm{G}\boxtimes \mathrm{F}}$.

 To verify the first diagram in Axiom~\ref{ZigZag}, consider the following diagram:
\[\begin{tikzcd}[ampersand replacement=\&]
	{\mathrm{G \boxtimes F}} \&\& {\mathrm{G \boxtimes F \boxtimes F^{\blackdiamond} \boxtimes F}} \&\& {\mathrm{G \boxtimes F \boxtimes F^{\blackdiamond} \boxtimes G^{\blackdiamond} \boxtimes G \boxtimes F}} \\
	\&\& {\mathrm{G \boxtimes F}} \&\& {\mathrm{G \boxtimes G^{\blackdiamond} \boxtimes G \boxtimes F}} \\
	\&\&\&\& {\mathrm{G \boxtimes F}}
	\arrow[""{name=0, anchor=center, inner sep=0}, from=1-1, to=2-3, equal]
	\arrow["{\eta^{l,\mathrm{F}}_{\mathrm{G \boxtimes F}}}", from=1-1, to=1-3]
	\arrow[""{name=1, anchor=center, inner sep=0}, "{\eta^{l,\mathrm{G}}_{\mathrm{G \boxtimes F \boxtimes F^{\blackdiamond}}} \boxtimes \mathrm{F}}", from=1-3, to=1-5]
	\arrow["{\mathrm{G} \boxtimes \varepsilon^{r,\mathrm{F}}_{\mathrm{F}}}"', from=1-3, to=2-3]
	\arrow["{\mathrm{G}\boxtimes \varepsilon^{r,\mathrm{F}}_{\mathrm{GG^{\blackdiamond}F}}}", from=1-5, to=2-5]
	\arrow[""{name=2, anchor=center, inner sep=0}, "{\mathrm{G} \boxtimes \eta^{r,\mathrm{G}}_{\mathrm{F}}}", from=2-3, to=2-5]
	\arrow["{\varepsilon^{r,\mathrm{G}}_{\mathrm{G \boxtimes F}}}", from=2-5, to=3-5]
	\arrow[""{name=3, anchor=center, inner sep=0}, from=2-3, to=3-5, equal]
	\arrow["{(1)}"{description}, draw=none, from=1-3, to=0]
	\arrow["{(3)}"{description}, draw=none, from=2-5, to=3]
	\arrow["{(2)}"{description, pos=0.3}, draw=none, from=1, to=2]
\end{tikzcd}\]
 Face (1) commutes since $\mathrm{G}\boxtimes \varepsilon^{r,\mathrm{F}}_{\mathrm{F}}\circ \eta^{l,\mathrm{F}}_{\mathrm{G \boxtimes F}} = \mathrm{G}\boxtimes \varepsilon^{r,\mathrm{F}}_{\mathrm{F}}\circ \mathrm{G}\boxtimes \eta^{l,\mathrm{F}}_{\mathrm{F}} = \mathrm{G} \boxtimes (\varepsilon^{r,\mathrm{F}}_{\mathrm{F}}\circ \eta^{l,\mathrm{F}}_{\mathrm{F}}) = \mathrm{G} \boxtimes \on{id}_{\mathrm{F}}$, where the last equality follows from Axiom~\ref{ZigZag}.

 Face (2) commutes since $\eta^{l,\mathrm{G}}_{\mathrm{G \boxtimes F \boxtimes F^{\blackdiamond}}} \boxtimes \mathrm{F} = \mathrm{G \boxtimes F \boxtimes F^{\blackdiamond}} \boxtimes \eta^{r,\mathrm{G}}_{\mathrm{F}}$ and since the diagram
\begin{equation}\label{ComposeAdjunctions}
\begin{tikzcd}[ampersand replacement=\&]
	{\mathrm{G \boxtimes F \boxtimes F^{\blackdiamond} \boxtimes F}} \&\& {\mathrm{G \boxtimes F \boxtimes F^{\blackdiamond} \boxtimes G^{\blackdiamond} \boxtimes G \boxtimes F}} \\
	{\mathrm{G \boxtimes F}} \&\& {\mathrm{G \boxtimes G^{\blackdiamond} \boxtimes G \boxtimes F}}
	\arrow["{\mathrm{G \boxtimes F \boxtimes F^{\blackdiamond}}\boxtimes \eta^{r,\mathrm{G}}_{\mathrm{F}}}", from=1-1, to=1-3]
	\arrow["{\mathrm{G} \boxtimes \varepsilon^{r,\mathrm{F}}_{\mathrm{F}}}"', from=1-1, to=2-1]
	\arrow["{\mathrm{G}\boxtimes \varepsilon^{r,\mathrm{F}}_{\mathrm{GG^{\blackdiamond}F}}}", from=1-3, to=2-3]
	\arrow["{\mathrm{G}\boxtimes \eta^{r,\mathrm{G}}_{\mathrm{F}}}", from=2-1, to=2-3]
\end{tikzcd}
\end{equation}
commutes.
The commutativity of Diagram~\eqref{ComposeAdjunctions} follows from naturality of $\mathrm{G}\boxtimes \varepsilon^{r,\mathrm{F}}$.
Face (3) commutes by unit-counit identities for $\eta^{r,\mathrm{G}},\varepsilon^{r,\mathrm{G}}$.

The second diagram in Axiom~\ref{ZigZag} can be checked analogously to the first one.

The commutativity of the first diagram in Axiom~\ref{Involution} follows from the commutativity of
\[\begin{tikzcd}[ampersand replacement=\&]
	{\mathrm{H \boxtimes K}} \&\& {\mathrm{H \boxtimes F^{\blackdiamond} \boxtimes F \boxtimes K}} \&\&\& {\mathrm{H \boxtimes F^{\blackdiamond} \boxtimes G^{\blackdiamond} \boxtimes G \boxtimes F \boxtimes K}} \\
	{\mathrm{H \boxtimes K}} \&\& {\mathrm{H \boxtimes F^{\blackdiamond} \boxtimes F \boxtimes K}} \&\&\& {\mathrm{H \boxtimes F^{\blackdiamond} \boxtimes G^{\blackdiamond} \boxtimes G \boxtimes F \boxtimes K}}
	\arrow["{\mathrm{H}\boxtimes \eta^{r,\mathrm{F}}_{\mathrm{K}}}", from=1-1, to=1-3]
	\arrow["{\mathrm{H\boxtimes F^{\blackdiamond}\boxtimes \eta^{r,\mathrm{G}}}_{\!\mathrm{F\boxtimes K}}}"{pos=0.4}, from=1-3, to=1-6]
	\arrow[from=1-1, to=2-1]
	\arrow[from=1-3, to=2-3]
	\arrow[from=1-6, to=2-6]
	\arrow["{\eta^{l,\mathrm{F}}_{\mathrm{H}} \boxtimes \mathrm{K}}"', from=2-1, to=2-3]
	\arrow["{\eta^{l,\mathrm{G}}_{\mathrm{H \boxtimes F^{\blackdiamond}}}\boxtimes \mathrm{F \boxtimes K}}"', from=2-3, to=2-6]
\end{tikzcd}\]
whose faces commute by applying Axiom~\ref{Involution} to $(\mathrm{F},\mathrm{F}^{\blackdiamond}, \eta^{l,\mathrm{F}},\eta^{r,\mathrm{F}}, \varepsilon^{l,\mathrm{F}},\varepsilon^{r,\mathrm{F}})$ and $(\mathrm{G},\mathrm{G}^{\blackdiamond}, \eta^{l,\mathrm{G}},\eta^{r,\mathrm{G}}, \varepsilon^{l,\mathrm{G}},\varepsilon^{r,\mathrm{G}})$. The second diagram in Axiom~\ref{Involution} follows analogously to the first one.
\end{proof}

\begin{definition}
 An object of a semigroup category is {\it rigid} if it admits both left and right duals. A {\it rigid semigroup category} is a semigroup category all of whose objects are rigid.
\end{definition}

In Example~\ref{FreeStrings} below, we define a ``free'' rigid strict semigroup category on a self-dual generator.
Indeed, using the coherence theorem for monoidal categories, one can show that the category of string diagrams with ``cups'' and ``caps'' is the free rigid strict monoidal category on a self-dual generator. Combining this with Proposition~\ref{StrictifyingEmbedding} justifies our use of the term ``free'' also in the case of Example~\ref{FreeStrings}. As a consequence, the proofs of many of the results in this section could be conducted diagrammatically.
\begin{example}\label{FreeStrings}
 We define $\csym{S}$ as the semigroup category of non-empty string diagrams, with one generating object, which we denote by $1$, denoted by a thin black string, and generating morphisms $\eta^{l}, \eta^{r}, \varepsilon^{l}, \varepsilon^{r}$ given by the following diagrams:
 \[
\eta^{l} = \eta^{l}_{1}:= \; \begin{tikzpicture}[style={line width=0.2mm, inner sep = 0pt, outer sep = 0pt},baseline=(D.base),scale=0.5]
 \node (D) at (0, 1.5) {};
 \node (Z0) at (-0.7,0) {};
 \node (Z1) at (-0.7,3) {};
 \node (K) at (0.7,2) {$\bullet$};
 \node (Y0) at (0,3) {};
 \node (Y1) at (1.4,3) {};
 \node (J) at (-0.7,0.6) {};
 \draw [-] (K) to [out=173,in=-90] (Y0);
 \draw [-] (K) to [out=7, in=-90] (Y1);
 \draw[dashed,gray] [-] (J) to [out=7,in= -97] (K);
 \draw [-] (Z0) to (Z1);
\end{tikzpicture}
\; ;\qquad \eta^{r} = \eta^{r}_{1}: = \begin{tikzpicture}[style={line width=0.2mm, inner sep = 0pt, outer sep = 0pt},baseline=(D.base),scale=0.5]
 \node (D) at (0, 1.5) {};
 \node (Z0) at (2.1,0) {};
 \node (Z1) at (2.1,3) {};
 \node (K) at (0.7,2) {$\bullet$};
 \node (Y0) at (0,3) {};
 \node (Y1) at (1.4,3) {};
 \node (J) at (2.1,0.6) {};
 \draw [-] (K) to [out=173,in=-90] (Y0);
 \draw [-] (K) to [out=7, in=-90] (Y1);
 \draw[dashed,gray] [-] (J) to [out=173,in= -83] (K);
 \draw [-] (Z0) to (Z1);
\end{tikzpicture}
\; ;\qquad \varepsilon^{l} = \varepsilon^{l}_{1}:=\; \begin{tikzpicture}[style={line width=0.2mm, inner sep = 0pt, outer sep = 0pt},baseline=(D.base),scale=0.5]
 \node (D) at (0, 1.5) {};
	\node (X0) at (0, 0) {};
	\node (X1) at (1.4, 0) {};
	\node (K) at (0.7,1) {$\bullet$};
	\node (Z0) at (-0.7,0) {};
	\node (Z1) at (-0.7,3) {};
	\node (J) at (-0.7,2.4) {};
	\draw [-] (Z0) to (Z1);
	\draw [-] (X0) to [out=90, in=187] (K);
	\draw [-] (X1) to [out=90, in=-7] (K);
	\draw[dashed,gray] [-] (K) to [out=97, in=-7] (J);
\end{tikzpicture}
\; ;\qquad \varepsilon^{r}=\varepsilon^{r}_{1}:= \begin{tikzpicture}[style={line width=0.2mm, inner sep = 0pt, outer sep = 0pt},baseline=(D.base),scale=0.5]
 \node (D) at (0, 1.5) {};
	\node (X0) at (0, 0) {};
	\node (X1) at (1.4, 0) {};
	\node (K) at (0.7,1) {$\bullet$};
	\node (Z0) at (2.1,0) {};
	\node (Z1) at (2.1,3) {};
	\node (J) at (2.1,2.4) {};
	\draw [-] (Z0) to (Z1);
	\draw [-] (X0) to [out=90, in=187] (K);
	\draw [-] (X1) to [out=90, in=-7] (K);
	\draw[dashed,gray] [-] (K) to [out=83, in=187] (J);
\end{tikzpicture}
 \]
 The category $\csym{S}$ is obtained by imposing the following relations:
 \[
  \begin{tikzpicture}[style={line width=0.2mm, inner sep = 0pt, outer sep = 0pt},baseline=(D.base),scale=0.5]
   \node (D) at (0, 1.5) {};
	\node (X0) at (0, 0) {};
	\node (X1) at (1.4, 0) {};
	\node (K) at (0.7,1) {$\bullet$};
	\node (Z0) at (2.1,0) {};
	\node (Z1) at (2.1,3) {};
	\node (J) at (2.1,2.4) {};
	\node (A0) at (-0.7,0) {};
	\node (A1) at (-0.7,3) {};
	\draw [-] (Z0) to (Z1);
	\draw [-] (X0) to [out=90, in=187] (K);
	\draw [-] (X1) to [out=90, in=-7] (K);
	\draw[dashed,gray] [-] (K) to [out=83, in=187] (J);
	\draw [-] (A0) to (A1);
\end{tikzpicture}
=
  \begin{tikzpicture}[style={line width=0.2mm, inner sep = 0pt, outer sep = 0pt},baseline=(D.base), scale=0.5]
   \node (D) at (0, 1.5) {};
	\node (X0) at (0, 0) {};
	\node (X1) at (1.4, 0) {};
	\node (K) at (0.7,1) {$\bullet$};
	\node (Z0) at (2.1,0) {};
	\node (Z1) at (2.1,3) {};
	\node (J) at (-0.7,2.4) {};
	\node (A0) at (-0.7,0) {};
	\node (A1) at (-0.7,3) {};
	\draw [-] (Z0) to (Z1);
	\draw [-] (X0) to [out=90, in=187] (K);
	\draw [-] (X1) to [out=90, in=-7] (K);
	\draw[dashed,gray] [-] (K) to [out=97, in=-7] (J);
	\draw [-] (A0) to (A1);
\end{tikzpicture}
\; ; \qquad
\begin{tikzpicture}[style={line width=0.2mm, inner sep = 0pt, outer sep = 0pt},baseline=(D.base),scale=0.5]
 \node (D) at (0, 1.5) {};
 \node (Z0) at (-0.7,0) {};
 \node (Z1) at (-0.7,3) {};
 \node (K) at (0.7,2) {$\bullet$};
 \node (Y0) at (0,3) {};
 \node (Y1) at (1.4,3) {};
 \node (A0) at (2.1,0) {};
 \node (A1) at (2.1,3) {};
 \node (J) at (-0.7,0.6) {};
 \draw [-] (K) to [out=173,in=-90] (Y0);
 \draw [-] (K) to [out=7, in=-90] (Y1);
 \draw[dashed,gray] [-] (J) to [out=7,in= -97] (K);
 \draw [-] (Z0) to (Z1);
 \draw [-] (A0) to (A1);
\end{tikzpicture}
=
\begin{tikzpicture}[style={line width=0.2mm, inner sep = 0pt, outer sep = 0pt},baseline=(D.base), scale=0.5]
 \node (D) at (0, 1.5) {};
 \node (Z0) at (-0.7,0) {};
 \node (Z1) at (-0.7,3) {};
 \node (K) at (0.7,2) {$\bullet$};
 \node (Y0) at (0,3) {};
 \node (Y1) at (1.4,3) {};
 \node (J) at (2.1,0.6) {};
 \node (A0) at (2.1,0) {};
 \node (A1) at (2.1,3) {};
 \draw [-] (K) to [out=173,in=-90] (Y0);
 \draw [-] (K) to [out=7, in=-90] (Y1);
 \draw[dashed,gray] [-] (J) to [out=173,in= -83] (K);
 \draw [-] (Z0) to (Z1);
 \draw [-] (A0) to (A1);
\end{tikzpicture}
 \]
and
\[
 \begin{tikzpicture}[style={line width=0.2mm, inner sep = 0pt, outer sep = 0pt},baseline=(D.base),scale = 0.5]
 \node (D) at (0, 3) {};
 \node (Z0) at (2.8,0) {};
 \node (Z1) at (2.8,3) {};
 \node (K) at (0.7,2) {$\bullet$};
 \node (Y0) at (0,3) {};
 \node (Y1) at (1.4,3) {};
 \node (J) at (2.8,0.6) {};
 \draw [-] (K) to [out=173,in=-90] (Y0);
 \draw [-] (K) to [out=7, in=-90] (Y1);
 \draw[dashed,gray] [-] (J) to [out=173,in= -83] (K);
 \draw [-] (Z0) to (Z1);
 	\node (M0) at (1.4, 3) {};
	\node (M1) at (2.8, 3) {};
	\node (C) at (2.1,4) {$\bullet$};
	\node (N0) at (0,3) {};
	\node (N1) at (0,6) {};
	\node (R) at (0,5.4) {};
	\draw [-] (N0) to (N1);
	\draw [-] (M0) to [out=90, in=187] (C);
	\draw [-] (M1) to [out=90, in=-7] (C);
	\draw[dashed,gray] [-] (C) to [out=97, in=-7] (R);
	\node (E) at (3.2,3) {$=$};
	\node (L0) at (3.6,0) {};
	\node (L1) at (3.6,6) {};
	\draw [-] (L0) to (L1);
\end{tikzpicture}
\; ; \qquad
\begin{tikzpicture}[style={line width=0.2mm, inner sep = 0pt, outer sep = 0pt},baseline=(D.base),scale = 0.5]

 \node (D) at (0, 3) {};
 \node (Z0) at (0,0) {};
 \node (Z1) at (0,3) {};
 \node (K) at (2.1,2) {$\bullet$};
 \node (Y0) at (1.4,3) {};
 \node (Y1) at (2.8,3) {};
 \node (J) at (0,0.6) {};
 \node (M0) at (0, 3) {};
 \node (M1) at (1.4, 3) {};
 \node (C) at (0.7,4) {$\bullet$};
 \node (N0) at (2.8,3) {};
 \node (N1) at (2.8,6) {};
 \node (R) at (2.8,5.4) {};
 \draw [-] (K) to [out=173,in=-90] (Y0);
 \draw [-] (K) to [out=7, in=-90] (Y1);
 \draw[dashed,gray] [-] (J) to [out=7,in= -97] (K);
 \draw [-] (Z0) to (Z1);
 \draw [-] (N0) to (N1);
 \draw [-] (M0) to [out=90, in=187] (C);
 \draw [-] (M1) to [out=90, in=-7] (C);
 \draw[dashed,gray] [-] (C) to [out=83, in=187] (R);
 \node (E) at (3.2,3) {$=$};
 \node (L0) at (3.6,0) {};
 \node (L1) at (3.6,6) {};
 \draw [-] (L0) to (L1);
\end{tikzpicture}
\]
One easily finds that $(1,1,\eta^{l},\eta^{r}, \varepsilon^{l},\varepsilon^{r})$ is an adjunction, and so $1$ is the left and right dual of itself.
Denoting by $n$ the object $1^{\otimes n}$ represented by $n$ strings, using Proposition~\ref{AdjunctionsCompose}, we find that generally $n$ is left and right dual to itself, establishing rigidity.
\end{example}

\begin{proposition}\label{Dualize}
 Let $\csym{S}$ be a semigroup category whose objects admit right duals. A choice of an adjunction $(\mathrm{F},\mathrm{F}^{\blackdiamond}, \eta^{l},\eta^{r}, \varepsilon^{l},\varepsilon^{r})$ for every object $\mathrm{F}$ extends to a full, faithful semigroup functor $(-)^{\blackdiamond}: \csym{S}^{\on{rev,opp}} \rightarrow \csym{S}$, sending a morphism $\mathrm{f}: \mathrm{F} \rightarrow \mathrm{G}$ to $\varepsilon^{l,\mathrm{G}}_{\mathrm{F}^{\blackdiamond}} \circ \mathrm{F}^{\blackdiamond} \boxtimes \mathrm{f} \boxtimes \mathrm{G}^{\blackdiamond} \circ \eta^{r,\mathrm{F}}_{\mathrm{G}^{\blackdiamond}}$.
 If the objects of $\csym{S}$ instead admit left duals, then a similar choice of adjunctions giving left adjoints yields a full, faithful semigroup functor ${}^{\blackdiamond}(-): \csym{S}^{\on{rev,opp}} \rightarrow \csym{S}$.
 If $\csym{S}$ is rigid, then $(-)^{\blackdiamond}$ and ${}^{\blackdiamond}(-)$ are mutually quasi-inverse equivalences.
\end{proposition}

\begin{proof}
 The claim of the lemma holds for monoidal categories, see \cite[Exercise~2.10.7]{EGNO}. The full subcategory $\csym{S}^{\mathbf{1}}$ of $\on{Mod-}\!\csym{S}(\csym{S},\csym{S})$ consisting of image of $\yo_{\mathbf{B}(\ccf{S})}$ and $\on{Id}_{\ccf{S}}$ is a monoidal category whose objects admit right duals, since $\on{Id}_{\ccf{S}}$ is left and right dual of itself. Thus, there is a full and faithful monoidal endofunctor $(-)^{\blackdiamond, \cccsym{S}^{\mathbf{1}}}$ of $\csym{S}^{\mathbf{1}}$, which sends the image of $\yo_{\mathbf{B}(\ccf{S})}$ to itself. Thus, the restriction of $(-)^{\blackdiamond, \cccsym{S}^{\mathbf{1}}}$ to $\csym{S}$ gives the semigroup functor in question.

 The case of left duals and the rigid case follow similarly, see \cite[Remark~2.10.3]{EGNO}.
\end{proof}

\begin{corollary}\label{propertyNotStructure}
 Right duals in a semigroup category are unique up to isomorphism: if $\mathrm{F}^{\blackdiamond,1}, \mathrm{F}^{\blackdiamond,2}$ are right dual to $\mathrm{F} \in \csym{S}$, then $\mathrm{F}^{\blackdiamond,1} \simeq \mathrm{F}^{\blackdiamond, 2}$.
 Similarly, left duals in $\csym{S}$ are unique up to isomorphism.
\end{corollary}

\begin{proof}
 Let $\csym{S}_{\euler{L}}$ be the full semigroup subcategory of $\csym{S}$ consisting of objects admitting a right dual, and let $\csym{S}_{\euler{R}}$ be the full semigroup subcategory of $\csym{S}$ consisting of objects admitting a left dual.
 As a consequence of Corollary~\ref{Dualize}, taking left duals is a full and faithful semigroup functor ${}^{\blackdiamond}(-): \csym{S}_{\euler{R}} \rightarrow \csym{S}_{\euler{L}}$.
 In particular, it is conservative, i.e. it reflects isomorphisms, and so, since our assumption implies ${}^{\blackdiamond}\big(\mathrm{F}^{\blackdiamond,1}\big) \simeq \mathrm{F} \simeq {}^{\blackdiamond}\big(\mathrm{F}^{\blackdiamond,2}\big)$, we have $\mathrm{F}^{\blackdiamond,1} \simeq \mathrm{F}^{\blackdiamond,2}$.
\end{proof}

\begin{corollary}\label{HomIsos}
 If $\mathrm{F}^{\blackdiamond}$ is right dual to $\mathrm{F} \in \csym{S}$, then there are isomorphisms
 \[
  \Hom{\mathrm{F\boxtimes H,K}}_{\ccf{S}} \simeq \Hom{\mathrm{H,F^{\blackdiamond} \boxtimes K}}_{\ccf{S}} \text{ and } \Hom{\mathrm{H,K\boxtimes F}}_{\ccf{S}} \simeq \Hom{\mathrm{H\boxtimes F^{\blackdiamond}, K}}_{\ccf{S}} \text{ natural in } \mathrm{H,K}.
 \]
Similarly, if ${}^{\blackdiamond}\mathrm{F}$ is a left dual, then there are isomorphisms
 \[
  \Hom{\mathrm{{}^{\blackdiamond}F\boxtimes H,K}}_{\ccf{S}} \simeq \Hom{\mathrm{H,F\boxtimes K}}_{\ccf{S}} \text{ and } \Hom{\mathrm{H,K\boxtimes {}^{\blackdiamond}F}}_{\ccf{S}} \simeq \Hom{\mathrm{H\boxtimes F, K}}_{\ccf{S}} \text{ natural in } \mathrm{H,K}.
 \]
\end{corollary}

\begin{proof}
 In view of Proposition~\ref{StrictifyingEmbedding}, this is a consequence of \cite[Proposition~2.10.8]{EGNO}.
\end{proof}

 By removing units and unitality properties from \cite{Day}, \cite{IK}, we find the following:
\begin{theorem}[\cite{Day},\cite{IK}]\label{UniversalPropertyDay}
 Given a semigroup category $\csym{S}$, its presheaf category $[\csym{S}^{\on{opp}},\mathbf{Vec}_{\Bbbk}]$ is a semigroup category under Day convolution. Given presheaves $\euler{P},\euler{Q} \in [\csym{S}^{\on{opp}},\mathbf{Vec}_{\Bbbk}]$, we set
 \[
  (\euler{P} \oast \euler{Q})(-) = \int^{\mathrm{H,K}\in\ccf{S}} \csym{S}(-,\mathrm{H \boxtimes K}) \kotimes \euler{P}(\mathrm{H}) \kotimes \euler{Q}(\mathrm{K}).
 \]
 The functor $- \oast -$ is cocontinuous in each variable, the Yoneda embedding $\yo_{\ccf{S}}: \csym{S}\hookrightarrow [\csym{S}^{\on{opp}},\mathbf{Vec}_{\Bbbk}]$ is a semigroup functor, and $[\csym{S}^{\on{opp}}, \mathbf{Vec}_{\Bbbk}]$ has the following universal property: for any cocomplete semigroup category $\csym{Z}$, precomposing with the Yoneda embedding $\yo_{\ccf{S}}$ gives an equivalence
 \[
 \mathbf{SemCat}_{\Bbbk}(\csym{S},\csym{Z}) \simeq \mathbf{SemCat}_{\Bbbk}^{\text{Scoc}}([\csym{S}^{\on{opp}},\mathbf{Vec}_{\Bbbk}], \csym{Z}),
 \]
 where $\mathbf{SemCat}_{\Bbbk}$ is the category of $\Bbbk$-linear semigroup functors, and $\mathbf{SemCat}_{\Bbbk}^{\text{Scoc}}$ is the category of $\Bbbk$-linear semigroup functors cocontinuous in each variable.
\end{theorem}

\begin{theorem}~\label{PromonoidalUnital}
 If $\csym{S}$ is a rigid semigroup category, then $[\csym{S}^{\on{opp}},\mathbf{Vec}_{\Bbbk}]$ endowed with Day convolution is a monoidal category. Its unit object is the presheaf $\Hom{\yo_{\mathbf{B}(\ccf{S})}(-),\on{Id}}_{\on{Mod-}\!\ccf{S}(\ccf{S},\ccf{S})}$ given by
 \[
  \mathrm{F} \mapsto \Hom{\mathrm{F} \boxtimes -, \on{Id}}_{\on{Mod-}\!\ccf{S}(\ccf{S},\ccf{S})}
 \]
\end{theorem}

\begin{proof}
 We show that $\Hom{\mathrm{F} \boxtimes -, \on{Id}}_{\on{Mod-}\!\ccf{S}(\ccf{S},\ccf{S})} \oast -$ and $- \oast \Hom{\mathrm{F} \boxtimes -, \on{Id}}_{\on{Mod-}\!\ccf{S}(\ccf{S},\ccf{S})}$ are naturally isomorphic to the identity functor; the result will then follow from Proposition~\ref{EGNOUnit}.

 Indeed, for any $\euler{P} \in [\csym{S}^{\on{opp}},\mathbf{Vec}]$, we have
 \[
 \resizebox{.99\hsize}{!}{$
 \begin{aligned}
  &\int^{\mathrm{H,K}} \Hom{\mathrm{F}, \mathrm{H \boxtimes K}}_{\ccf{S}} \kotimes \Hom{\mathrm{H} \boxtimes -, \on{Id}}_{\on{Mod-}\!\ccf{S}(\ccf{S},\ccf{S})} \kotimes \euler{P}(\mathrm{K}) {\overset{(1)}\simeq} \int^{\mathrm{K}}\int^{\mathrm{H}} \Hom{\mathrm{F \boxtimes K^{\blackdiamond}, H}}_{\ccf{S}} \kotimes \Hom{\mathrm{H} \boxtimes -, \on{Id}}_{\on{Mod-}\!\ccf{S}(\ccf{S},\ccf{S})} \kotimes \euler{P}(\mathrm{K}) \\
  &{\overset{(2)}\simeq} \int^{\mathrm{K}} \Hom{\mathrm{F \boxtimes K^{\blackdiamond}\boxtimes -},\on{Id}}_{\on{Mod-}\!\ccf{S}(\ccf{S},\ccf{S})} \kotimes \euler{P}(\mathrm{K}) {\overset{(3)}\simeq} \int^{\mathrm{K}} \Hom{\mathrm{(F \boxtimes -) \circ (K^{\blackdiamond}\boxtimes -)},\on{Id}}_{\on{Mod-}\!\ccf{S}(\ccf{S},\ccf{S})} \kotimes \euler{P}(\mathrm{K}) \\
  &  {\overset{(4)}\simeq} \int^{\mathrm{K}} \Hom{\mathrm{F \boxtimes -, K\boxtimes -}}_{\on{Mod-}\!\ccf{S}(\ccf{S},\ccf{S})} \kotimes \euler{P}(\mathrm{K}) {\overset{(5)}\simeq} \int^{\mathrm{K}} \Hom{\mathrm{F, K}}_{\ccf{S}} \kotimes \euler{P}(\mathrm{K}) {\overset{(6)}\simeq} \euler{P}(\mathrm{F}).
 \end{aligned}$}
 \]
 The isomorphisms are justified as follows
 \begin{multicols}{3}
 \begin{enumerate}
  \item Fubini's theorem for coends, together with Corollary~\ref{HomIsos};
  \item Yoneda lemma;
  \item definition of composition in ${\on{Mod-}\!\ccf{S}(\ccf{S},\ccf{S})}$;
  \item \cite[Proposition~2.10.8]{EGNO};
  \item Proposition~\ref{StrictifyingEmbedding};
  \item Yoneda lemma.
 \end{enumerate}
\end{multicols}
 Naturality in $\mathrm{F}$ and $\mathrm{H}$ follows from the naturality claim in Corollary~\ref{HomIsos}; naturality in $\mathrm{K}$ follows from Corollary~\ref{Dualize}; finally, naturality in $\euler{P}$ follows from the naturality of Yoneda lemma.

 This shows $\Hom{\mathrm{F} \boxtimes -, \on{Id}}_{\on{Mod-}\!\ccf{S}(\ccf{S},\ccf{S})} \oast - \simeq \on{Id}_{[\ccf{S}^{\on{opp}},\mathbf{Vec}_{\Bbbk}]}$. An analogous calculation, crucially using left duals in $\ccf{S}$, shows $-\oast \Hom{\mathrm{F} \boxtimes -, \on{Id}}_{\on{Mod-}\!\ccf{S}(\ccf{S},\ccf{S})} \simeq \on{Id}_{[\ccf{S}^{\on{opp}},\mathbf{Vec}_{\Bbbk}]}$.
\end{proof}

\begin{corollary}
 An adjunction $(\mathrm{F},\mathrm{F}^{\blackdiamond}, \eta^{l},\eta^{r}, \varepsilon^{l},\varepsilon^{r})$ in $\csym{S}$ gives an adjunction $(\Hom{-,\mathrm{F}},\Hom{-,\mathrm{F}^{\blackdiamond}},\widehat{\eta}^{l},\widehat{\varepsilon}^{l})$ in $[\csym{S}^{\on{opp}},\mathbf{Vec}_{\Bbbk}]$.
 The morphisms $\widehat{\eta}^{l}$ is defined as the composite
\[\begin{tikzcd}
	{\Hom{\yo_{\mathbf{B}(\ccf{S})}(-),\on{Id}}_{\on{Mod-}\!\ccf{S}(\ccf{S},\ccf{S})}} &&& {\Hom{\yo_{\mathbf{B}(\ccf{S})}(-),(\mathrm{F}^{\blackdiamond}\boxtimes \mathrm{F})\boxtimes -}_{\on{Mod-}\!\ccf{S}(\ccf{S},\ccf{S})}} & {\Hom{-,\mathrm{F}^{\blackdiamond}\boxtimes \mathrm{F}}}
	\arrow["{\Hom{\yo_{\mathbf{B}(\cccsym{S})}(-),\eta^{l}}_{\on{Mod-}\!\cccsym{S}(\cccsym{S},\cccsym{S})}}", shift left=1, from=1-1, to=1-4]
	\arrow["\simeq", from=1-4, to=1-5]
\end{tikzcd}\]
and the morphism $\widehat{\varepsilon}^{l}$ is defined analogously.
\end{corollary}

\begin{proof}
 By definition, $(\mathrm{F} \boxtimes -, \mathrm{F}^{\blackdiamond} \boxtimes -, \eta^{l}, \varepsilon^{l})$ is an adjunction in $\on{Mod-}\!\csym{S}(\csym{S},\csym{S})$. Since the Yoneda embedding $\on{Mod-}\!\csym{S}(\csym{S},\csym{S})$ is strong monoidal, we obtain an adjunction $(\Hom{-,\mathrm{F} \boxtimes -}, \Hom{-,\mathrm{F}^{\blackdiamond} \boxtimes -},\Hom{-, \eta^{l}}, \Hom{-,\varepsilon^{l})})$.
 Restricting along $\yo_{\mathbf{B}(\ccf{S})}$ in the left coordinate, the zigzag equations are clearly still satisfied, and indeed prove that $\left(\Hom{\yo_{\mathbf{B}(\ccf{S})}(-),\mathrm{F} \boxtimes -}, \Hom{\yo_{\mathbf{B}(\ccf{S})}(-),\mathrm{F}^{\blackdiamond} \boxtimes -},\Hom{\yo_{\mathbf{B}(\ccf{S})}(-), \eta^{l}}, \Hom{\yo_{\mathbf{B}(\ccf{S})}(-),\varepsilon^{l})}\right)$ is an adjunction, since $\Hom{-,\on{Id}}$ is the unit object for $[\on{Mod-}\!\csym{S}(\csym{S},\csym{S})^{\on{opp}},\mathbf{Vec}_{\Bbbk}]$ and, by Theorem~\ref{PromonoidalUnital}, $\Hom{\yo_{\mathbf{B}(\ccf{S})}(-),\on{Id}}$ is the unit object for $[\csym{S}^{\on{opp}},\mathbf{Vec}_{\Bbbk}]$.
\end{proof}

\subsection{Independence of Ansatz}\label{s21}

We will now show that the apparent choice of the presheaf
 \begin{equation}\label{Comparison}
  \mathrm{F} \mapsto \Hom{\mathrm{F} \boxtimes -, \on{Id}}_{\on{Mod-}\!\ccf{S}(\ccf{S},\ccf{S})}
 \text{ over the presheaf }
  \mathrm{F} \mapsto \Hom{\mathrm{F} \boxtimes -, \on{Id}}_{\ccf{S}\!\on{-Mod}(\ccf{S},\ccf{S})}
 \end{equation}
in Theorem~\ref{PromonoidalUnital} is inessential. In \cite[Section~8]{Hou}, it is remarked that, under the axiomatization of rigid semigroup categories provided therein, there is no natural isomorphism between these two presheaves. However, the axiomatization in \cite{Hou} does not have a ``left-right compatibility axiom'' analogous to Axiom~\eqref{Involution}. In our axiomatization, the two presheaves are isomorphic, since both give the unit object for the monoidal structure on $[\csym{S}^{\on{opp}},\mathbf{Vec}_{\Bbbk}]$ - this follows from Theorem~\ref{PromonoidalUnital}. However, in order to clarify that the asymmetry issues identified in \cite{Hou} do not apply to our case, we provide an explicit isomorphism below.

To that end, we need a couple of lemmata.

\begin{lemma}\label{EasySwirl}
 Let $\csym{C}$ be a strict monoidal category with unit object $\mathbb{1}$. Let $\mathrm{W} \in \csym{C}$ admitting a right dual $\mathrm{W}^{\blackdiamond}$ and a double right dual $\mathrm{W}^{\blackdiamond\blackdiamond}$. For any $\alpha \in \Hom{\mathrm{W},\mathbb{1}}$, we have
 \[
  \alpha = \varepsilon^{\mathrm{W}} \circ (\mathrm{W}\varepsilon^{\mathrm{W}^{\blackdiamond}}\mathrm{W}^{\blackdiamond}) \circ (\mathrm{W}\mathrm{W}^{\blackdiamond}\alpha\mathrm{W}^{\blackdiamond\blackdiamond}\mathrm{W}^{\blackdiamond}) \circ (\mathrm{WW^{\blackdiamond}W\eta^{W^{\blackdiamond}}}) \circ (\mathrm{W\eta^{W}}).
 \]
 Similarly,
 \[
  \alpha = \varepsilon^{{}^{\blackdiamond}\!\mathrm{W}}\circ ({}^{\blackdiamond}\!\mathrm{W}\varepsilon^{\!{}^{\blackdiamond\blackdiamond}\!\mathrm{W}}\,\mathrm{W}) \circ ({}^{\blackdiamond}\!\mathrm{W}\,{}^{\blackdiamond\blackdiamond}\!\mathrm{W}\alpha \, {}^{\blackdiamond}\!\mathrm{W}\,\mathrm{W})\circ (\eta^{\!{}^{\blackdiamond\blackdiamond}\!\mathrm{W}}\,\mathrm{W}\, {}^{\blackdiamond}\!\mathrm{W}\, \mathrm{W}) \circ (\eta^{\!{}^{\blackdiamond}\!\mathrm{W}}\,\mathrm{W})
 \]
 The presence of additional double duals in the latter expression is due to our convention for superscripts of units and counits ``preferring'' right adjoints.
\end{lemma}

\begin{proof}
The first claim can be proved diagrammatically as follows:
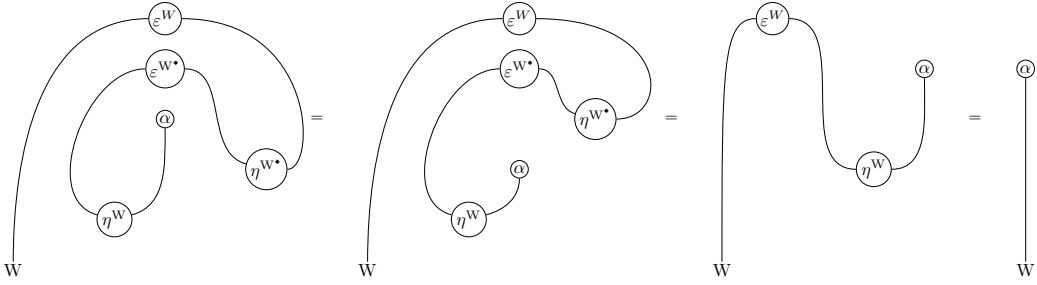
\[
 \resizebox{.8\hsize}{!}{
\begin{tikzpicture}[style={line width=0.2mm, inner sep = 0.8pt}]
	\node (X0) at (0, 0) {$\mathrm{W}$};
	\node [circle, draw] (X1) at (2, 1) {$\eta^{\mathrm{W}}$};
	\node [circle, draw,inner sep=0.8pt] (X2) at (5,2) {$\eta^{\mathrm{W}^{\blackdiamond}}$};
	\node [circle, draw,inner sep=0.8pt] (X3) at (3, 3) {$\alpha$};
	\node [circle, draw,inner sep=0.8pt] (X4) at (3, 4) {$\varepsilon^{\mathrm{W}^{\blackdiamond}}$};
	\node [circle, draw] (X5) at (3, 5) {$\varepsilon^{W}$};
	\draw [-] (X0) to [out=90, in=180] (X5);
	\draw [-] (X1) to [out=15, in=-90] (X3);
	\draw [-] (X1) to [out=165, in=180] (X4);
	\draw [-] (X2) to [out=165, in=0] (X4);
	\draw [-] (X2) .. controls (6,2) and (6,5) .. (X5);
    \node (E0) at (6,3) {$=$};
	\node (Y0) at (7, 0) {$\mathrm{W}$};
	\node [circle, draw] (Y1) at (9, 1) {$\eta^{\mathrm{W}}$};
	\node [circle, draw,inner sep=0.8pt] (Y2) at (10, 2) {$\alpha$};
	\node [circle, draw,inner sep=0.8pt] (Y3) at (11.5,3) {$\eta^{\mathrm{W}^{\blackdiamond}}$};
	\node [circle, draw,inner sep=0.8pt] (Y4) at (10, 4) {$\varepsilon^{\mathrm{W}^{\blackdiamond}}$};
	\node [circle, draw] (Y5) at (10, 5) {$\varepsilon^{W}$};
	\draw [-] (Y0) to [out=90, in=180] (Y5);
	\draw [-] (Y1) to [out=15, in=-90] (Y2);
	\draw [-] (Y1) to [out=165, in=180] (Y4);
	\draw [-] (Y3) to [out=165, in=0] (Y4);
	\draw [-] (Y3) .. controls (13,3) and (13,5) .. (Y5);
    \node (E1) at (13,3) {$=$};
	\node (Z0) at (14, 0) {$\mathrm{W}$};
	\node [circle, draw] (Z1) at (17, 2) {$\eta^{\mathrm{W}}$};
	\node [circle, draw,inner sep=0.8pt] (Z2) at (18, 4) {$\alpha$};
	\node [circle, draw] (Z3) at (15, 5) {$\varepsilon^{W}$};
	\draw [-] (Z0) .. controls (14,4) and (14,5) .. (Z3);
	\draw [-] (Z1) to [out=180, in = 0] (Z3);
	\draw [-] (Z1) to [out=0,in=-90] (Z2);
    \node (E2) at (19,3) {$=$};
    \node (W0) at (20,0) {$\mathrm{W}$};
    \node [circle,draw] (W1) at (20,4) {$\alpha$};
    \draw [-] (W0) to (W1);
\end{tikzpicture}
}
\]
The diagrammatic proof of the latter claim is the mirror image of the above.
\end{proof}

Applying Lemma~\ref{EasySwirl} in the case $\csym{C} = \on{Mod-}\csym{S}(\csym{S},\csym{S})$, we obtain the following:

\begin{corollary}\label{Simplify}
 Let $\csym{S}$ be a rigid semigroup category and let $\alpha \in \Hom{\mathrm{F} \boxtimes -, \on{Id}}_{\on{Mod-}\!\ccf{S}(\ccf{S},\ccf{S})}$. Writing $\overline{\mathrm{F}}$ for the right $\csym{S}$-module functor $\mathrm{F} \boxtimes - \in \on{Mod-}\!\csym{S}(\csym{S},\csym{S})$, we have:
 \[
  \alpha = \varepsilon^{\mathrm{F},r} \circ (\overline{\mathrm{F}}\varepsilon^{\mathrm{F}^{\blackdiamond},r}\overline{\mathrm{F}^{\blackdiamond}}) \circ (\overline{\mathrm{F}}\overline{\mathrm{F}^{\blackdiamond}}\alpha\overline{\mathrm{F}^{\blackdiamond\blackdiamond}}\,\overline{\mathrm{F}^{\blackdiamond}}) \circ (\overline{\mathrm{F}}\,\overline{\mathrm{F}^{\blackdiamond}}\,\overline{\mathrm{F}}\eta^{\mathrm{F}^{\blackdiamond},r}) \circ (\overline{\mathrm{F}}\eta^{\mathrm{F},r}).
 \]

 This equality allows us to rewrite $\alpha_{\mathrm{H}}$, for $\mathrm{H} \in \csym{S}$, in terms of graphical calculus for $\csym{S}$ (rather than the graphical calculus for $\on{Mod-}\csym{S}(\csym{S},\csym{S})$, which Lemma~\ref{EasySwirl} refers to in this application thereof) as follows:
 \[
 \scalebox{.7}
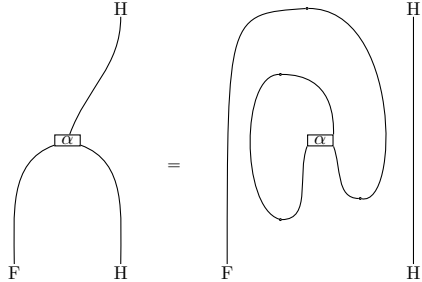
{
  \begin{tikzpicture}[style={line width=0.2mm, inner sep = 0.8pt}]
   \node (X0) at (0,0) {$\mathrm{F}$};
   \node (X1) at (2,0) {$\mathrm{H}$};
   \node [rectangle,draw] (X2) at (1,2.5) {$\;\alpha \;$};
   \node (X3) at (2,5) {$\mathrm{H}$};
   \draw [-] (X0) to [out=90,in=-160] (X2) ;
   \draw [-] (X1) to [out=90,in=-20] (X2);
   \draw [-] (X2) to [out=70,in=-90] (X3);
   \node (E0) at (3,2) {$=$};
   \node (Y0) at (4,0) {$\mathrm{F}$};
   \node (Y1) [circle,draw,inner sep=0pt] at (5,1) {$\hspace{0mm}$};
   \node (Y2) [circle,draw,inner sep=0pt] at (6.5,1.4) {$\hspace{0mm}$};
   \node [rectangle,draw] (Y3) at (5.75,2.5) {$\; \alpha \;$};
   \node (Y4) [circle,draw,inner sep=0pt] at (5,3.75) {$\hspace{0mm}$};
   \node (Y5) [circle,draw,inner sep=0pt] at (5.5,5) {$\hspace{0mm}$};
   \node (Y6) at (7.5,0) {$\mathrm{H}$};
   \node (Y7) at (7.5,5) {$\mathrm{H}$};
   \draw [-] (Y0) .. controls (4,4.9) .. (Y5);
   \draw [-] (Y1) to [out=0,in=-110] (Y3.south west);
   \draw [-] (Y2) to [out=180,in=-70] (Y3.south east);
   \draw [-] (Y1) .. controls (4.25,1.25) and (4.25,3.75) .. (Y4);
   \draw [-] (Y3.north east) .. controls (6,2.9) and (6,3.75) .. (Y4);
   \draw [-] (Y2) .. controls (7.25,1.25) and (7.25,5) .. (Y5);
   \draw [-] (Y6) to (Y7);
  \end{tikzpicture}
 }
 \]
 On the level of components, we find
 \[
  \alpha_{H} = \varepsilon^{\mathrm{F},r}_{H} \circ (\mathrm{F}\varepsilon^{\mathrm{F}^{\blackdiamond},r}_{\mathrm{F}^{\blackdiamond}}\mathrm{H}) \circ \mathrm{FF^{\blackdiamond}}\alpha_{\mathrm{F}^{\blackdiamond\blackdiamond}}\mathrm{F}^{\blackdiamond}\mathrm{H} \circ (\mathrm{FF^{\blackdiamond}F}\eta^{\mathrm{F}^{\blackdiamond},r}_{\mathrm{H}}) \circ \mathrm{F}\eta^{\mathrm{F},r}_{\mathrm{H}}
 \]

 Similarly, for any $\alpha \in \Hom{- \boxtimes \mathrm{F}, \on{Id}}_{\ccf{S}\!\on{-Mod}(\ccf{S},\ccf{S})}$, writing $\underline{\mathrm{F}}$ for $-\boxtimes \mathrm{F} \in \csym{S}\!\on{-Mod}(\csym{S},\csym{S})$, we have:
   \[
    \alpha = \varepsilon^{{}^{\blackdiamond}\!\mathrm{F},l} \circ (\underline{\mathrm{F}}\varepsilon^{\!{}^{\blackdiamond\blackdiamond}\!\mathrm{F},l}\,\underline{{}^{\blackdiamond}\!\mathrm{F}}) \circ (\underline{\mathrm{F}}\,\underline{{}^{\blackdiamond}\!\mathrm{F}}\alpha \, \underline{\mathrm{F}}\,\underline{{}^{\blackdiamond\blackdiamond}\!\mathrm{F}})\circ (\underline{\mathrm{F}}\, \underline{{}^{\blackdiamond}\!\mathrm{F}}\, \underline{\mathrm{F}}\,\eta^{\!{}^{\blackdiamond\blackdiamond}\!\mathrm{F},l}) \circ (\underline{\mathrm{F}}\,\eta^{\!{}^{\blackdiamond}\!\mathrm{F},l}).
 \]
 Since the embedding $\csym{S}^{\on{rev}} \rightarrow \csym{S}\!\on{-Mod}(\csym{S},\csym{S})$ reverses the monoidal product, on components we find:
 \[
    \alpha_{\mathrm{H}} = \varepsilon^{{}^{\blackdiamond}\!\mathrm{F},l}_{\mathrm{H}}\circ (\mathrm{H}\varepsilon^{\!{}^{\blackdiamond\blackdiamond}\!\mathrm{F},l}_{{}^{\blackdiamond}\!\mathrm{F}}\,\mathrm{F}) \circ (\mathrm{H}\,{}^{\blackdiamond}\!\mathrm{F}\alpha_{{}^{\blackdiamond\blackdiamond}\!\mathrm{F}} \, {}^{\blackdiamond}\!\mathrm{F}\,\mathrm{F})\circ (\eta^{\!{}^{\blackdiamond\blackdiamond}\!\mathrm{F},l}_{\mathrm{H}}\,\mathrm{F}\, {}^{\blackdiamond}\!\mathrm{F}\, \mathrm{F}) \circ (\eta^{\!{}^{\blackdiamond}\!\mathrm{F},l}_{\mathrm{H}}\,\mathrm{F}).
 \]
\end{corollary}

\begin{theorem}\label{Ansatzes}
 The assignments
 \[
 \begin{aligned}
  \Psi: \Hom{\mathrm{F} \boxtimes -, \on{Id}}_{\on{Mod-}\!\ccf{S}(\ccf{S},\ccf{S})} &\rightarrow
  \Hom{-\boxtimes \mathrm{F}, \on{Id}}_{\ccf{S}\!\on{-Mod}(\ccf{S},\ccf{S})} \\
  \alpha
  &\mapsto
  \varepsilon^{\mathrm{F},l} \circ (\mathrm{F}\varepsilon^{\mathrm{F}^{\blackdiamond},l}\mathrm{F}^{\blackdiamond}) \circ (- \boxtimes \mathrm{FF^{\blackdiamond}}\alpha_{\mathrm{F}^{\blackdiamond\blackdiamond}}\mathrm{F}^{\blackdiamond}) \circ (\mathrm{FF^{\blackdiamond}F}\eta^{\mathrm{F}^{\blackdiamond},l}) \circ (\mathrm{F}\eta^{\mathrm{F},l})
 \end{aligned}
 \]
 and
 \[
 \begin{aligned}
  \Phi: \Hom{-\boxtimes \mathrm{F}, \on{Id}}_{\ccf{S}\!\on{-Mod}(\ccf{S},\ccf{S})}
  &\rightarrow
  \Hom{\mathrm{F} \boxtimes -, \on{Id}}_{\on{Mod-}\!\ccf{S}(\ccf{S},\ccf{S})} \\
  \beta
  &\mapsto
  \varepsilon^{{}^{\blackdiamond}\!\mathrm{F},r} \circ ({}^{\blackdiamond}\!\mathrm{F}\varepsilon^{{}^{\blackdiamond\blackdiamond}\!\mathrm{F},r}\mathrm{F}) \circ ({}^{\blackdiamond}\!\mathrm{F}\beta_{{}^{\blackdiamond \blackdiamond}\!\mathrm{F}}{}^{\blackdiamond}\!\mathrm{F}\, \mathrm{F} \boxtimes -) \circ (\eta^{{}^{\blackdiamond\blackdiamond}\!\mathrm{F},r}\mathrm{F}\, {}^{\blackdiamond}\!\mathrm{F}\, \mathrm{F}) \circ \eta^{{}^{\blackdiamond}\!\mathrm{F},r}\mathrm{F}
 \end{aligned}
 \]
 are mutually inverse isomorphisms of presheaves in $[\csym{S}^{\on{op}},\mathbf{Vec}_{\Bbbk}]$.
\end{theorem}

\begin{proof}
 We verify that $(\Phi \circ \Psi)(\alpha) = \alpha$. Verifying that $(\Psi \circ \Phi)(\beta) = \beta$ is analogous.
 We have
 \[
 \resizebox{.999\hsize}{!}{$
 \begin{aligned}
  &(\Phi \circ \Psi)(\alpha) = \\
  &\varepsilon^{{}^{\blackdiamond}\!\mathrm{F},r} \circ ({}^{\blackdiamond}\!\mathrm{F}\varepsilon^{{}^{\blackdiamond\blackdiamond}\!\mathrm{F},r}\mathrm{F}) \circ \Big({}^{\blackdiamond}\!\mathrm{F}\big( \varepsilon^{\mathrm{F},l} \circ (\mathrm{F}\varepsilon^{\mathrm{F}^{\blackdiamond},l}\mathrm{F}^{\blackdiamond}) \circ (- \boxtimes \mathrm{FF^{\blackdiamond}}\alpha_{\mathrm{F}^{\blackdiamond\blackdiamond}}\mathrm{F}^{\blackdiamond}) \circ (\mathrm{FF^{\blackdiamond}F}\eta^{\mathrm{F}^{\blackdiamond},l}) \circ (\mathrm{F}\eta^{\mathrm{F},l}) \big)_{{}^{\blackdiamond \blackdiamond}\!\mathrm{F}}{}^{\blackdiamond}\!\mathrm{F}\, \mathrm{F} \boxtimes -\Big) \circ (\eta^{{}^{\blackdiamond\blackdiamond}\!\mathrm{F},r}\mathrm{F}\, {}^{\blackdiamond}\!\mathrm{F}\, \mathrm{F}) \circ \eta^{{}^{\blackdiamond}\!\mathrm{F},r}\mathrm{F} \\
  &\varepsilon^{{}^{\blackdiamond}\!\mathrm{F},r} \circ ({}^{\blackdiamond}\!\mathrm{F}\varepsilon^{{}^{\blackdiamond\blackdiamond}\!\mathrm{F},r}\mathrm{F}) \circ \Bigg({}^{\blackdiamond}\!\mathrm{F}\, {}^{\blackdiamond\blackdiamond}\!\mathrm{F}\, \Big( \big(\varepsilon^{\mathrm{F},r} \circ (\mathrm{F}\varepsilon^{\mathrm{F}^{\blackdiamond},r}\mathrm{F}^{\blackdiamond})\big)\,{}^{\blackdiamond}\!\mathrm{F}\,\mathrm{F} \circ (\mathrm{FF^{\blackdiamond}}\alpha_{\mathrm{F}^{\blackdiamond\blackdiamond}}\mathrm{F}^{\blackdiamond}\,{}^{\blackdiamond}\!\mathrm{F}\,\mathrm{F} \boxtimes -) \circ \big((\mathrm{FF^{\blackdiamond}F}\eta^{\mathrm{F}^{\blackdiamond},r}) \circ (\mathrm{F}\eta^{\mathrm{F},r})\big){}^{\blackdiamond}\!\mathrm{F}\,\mathrm{F} \Big)\Bigg) \circ (\eta^{{}^{\blackdiamond\blackdiamond}\!\mathrm{F},r}\mathrm{F}\, {}^{\blackdiamond}\!\mathrm{F}\, \mathrm{F}) \circ \eta^{{}^{\blackdiamond}\!\mathrm{F},r}\mathrm{F} \\
 \end{aligned}
 $}
 \]
The first equality follows by definition. The second equality follows by observing that Axiom~\ref{Involution} implies
\[
\eta^{\mathrm{H},l}_{\mathrm{G}} \boxtimes - = \mathrm{G}\eta^{\mathrm{H},r}, \text{ and similarly } \varepsilon^{\mathrm{H},l}_{\mathrm{G}} \boxtimes - = \mathrm{G}\varepsilon^{\mathrm{H},r},
\]
 for all $\mathrm{G,H}$, and then regrouping terms.
 The obtained expression can be written diagrammatically as follows:
 \[
 \scalebox{.35}
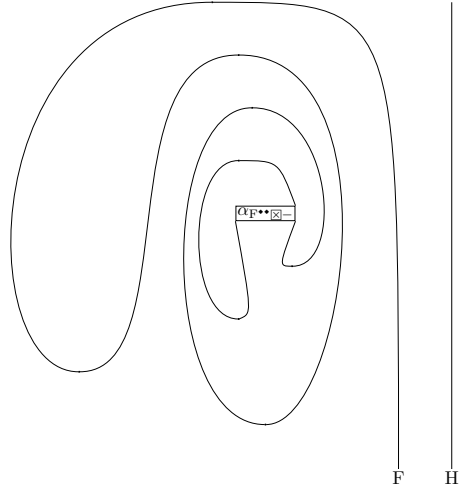
{
  \begin{tikzpicture}[style={line width=0.2mm, inner sep = 0.8pt}]
   \node [scale=2] (X) at (14,0) {$\mathrm{F}$};
   \node [scale=2] (H) at (16,0) {$\mathrm{H}$};
   \node [circle,draw,inner sep=0pt] (T1) at (9,2) {$\hspace{0pt}$};
   \node [circle,draw,inner sep=0pt] (T2) at (2,4) {$\hspace{0pt}$};
   \node [circle,draw,inner sep=0pt] (T3) at (8,6) {$\hspace{0pt}$};
   \node [circle,draw,inner sep=0pt] (T4) at (10,8) {$\hspace{0pt}$};
   \node [rectangle,draw,scale=1.9] (A) at (9,10) {$\alpha_{\mathrm{F}^{\blackdiamond\blackdiamond}\boxtimes -}$};
   \node [circle,draw,inner sep=0pt] (P1) at (8,12) {$\hspace{0pt}$};
   \node [circle,draw,inner sep=0pt] (P2) at (8.5,14) {$\hspace{0pt}$};
   \node [circle,draw,inner sep=0pt] (P3) at (8,16) {$\hspace{0pt}$};
   \node [circle,draw,inner sep=0pt] (P4) at (7,18) {$\hspace{0pt}$};
   \draw [-] (X) .. controls (14,18) .. (P4);
   \draw [-] (T1) .. controls (5,2) and (5,14) .. (P2);
   \draw [-] (T1) .. controls (12,2) and (14,16) .. (P3);
   \draw [-] (T2) .. controls (6,4) and (3,16) .. (P3);
   \draw [-] (T2) .. controls (-2,4) and (-2,18) .. (P4);
   \draw [-] (T3) .. controls (6,6) and (6,12) .. (P1);
   \draw [-] (T3) .. controls (8.5,6.2) .. (A.south west);
   \draw [-] (T4) .. controls (9.5,8) .. (A.south east);
   \draw [-] (T4) .. controls (12,8) and (11.5,14) .. (P2);
   \draw [-] (A.north east) .. controls (9.5,12) .. (P1);
   \draw [-] (H) to (16,18);
  \end{tikzpicture}
 }
 \]
  which, after straightening two zigzags, becomes
   \[
 \scalebox{.4}
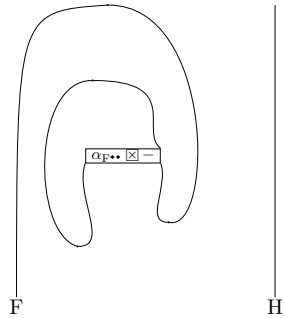
{
  \begin{tikzpicture}[style={line width=0.2mm, inner sep = 0.8pt}]
   \node [scale=2] (Y0) at (0,0) {$\mathrm{F}$};
   \node [scale=2] (H) at (8.5,0) {$\mathrm{H}$};
   \node (Y1) [circle,draw,inner sep=0pt] at (2,2) {$\hspace{0mm}$};
   \node (Y2) [circle,draw,inner sep=0pt] at (5,2.8) {$\hspace{0mm}$};
   \node [rectangle,draw,scale=1.5] (Y3) at (3.5,5) {$\; \alpha_{\mathrm{F}^{\blackdiamond\blackdiamond}} \boxtimes - \;$};
   \node (Y4) [circle,draw,inner sep=0pt] at (2.5,7.5) {$\hspace{0mm}$};
   \node (Y5) [circle,draw,inner sep=0pt] at (3,10) {$\hspace{0mm}$};
   \draw [-] (Y0) .. controls (0,9.8) .. (Y5);
   \draw [-] (Y1) to [out=0,in=-110] (Y3.south west);
   \draw [-] (Y2) to [out=180,in=-70] (Y3.south east);
   \draw [-] (Y1) .. controls (0.5,2.5) and (0.5,7.5) .. (Y4);
   \draw [-] (Y3.north east) .. controls (4,5.8) and (5.5,7.5) .. (Y4);
   \draw [-] (Y2) .. controls (6.5,2.5) and (6.5,10) .. (Y5);
   \draw [-] (H) to (8.5,10);
  \end{tikzpicture}
 }
 \]
 Again writing $\overline{\mathrm{F}}$ for $\mathrm{F} \boxtimes -$, the above diagram gives us the morphism
 $
  \varepsilon^{\mathrm{F},r} \circ (\overline{\mathrm{F}}\varepsilon^{\mathrm{F}^{\blackdiamond},r}\mathrm{F}^{\blackdiamond}) \circ (\overline{\mathrm{F}}\overline{\mathrm{F}^{\blackdiamond}}\overline{\alpha_{\mathrm{F}^{\blackdiamond\blackdiamond}}}\overline{\mathrm{F}^{\blackdiamond}}) \circ \overline{\mathrm{F}}\overline{\mathrm{F}^{\blackdiamond}}\overline{\mathrm{F}}\eta^{\mathrm{F}^{\blackdiamond},r}\circ \overline{\mathrm{F}}\eta^{\mathrm{F},r}
 $
 which, by Corollary~\ref{Simplify}, is equal to $\alpha$.
\end{proof}

\section{The bar complex approach}\label{s3}

\subsection{Comonad cohomology} \label{s31}

We give a very brief summary of the theories of relative homological algebra and comonad cohomology. More extensive accounts can be found in \cite{FGS}, \cite{GHS} \cite{MacL}.

Given a preadditive category $\mathcal{A}$ and a comonad $\euler{C}$ on $\mathcal{A}$, one obtains a cohomology theory by declaring a sequence $X \xrightarrow{f} Y \xrightarrow{g} Z$ to be {\it $\euler{C}$-exact} if the sequence
\[
 \Hom{\euler{C}(W), X} \xrightarrow{\Hom{\euler{C}(W),f}} \Hom{\euler{C}(W), Y} \xrightarrow{\Hom{\euler{C}(W),g}} \Hom{\euler{C}(W), Z}
\]
is an exact sequence in $\mathbf{Ab}$, for all $W \in \mathcal{A}$.
Recall that the categories and functors we consider are $\Bbbk$-linear. Thus, in our case, we obtain exact sequences in $\mathbf{Vec}_{\Bbbk}$, rather than just in $\mathbf{Ab}$.

An object $Q \in \mathcal{A}$ is said to be {\it $\euler{C}$-projective} if there is a morphism $s: Q \rightarrow \euler{C}(Q)$ such that $\varepsilon_{Q} \circ s = \on{id}_{Q}$. By \cite[Lemma~2.5]{GHS}, this is equivalent to $Q$ being a direct summand of an object of the form $\euler{C}(Q')$, for some $Q' \in \mathcal{A}$. A sequence
\[\begin{tikzcd}
	\cdots & {Q_{2}} & {Q_{1}} & {Q_{0}} & X & 0
	\arrow[from=1-5, to=1-6]
	\arrow[from=1-1, to=1-2]
	\arrow[from=1-2, to=1-3]
	\arrow[from=1-3, to=1-4]
	\arrow[from=1-4, to=1-5]
\end{tikzcd}\]
is said to be a {\it $\euler{C}$-resolution} if it is $\euler{C}$-exact and $Q_{i}$ is $\euler{C}$-projective. Every object of $X$ admits a $\euler{C}$-resolution, known as the {\it bar resolution}, given by the following sequence:
\[\begin{tikzcd}
	\cdots & {\euler{C}^{2}(X)} & {\euler{C}(X)} & X & 0
	\arrow["{d_{1}}", from=1-2, to=1-3]
	\arrow["{d_{0}}", from=1-3, to=1-4]
	\arrow[from=1-4, to=1-5]
	\arrow[from=1-1, to=1-2]
\end{tikzcd}\]
where $d_{n} = \sum_{i=0}^{n} (-1)^{i}\euler{C}^{n-i}(\varepsilon_{\euler{C}^{i}(X)})$. In other words, it is the chain complex associated to the simplicial object
\[\begin{tikzcd}
	\cdots & {\euler{C}^{3}(X)} &&& {\euler{C}^{2}(X)} && {\euler{C}(X)} & X
	\arrow["{\varepsilon_{X}}", from=1-7, to=1-8]
	\arrow["{\euler{C}\varepsilon_{X}}", shift left=4, from=1-5, to=1-7]
	\arrow["{\varepsilon_{\euler{C}(X)}}"', shift right=3, from=1-5, to=1-7]
	\arrow["{\euler{C}^{2}\varepsilon_{X}}", shift left=5, curve={height=-6pt}, from=1-2, to=1-5]
	\arrow["{\varepsilon_{\euler{C}^{2}(X)}}"', shift right=5, curve={height=6pt}, from=1-2, to=1-5]
	\arrow["{\euler{C}\varepsilon_{\euler{C}(X)}}"{description}, from=1-2, to=1-5]
	\arrow[from=1-1, to=1-2]
	\arrow["\Delta"{description}, from=1-7, to=1-5]
	\arrow["{\euler{C}\Delta}"{description}, shift left=4, from=1-5, to=1-2]
	\arrow["{\Delta\euler{C}}"{description}, shift right=4, from=1-5, to=1-2]
\end{tikzcd}\]
under Dold-Kan correspondence (for the appropriate generalization of the case $\mathcal{A} = \mathbf{Ab}$, see \cite[Theorem~1.2.3.7]{Lu2}).

\begin{definition}\label{ResolventPair}
If the category $\mathcal{A}$ is abelian, and the comonad $\euler{C}$ is of the form $\euler{C} = \euler{F} \circ \euler{G}$, for an adjoint pair of functors
\[\begin{tikzcd}
	{\mathcal{B}} & {\mathcal{A}}
	\arrow[""{name=0, anchor=center, inner sep=0}, "{\euler{F}}", shift left=2, from=1-1, to=1-2]
	\arrow[""{name=1, anchor=center, inner sep=0}, "{\euler{G}}", shift left=2, from=1-2, to=1-1]
	\arrow["\dashv"{anchor=center, rotate=-90}, draw=none, from=0, to=1]
\end{tikzcd},\]
where $\mathcal{B}$ also is abelian, and the right adjoint $\euler{G}$ is exact and faithful, then $(\euler{F,G})$ is said to be a {\it resolvent pair}.
\end{definition}
  For a resolvent pair $(\mathrm{F,G})$, the comonad cohomology associated to the comonad $\mathrm{F}\circ \mathrm{G}$ can be recast in terms of relative homological algebra of \cite{Hoc} ,\cite{MacL}, see \cite[Lemma~2.16]{FGS} and \cite[Proposition~2.17]{FGS}.
  In particular, we have the following:
  \begin{proposition}[{\cite[Lemma~2.16.4 and Proposition~2.17.2]{FGS}}]\label{Soundness}
 If $\euler{C}$ is a comonad coming from a resolvent pair, then a $\euler{C}$-exact sequence is exact in the ordinary sense of homological algebra on the abelian category $\mathcal{A}$, and, similarly, a $\euler{C}$-resolution is a resolution in the sense of homological algebra on $\mathcal{A}$.
  \end{proposition}
Using Proposition~\ref{Soundness}, combined with the fact that colimits in functor categories are constructed pointwise, we obtain the following consequence:
\begin{corollary}\label{ResolvingIdentity}
 If $\euler{C}$ is a comonad coming from a resolvent pair, then there is a natural isomorphism
\[
  \on{Id}_{\mathcal{A}} \simeq \on{coeq}
  \Big(
\begin{tikzcd}[ampersand replacement=\&]
	{\euler{C}^{2}} \& {\euler{C}}
	\arrow["{\euler{C}\varepsilon}", shift left=2, from=1-1, to=1-2]
	\arrow["{\varepsilon\euler{C}}"', shift right=2, from=1-1, to=1-2]
\end{tikzcd}
  \Big).
\]
\end{corollary}

Recall that a $\Bbbk$-linear functor $\euler{F}: \mathcal{A} \rightarrow \mathcal{B}$ gives rise to an adjoint triple
\[\begin{tikzcd}[ampersand replacement=\&]
	{[\mathcal{A}^{\on{op}},\mathbf{Vec}_{\Bbbk}]} \&\& {[\mathcal{B}^{\on{op}},\mathbf{Vec}_{\Bbbk}]}
	\arrow[""{name=0, anchor=center, inner sep=0}, "{\euler{F}_{!}}"{description}, shift left=4, curve={height=-6pt}, from=1-1, to=1-3]
	\arrow[""{name=1, anchor=center, inner sep=0}, "{\euler{F}^{\ast}}"{description}, from=1-3, to=1-1]
	\arrow[""{name=2, anchor=center, inner sep=0}, "{\euler{F}_{\ast}}"{description}, shift right=4, curve={height=6pt}, from=1-1, to=1-3]
	\arrow["\dashv"{anchor=center, rotate=-90}, draw=none, from=1, to=2]
	\arrow["\dashv"{anchor=center, rotate=-90}, draw=none, from=0, to=1]
\end{tikzcd}\]
where $\euler{F}^{\ast}(\euler{P})(X) = \euler{P}(\euler{F}(X))$, and $\euler{F}_{!}$ and $\euler{F}_{\ast}$ are given by the left and right Kan extensions of $\yo_{\mathcal{B}} \circ \euler{F}$ along $\yo_{\mathcal{A}}$, respectively. Thus,
\[
 \euler{F}_{!}(\euler{Q})(Y) = \int^{X} \euler{Q}(X) \kotimes \Hom{Y,\euler{F}(X)} \text{ and } \euler{F}_{\ast}(\euler{Q})(Y) = \int_{X}\Hom{\Hom{\euler{F}(X),Y}, \euler{Q}(X)} = \on{Hom}_{[\mathcal{A}^{\on{op}},\mathbf{Vec}_{\Bbbk}]}(\Hom{\euler{F}(-),Y},\euler{Q}).
\]

As a consequence, the functor $\euler{F}^{\ast}$ is continuous and cocontinuous. Thus, the adjoint pair $(\euler{F}_{!},\euler{F}^{\ast})$ is resolvent if and only if $\euler{F}^{\ast}$ is faithful.

In the non-enriched case, \cite[Proposition~4.1]{AEBSV} provide an elementary proof of the following claim first established in \cite{CJSV}: $\euler{F}^{\ast}$ is faithful if and only if every object of $\mathcal{B}$ is a retract of an object of the form $\euler{F}(X)$, for some $X \in \mathcal{A}$. Following the terminology of \cite{CJSV}, these functors are known as {\it liberal functors}. Following the arguments of \cite{CJSV} in the case of the $2$-category $\mathbf{Cat}_{\Bbbk}$, rather than $\mathbf{Cat}$, we obtain an analogous notion for $\Bbbk$-linear functors, mirroring the change in the set of absolute weights:
\begin{definition}\label{LiberalFunctors}
 A $\Bbbk$-linear functor $\euler{F}: \mathcal{A} \rightarrow \mathcal{B}$ is said to be {\it liberal} if every object of $\mathcal{B}$ is a direct summand of a finite direct sum of objects of the form $\euler{F}(X)$, for $X \in \mathcal{A}$.
\end{definition}

We now give the $\Bbbk$-linear analogue of \cite[Proposition~4.1]{AEBSV}:

\begin{proposition}\label{SufficientCondition}
 Let $\euler{F}: \mathcal{A} \rightarrow \mathcal{B}$ be a $\Bbbk$-linear functor, with $\mathcal{B}$ additive and idempotent split. The functor $\euler{F}^{\ast}$ is faithful if and only if $\euler{F}$ is liberal.
\end{proposition}

\begin{proof}
Assume that $\euler{F}$ is liberal.
 Let $\euler{P,Q} \in [\mathcal{B}^{\on{op}},\mathbf{Vec}_{\Bbbk}]$, and let
$\begin{tikzcd}[ampersand replacement=\&,sep=scriptsize]
	{\euler{P}} \& {\euler{Q}}
	\arrow["\alpha", shift left=2, from=1-1, to=1-2]
	\arrow["\beta"', shift right=2, from=1-1, to=1-2]
\end{tikzcd}$ be a pair of transformations such that $\alpha \neq \beta$. Let $Z \in \mathcal{B}$ be such that $\alpha_{Z} \neq \beta_{Z}$. Let $(\varphi_{i})_{i=1}^{m}: \bigoplus_{i=1}^{m}\euler{F}(X_{i}) \twoheadrightarrow Z$ be a split epimorphism. Naturality of $\alpha$ yields the commutativity of
\[\begin{tikzcd}[ampersand replacement=\&,sep=scriptsize]
	{\euler{P}(\bigoplus_{i=1}^{m}\euler{F}(X_{i}))} \& {\euler{P}(Z)} \\
	{\euler{Q}(\bigoplus_{i=1}^{m}\euler{F}(X_{i}))} \& {\euler{Q}(Z)}
	\arrow["{\euler{P}((\varphi_{i}))}", shift left=1, from=1-1, to=1-2]
	\arrow["{\alpha_{Z}}", from=1-2, to=2-2]
	\arrow["{\alpha_{\bigoplus_{i=1}^{m}\euler{F}(X_{i})}}"', shift right=1, from=1-1, to=2-1]
	\arrow["{\euler{Q}((\varphi_{i}))}"', shift right=1, from=2-1, to=2-2]
\end{tikzcd}
\quad \text{ and of } \quad
\begin{tikzcd}[ampersand replacement=\&,sep=scriptsize]
	{\euler{P}(\bigoplus_{i=1}^{m}\euler{F}(X_{i}))} \& {\euler{P}(Z)} \\
	{\euler{Q}(\bigoplus_{i=1}^{m}\euler{F}(X_{i}))} \& {\euler{Q}(Z)}
	\arrow["{\euler{P}((\varphi_{i}))}", shift left=1, from=1-1, to=1-2]
	\arrow["{\beta_{Z}}", from=1-2, to=2-2]
	\arrow["{\beta_{\bigoplus_{i=1}^{m}\euler{F}(X_{i})}}"', shift right=1, from=1-1, to=2-1]
	\arrow["{\euler{Q}((\varphi_{i}))}"', shift right=1, from=2-1, to=2-2]
\end{tikzcd}
\]
Since $(\varphi_{i})$ is a split epimorphism, so is $\euler{P}((\varphi_{i}))$. The assumption $\alpha_{Z} \neq \beta_{Z}$ shows $\alpha_{Z} \circ (\varphi_{i}) \neq \beta_{Z} \circ (\varphi_{i})$. The naturality squares above then yield $\alpha_{\bigoplus_{i=1}^{m}\euler{F}(X_{i})} \neq \beta_{\bigoplus_{i=1}^{m}\euler{F}(X_{i})}$. Since $\alpha_{\bigoplus_{i=1}^{m}\euler{F}(X_{i})} = \bigoplus_{i=1}^{m}\alpha_{\euler{F}(X_{i})}$ and similarly for $\beta$, we find $\bigoplus_{i=1}^{m}\alpha_{\euler{F}(X_{i})} \neq \bigoplus_{i=1}^{m}\beta_{\euler{F}(X_{i})}$, which is the case if and only if there is some $j \in \setj{1,\ldots,m}$ such that $\alpha_{\euler{F}(X_{j})} \neq \beta_{\euler{F}(X_{j})}$. This shows $\euler{F}^{\ast}(\alpha) \neq \euler{F}^{\ast}(\beta)$.

For the converse, assume that $\euler{F}^{\ast}$ is faithful. This is equivalent to the counit $\euler{\varepsilon^{F}}: \euler{F_{!}F^{\ast}} \Rightarrow \on{Id}_{[\mathcal{B}^{\on{op}},\mathbf{Vec}_{\Bbbk}]}$ of the adjunction $(\euler{F_{!},F^{\ast}, \eta^{F},\varepsilon^{F}})$ being an epimorphism (see e.g. \cite[IV.3]{MacL}), and thus, equivalently, a pointwise epimorphism. For $\euler{P} \in [\mathcal{B}^{\on{op}},\mathbf{Vec}_{\Bbbk}]$ and $Y \in \mathcal{B}$, the component $\euler{\varepsilon^{F}_{P}}$ is given by the map
\[
 \int^{X \in \mathcal{A}} \euler{P}(\euler{F}(X)) \kotimes \Hom{Y,\euler{F}(X)} \rightarrow \euler{P}(Y),
\]
which, precomposed with the projection $\coprod_{X \in \mathcal{A}} \euler{P}(\euler{F}(X)) \kotimes \Hom{Y,\euler{F}(X)} \rightarrow \int^{X \in \mathcal{A}} \euler{P}(\euler{F}(X)) \kotimes \Hom{Y,\euler{F}(X)}$, gives the map corresponding to the collection of morphisms
\[
\begin{aligned}
  \euler{P}(\euler{F}(X)) \kotimes \Hom{Y,\euler{F}(X)} \rightarrow &\euler{P}(Y) \\
  v \otimes f \longmapsto &\euler{P}(f)(v)
\end{aligned}
\]
 indexed by $X$. Setting $\euler{P} = \Hom{-,Y}$, we find the epimorphism
 \[
 \begin{aligned}
  \coprod_{X \in \mathcal{A}} \Hom{\euler{F}(X),Y} \kotimes \Hom{Y,\euler{F}(X)} &\twoheadrightarrow \Hom{Y,Y} \\
  f \otimes g &\mapsto g \circ f
 \end{aligned}
 \]
 An element in the domain of the above morphism is of the form $\sum_{i=1}^{n} f_{i} \otimes g_{i}$, for some $f_{i}: \euler{F}(X_{i}) \rightarrow Y$ and $g_{i}: Y \rightarrow \euler{F}(X_{i}), \, X_{i} \in \mathcal{A}$ and $n \in \mathbb{Z}_{\geq 0}$. Since the sum is finite, this gives a factorization of $\on{id}_{Y}$ through $\bigoplus_{i=1}^{n} \euler{F}(X_{i})$, showing that the former is a summand of the latter.
\end{proof}

\begin{remark}
 In the case the codomain is not additive or idempotent split, we may modify Definition~\ref{LiberalFunctors} and Proposition~\ref{SufficientCondition} to refer to the Cauchy completion of $\mathcal{B}$, universally endowing it with finite direct sums and summands, rather than $\mathcal{B}$ itself; alternatively, we may require $\Hom{-,Y}$ to be a direct summand of a finite direct sum of those of the form $\Hom{-,\euler{F}(X_{i})}$.
\end{remark}

Combining Proposition~\ref{Soundness} with Proposition~\ref{SufficientCondition}, we obtain the following statement:

\begin{corollary}
 Let $\csym{S}$ be a semigroup category. If $(\mathrm{F},\mathrm{F}^{\blackdiamond}, \eta^{l},\eta^{r}, \varepsilon^{l},\varepsilon^{r})$ is an adjunction in $\csym{S}$ such that $\mathrm{F}\boxtimes- $ is liberal, then
\begin{equation}\label{LeftUnitality}
  \on{Id}_{[\ccf{S}^{\on{op}},\mathbf{Vec}_{\Bbbk}]} \simeq
\on{Coeq}\Bigg(
\begin{tikzcd}[ampersand replacement=\&]
	{\Hom{-,\mathrm{FF^{\blackdiamond}FF^{\blackdiamond}}} \oast -} \&\&\& {\Hom{-,\mathrm{FF^{\blackdiamond}}} \oast -}
	\arrow["{(\varepsilon^{\mathrm{F},r}\mathrm{F}\mathrm{F}^{\blackdiamond})_{!}}", shift left=3, from=1-1, to=1-4]
	\arrow["{(\mathrm{FF^{\blackdiamond}}\varepsilon^{\mathrm{F},r})_{!}}"', shift right=3, from=1-1, to=1-4]
\end{tikzcd}\Bigg)
\end{equation}
Similarly, if the functor $-\boxtimes \mathrm{F}^{\blackdiamond}$ is liberal, then
\begin{equation}\label{RightUnitality}
\begin{tikzcd}
  \on{Id}_{[\ccf{S}^{\on{op}},\mathbf{Vec}_{\Bbbk}]} \simeq
\on{Coeq}\Bigg(
	{- \oast \Hom{-,\mathrm{FF^{\blackdiamond}FF^{\blackdiamond}}}} &&& {-\oast \Hom{-,\mathrm{FF^{\blackdiamond}}}}
	\arrow["{(\varepsilon^{\mathrm{F},l}\mathrm{F}\mathrm{F}^{\blackdiamond})_{!}}", shift left=3, from=1-1, to=1-4]
	\arrow["{(\mathrm{FF^{\blackdiamond}}\varepsilon^{\mathrm{F},l})_{!}}"', shift right=3, from=1-1, to=1-4]
\end{tikzcd}\Bigg)
\end{equation}
\end{corollary}

\begin{proof}
 From Proposition~\ref{SufficientCondition} we find that the (ordinary) adjunction $((\mathrm{F} \boxtimes -)_{!},(\mathrm{F}^{\blackdiamond})_{!}, \eta^{\mathrm{F},r}_{!},\varepsilon^{\mathrm{F},r}_{!})$ gives a resolvent pair. Since the Yoneda embedding is strong monoidal, we have $(\mathrm{F} \boxtimes -)_{!} \simeq \Hom{-,\mathrm{F}} \oast -$; the first claim now follows from Proposition~\ref{Soundness}. The second claim follows similarly - observe that we require $- \boxtimes \mathrm{F}^{\blackdiamond}$ and not $-\boxtimes \mathrm{F}$ to be liberal, since the former is left adjoint to the latter.
\end{proof}

\begin{lemma}
 If $\csym{S}$ is a rigid semigroup category, then $\mathrm{F} \boxtimes -$ is liberal if and only if $- \boxtimes \mathrm{F}^{\blackdiamond}$ is liberal.
\end{lemma}

\begin{proof}
 Assume $\mathrm{F} \boxtimes -$ is liberal. By definition, any $\mathrm{H} \in \csym{S}$ is a direct summand of $\mathrm{F} \boxtimes \mathrm{G}$ for some $\mathrm{G} \in \csym{S}$. Let $\mathrm{K}$ be such that, for the $\mathrm{H}$ above, ${}^{\blackdiamond}\!\mathrm{H}$ is a direct summand of $\mathrm{F}\boxtimes \mathrm{K}$. Then, by Proposition~\ref{Dualize}, $({}^{\blackdiamond}\!\mathrm{H})^{\blackdiamond} \simeq \mathrm{H}$ is a direct summand of $\mathrm{K}^{\blackdiamond}\mathrm{F}^{\blackdiamond}$.
\end{proof}
A crucial application of Axiom~\ref{Involution} allows us to ``internalize'' Diagram~\ref{LeftUnitality} and Diagram~\ref{RightUnitality}: we find a diagram inside $[\csym{S}^{\on{op}},\mathbf{Vec}_{\Bbbk}]$ such that tensoring with it from the left gives Diagram~\ref{LeftUnitality}, and tensoring with it from the right gives Diagram~\ref{RightUnitality}.
\begin{theorem}\label{LiberalUnitBar}
 Let $\csym{S}$ be a rigid semigroup category and let $\mathrm{F}$ be such that $\mathrm{F} \boxtimes -$ is liberal. Then
\begin{equation}\label{UnitResolution}
  \mathbb{1}_{[\ccf{S}^{\on{op}},\mathbf{Vec}_{\Bbbk}]} \simeq
  \on{Coeq}\Bigg(
\begin{tikzcd}
	{\Hom{-,\mathrm{FF^{\blackdiamond}FF^{\blackdiamond}}}} && {\Hom{-,\mathrm{FF^{\blackdiamond}}}}
	\arrow["{\Hom{-,\varepsilon^{\mathrm{F},r}_{\mathrm{F}}\mathrm{F}^{\blackdiamond}}}", shift left=3, from=1-1, to=1-3]
	\arrow["{\Hom{-,\mathrm{F}\varepsilon^{\mathrm{F},l}_{\mathrm{F}^{\blackdiamond}}}}"', shift right=3, from=1-1, to=1-3]
\end{tikzcd}\Bigg)
\end{equation}

\end{theorem}

\begin{proof}
 Since the units and counits in the adjunction $(\mathrm{F},\mathrm{F}^{\blackdiamond}, \eta^{l},\eta^{r}, \varepsilon^{l},\varepsilon^{r})$ are $\csym{S}$-module transformations, we have $\varepsilon^{\mathrm{F},r}\mathrm{FF^{\blackdiamond}} = \varepsilon^{\mathrm{F},r}_{\mathrm{F}}\mathrm{F}^{\blackdiamond} \boxtimes -$, and similarly, $\mathrm{FF^{\blackdiamond}}\varepsilon^{\mathrm{F},l} = - \boxtimes \mathrm{F}\varepsilon^{\mathrm{F},l}_{\mathrm{F}^{\blackdiamond}}$.
 Using Axiom~\ref{Involution}, we also find $\mathrm{FF^{\blackdiamond}}\varepsilon^{\mathrm{F},r} = \mathrm{F}\varepsilon^{\mathrm{F},l}_{\mathrm{F}^{\blackdiamond}} \boxtimes -$ and $\varepsilon^{\mathrm{F},l}\mathrm{FF^{\blackdiamond}} = -\boxtimes \varepsilon^{\mathrm{F},r}_{\mathrm{F}}\mathrm{F}^{\blackdiamond}$. Using these identifications, we may rewrite Diagram~\eqref{LeftUnitality} as
\begin{equation}\label{COast}
\on{Id}_{[\ccf{S}^{\on{op}},\mathbf{Vec}_{\Bbbk}]} \simeq
\on{Coeq}\Bigg(
\begin{tikzcd}
	{\Hom{-,\mathrm{FF^{\blackdiamond}FF^{\blackdiamond}}} \oast -} &&& {\Hom{-,\mathrm{FF^{\blackdiamond}}} \oast -}
	\arrow["{\Hom{-,\varepsilon^{\mathrm{F},r}_{\mathrm{F}}\mathrm{F}^{\blackdiamond}}\oast-}", shift left=3, from=1-1, to=1-4]
	\arrow["{\Hom{-,\mathrm{F}\varepsilon^{\mathrm{F},l}_{F^{\blackdiamond}}}\oast -}"', shift right=3, from=1-1, to=1-4]
\end{tikzcd}
\Bigg)
\end{equation}
and Diagram~\eqref{RightUnitality} as
\begin{equation}\label{OastC}
\on{Id}_{[\ccf{S}^{\on{op}},\mathbf{Vec}_{\Bbbk}]} \simeq
\on{Coeq}\Bigg(
\begin{tikzcd}
	{-\oast \Hom{-,\mathrm{FF^{\blackdiamond}FF^{\blackdiamond}}}} &&& {\oast \Hom{-,\mathrm{FF^{\blackdiamond}}}}
	\arrow["{- \oast \Hom{-,\varepsilon^{\mathrm{F},r}_{\mathrm{F}}\mathrm{F}^{\blackdiamond}}}"', shift right=3, from=1-1, to=1-4]
	\arrow["{-\oast \Hom{-,\mathrm{F}\varepsilon^{\mathrm{F},l}_{F^{\blackdiamond}}}}", shift left=3, from=1-1, to=1-4]
\end{tikzcd}\Bigg).
\end{equation}
Write $C$ for the right-hand side in Diagram~\eqref{UnitResolution}. Using the cocontinuity of $- \oast -$ in both variables together with Diagram~\eqref{COast} and Diagram~\eqref{OastC}, we conclude that $C \oast - \simeq \on{Id}_{[\ccf{S}^{\on{op}},\mathbf{Vec}_{\Bbbk}]}$ and $- \oast C \simeq \on{Id}_{[\ccf{S}^{\on{op}},\mathbf{Vec}_{\Bbbk}]}$. Proposition~\ref{EGNOUnit} now implies $C \simeq \mathbb{1}_{[\ccf{S}^{\on{op}},\mathbf{Vec}_{\Bbbk}]}$, as claimed.
\end{proof}

\begin{remark}
Identifying \eqref{UnitResolution} as a diagram in the image of the Yoneda embedding, we may write
\[
 \mathbb{1} \simeq \on{Coker}\Big(
\begin{tikzcd}
	{\mathrm{FF^{\blackdiamond}FF^{\blackdiamond}}} && {\mathrm{FF^{\blackdiamond}}}
	\arrow["{\varepsilon^{\mathrm{F},r}_{\mathrm{F}}\mathrm{F}^{\blackdiamond}-\mathrm{F}\varepsilon^{\mathrm{F},l}_{\mathrm{F}^{\blackdiamond}}}", from=1-1, to=1-3]
\end{tikzcd}\Big)
\]
\end{remark}

\subsection{Lifting of module categories}\label{s32}

One of the motivating observations for this chapter is the following analogue of Proposition~\ref{EGNOUnit} for module categories:
\begin{proposition}[{\cite[Proposition~2.4.3 and Chapter~7.1]{EGNO}}]\label{AnotherEGNOLemma}
 Given monoidal categories $\csym{C},\csym{D}$, a semigroup functor $\mathbf{F}$ from $\csym{C}$ to $\csym{D}$ is monoidal if and only if $\mathbf{F}(\mathbb{1}_{\ccf{C}}) \simeq \mathbb{1}_{\ccf{D}}$.
\end{proposition}

A special case of this is that of $\csym{C}$-module categories. To avoid possible confusion, we introduce some non-standard terminology allowing us to separate the monoidal case from the semigroup case:
\begin{terminology}\label{dispel}
 We refer to module categories over a monoidal category $\csym{C}$, in the usual sense of \cite[Definition~7.1.1]{EGNO}, as {\it monoidal $\csym{C}$-module categories}. A module category over $\csym{C}$ in the sense of Definition~\ref{SemigroupModules}, thus not requiring $\mathbb{1}_{\ccf{C}}$ to act isomorphically to the identity functor, will be referred to as a {\it semigroup $\csym{C}$-module category}. Thus, in parallel to semigroups and monoids, the monoidal notion is a special, unital case of the semigroup notion.
\end{terminology}

\begin{definition}\label{MV}
Given a semigroup category $\csym{S}$ and a semigroup $\csym{S}$-module category $\mathbf{M}$, Yoneda extension endows the category $[\mathbf{M}^{\on{op}},\mathbf{Vec}_{\Bbbk}]$ with a semigroup $\csym{S}$-module structure, via $\mathrm{F} \star_{[\mathbf{M}^{\on{op}},\mathbf{Vec}_{\Bbbk}]} - = (\mathrm{F} \star_{\mathbf{M}} -)_{!}$. Extending the corresponding semigroup functor $\csym{S} \rightarrow \mathbf{Cat}_{\Bbbk}(\mathbf{M},\mathbf{M})$ via the universal property in Theorem~\ref{UniversalPropertyDay}, we obtain a cocontinuous semigroup functor $[\csym{S}^{\on{op}},\mathbf{Vec}_{\Bbbk}] \rightarrow \mathbf{Cat}_{\Bbbk}([\mathbf{M}^{\on{op}},\mathbf{Vec}],[\mathbf{M}^{\on{op}},\mathbf{Vec}])$, which makes $[\mathbf{M}^{\on{op}},\mathbf{Vec}_{\Bbbk}]$ into a semigroup $[\csym{S}^{\on{op}},\mathbf{Vec}_{\Bbbk}]$-module category.
\end{definition}

If $\csym{S}$ is a monoidal category and $\mathbf{M}$ a monoidal $\csym{S}$-module category, then the presheaf category $[\mathbf{M}^{\on{op}},\mathbf{Vec}_{\Bbbk}]$ is a monoidal $[\csym{S}^{\on{op}},\mathbf{Vec}_{\Bbbk}]$-module category. Recall from Theorem~\ref{PromonoidalUnital} and Theorem~\ref{LiberalUnitBar} that if $\csym{S}$ is rigid, then $[\csym{S}^{\on{op}},\mathbf{Vec}_{\Bbbk}]$ is monoidal, and if there is an object $\mathrm{F}$ such that $\mathrm{F} \boxtimes -$ is liberal, then the monoidal unit in $[\csym{S}^{\on{op}},\mathbf{Vec}_{\Bbbk}]$ admits a resolution via a bar complex in $\csym{S}$. We now study the problem of finding sufficient conditions for $\mathbf{M}$ to make $[\mathbf{M}^{\on{op}},\mathbf{Vec}_{\Bbbk}]$ into a monoidal, and not only semigroup, $[\csym{S}^{\on{op}},\mathbf{Vec}_{\Bbbk}]$-module category.

We now study the problem of preservation of adjunctions in semigroup module categories and, more generally, under general semigroup functors.

\begin{definition}\label{PairsPreserved}
 Let $\csym{S},\csym{S}'$ be semigroup categories and let $(\mathrm{F},\mathrm{F}^{\blackdiamond}, \eta^{\mathrm{F},l},\eta^{\mathrm{F},r}, \varepsilon^{\mathrm{F},l},\varepsilon^{\mathrm{F},r})$ be an adjunction in $\csym{S}$. We say that a semigroup functor $\mathbf{G}: \csym{S} \rightarrow \csym{S}'$ {\it preserves the adjoint pair} $(\mathrm{F},\mathrm{F}^{\blackdiamond})$ if $(\mathbf{G}(\mathrm{F}),\mathbf{G}(\mathrm{F}^{\blackdiamond}))$ is an adjoint pair.

 If $\mathbf{M}$ is a semigroup $\csym{S}$-module category and $\mathbf{M}(-): \csym{S} \rightarrow \mathbf{Cat}(\mathbf{M},\mathbf{M})$ is its associated semigroup functor, we say that $\mathbf{M}$ preserves the adjoint pair $(\mathrm{F},\mathrm{F}^{\blackdiamond})$ if $\mathbf{M}(-)$ does.
 Thus, $\mathbf{M}$ preserves the adjoint pair $(\mathrm{F},\mathrm{F}^{\blackdiamond})$ if the pair $(\mathbf{M}\mathrm{F},\mathbf{M}\mathrm{F}^{\blackdiamond})$ of functors is an adjoint pair in the ordinary sense.
\end{definition}

Example~\ref{NoAdjunctions} below illustrates the fact that a semigroup functor does not necessarily preserve adjoint pairs in the sense of Definition~\ref{PairsPreserved}, even if both the domain and the codomain are in fact monoidal categories.
\begin{example}\label{NoAdjunctions}
 Let $\mathsf{C}^{s}(\setj{\ast})$ denote the strict monoidal category obtained by applying the construction of Example~\ref{NaiveCat} to the one-element monoid $\setj{\ast}$. Clearly, this monoidal category is rigid.

 Given a category $\mathbf{M}$, a functor from $\mathsf{C}^{s}(\setj{\ast})$ to $\mathbf{Cat}(\mathbf{M},\mathbf{M})$ is given by a choice of a functor $\mathbf{M}(\ast): \mathbf{M} \rightarrow \mathbf{M}$.

 A monoidal $\mathsf{C}^{s}(\setj{\ast})$-module category $\mathbf{M}$ is a monoidal functor $\mathsf{C}^{s}(\setj{\ast}) \rightarrow \mathbf{Cat}(\mathbf{M},\mathbf{M})$. As such, it is specified by a functor $\mathbf{M}(\ast): \mathbf{M} \rightarrow \mathbf{M}$, together with a coherent isomorphism $\mathbf{M}(\ast) \circ \mathbf{M}(\ast) \simeq \mathbf{M}(\ast)$, and such that $\mathbf{M}(\ast) \simeq \on{Id}_{\mathbf{M}}$.

 A semigroup $\mathsf{C}^{s}(\setj{\ast})$-module category $\mathbf{M}$ is a semigroup functor $\mathsf{C}^{s}(\setj{\ast}) \rightarrow \mathbf{Cat}(\mathbf{M},\mathbf{M})$. As such, it is specified by a functor $\mathbf{M}(\ast): \mathbf{M} \rightarrow \mathbf{M}$, together with a coherent isomorphism $\mathbf{M}(\ast) \circ \mathbf{M}(\ast) \simeq \mathbf{M}(\ast)$, which does not necessarily satisfy $\mathbf{M}(\ast) \simeq \on{Id}_{\mathbf{M}}$.

 In view of the above description of semigroup $\mathsf{C}^{s}(\setj{\ast})$-module categories, any idempotent endofunctor of a category which is not self-adjoint gives a counterexample to the preservation of adjoint pairs being automatic.

 We now give an example of such a functor. Let $A$ be a basic, finite-dimensional, unital associative $\Bbbk$-algebra with a decomposition $1_{A} = \sum_{i=1}^{n} e_{i}$ into a sum of primitive, orthogonal idempotents. We define the endofunctor $\euler{S}_{A}$ of $A\!\on{-mod}$ by sending an $A$-module $M= \bigoplus_{i=1}^{n}e_{i}M$ to the $A$-module with the same underlying vector space and the same action of the idempotents $e_{1},\ldots,e_{n}$, but where $a\cdot m = 0$ for all $a \in \on{Rad}A, m \in M$. This functor is strictly idempotent, but it is easy to verify that it is not self-adjoint.
\end{example}

In order to prove the main results of this section, we need to introduce a variant of Definition~\ref{PairsPreserved} which requires an additional coherence condition.
Later, in Proposition~\ref{Djunctions}, we will show that preservation of adjoint pairs in the sense of Definition~\ref{PairsPreserved} is in fact equivalent to the a priori stronger, coherent variant introduced in Definition~\ref{Respect} below.

\begin{definition}\label{Respect}
 Let $\csym{S}$ be a semigroup category, and let $\mathbf{M}$ be a semigroup $\csym{S}$-module category.
 An adjunction $(\mathrm{F},\mathrm{F}^{\blackdiamond}, \eta^{\mathrm{F},l},\eta^{\mathrm{F},r}, \varepsilon^{\mathrm{F},l},\varepsilon^{\mathrm{F},r})$ in $\csym{S}$ is said to be {\it respected in $\mathbf{M}$} if there is an adjunction $(\mathbf{M}\mathrm{F}, \mathbf{M}\mathrm{F}^{\blackdiamond},\eta^{\mathrm{F},\mathbf{M}},\varepsilon^{\mathrm{F},\mathbf{M}})$ such that
\begin{equation}\label{CoherentlyRespectful}
\begin{aligned}
 \mathbf{M}\mathrm{H}\varepsilon^{\mathrm{F},\mathbf{M}}_{X} = \mathbf{M}(\varepsilon^{\mathrm{F},l}_{\mathrm{H}})_{X} \text{ and }
 \mathbf{M}\mathrm{H}\eta^{\mathrm{F},\mathbf{M}}_{X} = \mathbf{M}(\eta^{\mathrm{F},l}_{\mathrm{H}})_{X} \text{ for all }\mathrm{H} \in \csym{S} \text{ and } X \in \mathbf{M}. \\
\end{aligned}
\end{equation}
\end{definition}

 \begin{proposition}\label{Djunctions}
  Let $(\mathrm{F},\mathrm{F}^{\blackdiamond}, \eta^{\mathrm{F},l},\eta^{\mathrm{F},r}, \varepsilon^{\mathrm{F},l},\varepsilon^{\mathrm{F},r})$ be an adjunction in in a semigroup category $\csym{S}$, and let $\mathbf{M}$ be a semigroup $\csym{S}$-module category.

  The adjunction $(\mathrm{F},\mathrm{F}^{\blackdiamond}, \eta^{\mathrm{F},l},\eta^{\mathrm{F},r}, \varepsilon^{\mathrm{F},l},\varepsilon^{\mathrm{F},r})$ is respected in $\mathbf{M}$ if and only if the adjoint pair $(\mathbf{M}\mathrm{F},\mathbf{M}\mathrm{F}^{\blackdiamond})$ is preserved in $\mathbf{M}$, i.e. if the pair of functors $(\mathbf{M}\mathrm{F},\mathbf{M}\mathrm{F}^{\blackdiamond})$ is an adjoint pair.
 \end{proposition}

 \begin{proof}
  If $(\mathrm{F},\mathrm{F}^{\blackdiamond}, \eta^{\mathrm{F},l},\eta^{\mathrm{F},r}, \varepsilon^{\mathrm{F},l},\varepsilon^{\mathrm{F},r})$ is preserved in $\mathbf{M}$, then, by definition, $(\mathbf{M}\mathrm{F},\mathbf{M}\mathrm{F}^{\blackdiamond})$ is an adjoint pair.

  Assume that $(\mathbf{M}\mathrm{F},\mathbf{M}\mathrm{F}^{\blackdiamond})$ is an adjoint pair. Let $(\mathbf{M}\mathrm{F},\mathbf{M}\mathrm{F}^{\blackdiamond},\widehat{\eta},\widehat{\varepsilon})$ be an adjunction realizing $(\mathbf{M}\mathrm{F},\mathbf{M}\mathrm{F}^{\blackdiamond})$ as an adjoint pair. Define
  \[
   \eta = (\mathbf{M}\mathrm{F}^{\blackdiamond} \hcomp \widehat{\varepsilon} \hcomp \widehat{\varepsilon} \hcomp \mathbf{M}\mathrm{F}) \circ \left(\mathbf{M}\left(\mathrm{F^{\blackdiamond}F}\eta^{\mathrm{F},l}_{\mathrm{F}^{\blackdiamond}}\mathrm{F}\right)\right) \circ (\widehat{\eta} \hcomp \widehat{\eta})
  \]
  and
  \[
  \varepsilon = (\widehat{\varepsilon} \hcomp \widehat{\varepsilon}) \circ \left(\mathbf{M}\left(\mathrm{F^{\blackdiamond}F}\varepsilon^{\mathrm{F},l}_{\mathrm{F}^{\blackdiamond}}\mathrm{F}\right)\right) \circ (\mathbf{M}\mathrm{F} \hcomp \widehat{\eta} \hcomp \widehat{\eta} \hcomp \mathbf{M}\mathrm{F}^{\blackdiamond}).
  \]
 Clearly, the above assignments define natural transformations $\eta: \on{Id}_{\mathbf{M}} \Rightarrow \mathbf{M}(\mathrm{F^{\blackdiamond}F})$ and $\varepsilon: \mathbf{M}(\mathrm{FF^{\blackdiamond}}) \Rightarrow \on{Id}_{\mathbf{M}}$.
 We will use string diagrams to show that $(\mathbf{M}\mathrm{F},\mathbf{M}\mathrm{F}^{\blackdiamond}, \eta, \varepsilon)$ is an adjunction, via which $(\mathrm{F},\mathrm{F}^{\blackdiamond}, \eta^{\mathrm{F},l},\eta^{\mathrm{F},r}, \varepsilon^{\mathrm{F},l},\varepsilon^{\mathrm{F},r})$ is respected in $\mathbf{M}$.

 Let $X \in \mathbf{M}$. We represent it diagrammatically as a thick strand right-most in the diagram. This justifies the following notation:
{\setlength{\abovedisplayskip}{3pt}
\setlength{\belowdisplayskip}{3pt}
 \begin{equation}
  \begin{tikzpicture}[scale=0.5]
	\begin{pgfonlayer}{nodelayer}
		\node [style=none] (1) at (-0.75, 2) {};
		\node [style=none] (2) at (0, 1) {};
		\node [style=none] (3) at (0.75, 2) {};
		\node [style=none] (4) at (1.5, 2) {};
		\node [style=none] (5) at (1.5, 0) {};
		\node [style=none] (6) at (-1.25, 1) {$=$};
		\node [style=none] (7) at (-2, 1) {$\widehat{\eta}_{X}$};
	\end{pgfonlayer}
	\begin{pgfonlayer}{edgelayer}
		\draw [in=0, out=-90, looseness=1.25] (3.center) to (2.center);
		\draw [in=-180, out=-90, looseness=1.25] (1.center) to (2.center);
		\draw [style=thickstrand] (5.center) to (4.center);
	\end{pgfonlayer}
\end{tikzpicture}
  \quad \text{ and } \quad
 \begin{tikzpicture}[scale=0.5]
	\begin{pgfonlayer}{nodelayer}
		\node [style=none] (1) at (-0.75, 0) {};
		\node [style=none] (2) at (0, 1) {};
		\node [style=none] (3) at (0.75, 0) {};
		\node [style=none] (4) at (2, 2) {};
		\node [style=none] (5) at (2, 0) {};
		\node [style=none] (6) at (-1, 1) {=};
		\node [style=none] (7) at (-1.75, 1) {$\widehat{\varepsilon}_{X}$};
	\end{pgfonlayer}
	\begin{pgfonlayer}{edgelayer}
		\draw [in=0, out=90, looseness=1.25] (3.center) to (2.center);
		\draw [in=180, out=90, looseness=1.25] (1.center) to (2.center);
		\draw [style=thickstrand] (5.center) to (4.center);
	\end{pgfonlayer}
\end{tikzpicture}
\end{equation}}
Further, we adapt the notation of Example~\ref{FreeStrings}. Thus, for example, we have
{\setlength{\abovedisplayskip}{3pt}
\setlength{\belowdisplayskip}{3pt}
\[
\begin{tikzpicture}[scale=0.5]
	\begin{pgfonlayer}{nodelayer}
		\node [style=none, label={below:$\mathrm{F}$}] (1) at (-0.75, 0) {};
		\node [style=none] (2) at (0, 1) {};
		\node [style=none, label={below:$\mathrm{F}^{\blackdiamond}$}] (3) at (0.75, 0) {};
		\node [style=none] (4) at (2, 2) {};
		\node [style=none, label={below:$X$}] (5) at (2, 0) {};
		\node [style=none] (6) at (-3.25, 1) {=};
		\node [style=none] (7) at (-5.75, 1) {$\mathbf{M}(\varepsilon^{l,\mathrm{F}}_{\mathrm{G}})_{X}$};
		\node [style=none, label={below:$\mathrm{G}$}] (8) at (-2, 0) {};
		\node [style=none] (9) at (-2, 2) {};
		\node [style=none] (10) at (-2, 1.5) {};
	\end{pgfonlayer}
	\begin{pgfonlayer}{edgelayer}
		\draw [in=0, out=90, looseness=1.25] (3.center) to (2.center);
		\draw [in=180, out=90, looseness=1.25] (1.center) to (2.center);
		\draw [style=thickstrand] (5.center) to (4.center);
		\draw (8.center) to (9.center);
		\draw [style=bieski, in=0, out=90, looseness=0.75] (2.center) to (10.center);
	\end{pgfonlayer}
\end{tikzpicture}
\]
}
Using this notation, we find
{\setlength{\abovedisplayshortskip}{3pt}
\setlength{\belowdisplayshortskip}{3pt}\[
\begin{tikzpicture}[scale=0.5]
	\begin{pgfonlayer}{nodelayer}
		\node [style=none] (1) at (-0.25, 1.5) {};
		\node [style=none] (2) at (0.5, 0.5) {};
		\node [style=none] (3) at (1.25, 1.5) {};
		\node [style=none] (5) at (2, 0) {};
		\node [style=none] (6) at (-5.25, 1.5) {};
		\node [style=none] (7) at (-4.5, 0.5) {};
		\node [style=none] (8) at (-3.75, 1.5) {};
		\node [style=none] (9) at (-2.75, 3) {};
		\node [style=none] (10) at (-2, 2) {};
		\node [style=none] (11) at (-1.25, 3) {};
		\node [style=none] (12) at (-0.25, 3) {};
		\node [style=none] (13) at (1.25, 3) {};
		\node [style=none] (14) at (-3.75, 3) {};
		\node [style=none] (15) at (-5.25, 3) {};
		\node [style=none] (16) at (-3.25, 3.75) {};
		\node [style=none] (17) at (-2.75, 3) {};
		\node [style=none] (18) at (-3.25, 3.75) {};
		\node [style=none] (19) at (-1.25, 3) {};
		\node [style=none] (20) at (-0.75, 3.75) {};
		\node [style=none] (21) at (-0.25, 3) {};
		\node [style=none] (22) at (-0.75, 3.75) {};
		\node [style=none] (23) at (1.25, 5) {};
		\node [style=none] (24) at (-5.25, 5) {};
		\node [style=none] (25) at (2, 5) {};
		\node [style=none] (26) at (-6, 3) {$=$};
		\node [style=none] (27) at (-6.75, 3) {$\eta_{X}$};
	\end{pgfonlayer}
	\begin{pgfonlayer}{edgelayer}
		\draw [in=0, out=-90, looseness=1.25] (3.center) to (2.center);
		\draw [in=-180, out=-90, looseness=1.25] (1.center) to (2.center);
		\draw [in=0, out=-90, looseness=1.25] (8.center) to (7.center);
		\draw [in=-180, out=-90, looseness=1.25] (6.center) to (7.center);
		\draw [in=0, out=-90, looseness=1.25] (11.center) to (10.center);
		\draw [in=-180, out=-90, looseness=1.25] (9.center) to (10.center);
		\draw [style=bieski, in=0, out=-90, looseness=0.75] (10.center) to (8.center);
		\draw (14.center) to (8.center);
		\draw (15.center) to (6.center);
		\draw (12.center) to (1.center);
		\draw (13.center) to (3.center);
		\draw [in=180, out=90] (14.center) to (16.center);
		\draw [in=0, out=90] (17.center) to (18.center);
		\draw [in=180, out=90] (19.center) to (20.center);
		\draw [in=0, out=90] (21.center) to (22.center);
		\draw (15.center) to (24.center);
		\draw (13.center) to (23.center);
		\draw [style=thickstrand] (5.center) to (25.center);
	\end{pgfonlayer}
\end{tikzpicture}
\]}
The following shows that Equation~\ref{CoherentlyRespectful} holds for $\eta$:
{\setlength{\abovedisplayskip}{3pt}
\setlength{\belowdisplayskip}{3pt}\[
\begin{tikzpicture}[scale=0.5]
	\begin{pgfonlayer}{nodelayer}
		\node [style=none] (1) at (-0.25, 1.5) {};
		\node [style=none] (2) at (0.5, 0.5) {};
		\node [style=none] (3) at (1.25, 1.5) {};
		\node [style=none] (5) at (2, 0) {};
		\node [style=none] (6) at (-5.25, 1.5) {};
		\node [style=none] (7) at (-4.5, 0.5) {};
		\node [style=none] (8) at (-3.75, 1.5) {};
		\node [style=none] (9) at (-2.75, 3) {};
		\node [style=none] (10) at (-2, 2) {};
		\node [style=none] (11) at (-1.25, 3) {};
		\node [style=none] (12) at (-0.25, 3) {};
		\node [style=none] (13) at (1.25, 3) {};
		\node [style=none] (14) at (-3.75, 3) {};
		\node [style=none] (15) at (-5.25, 3) {};
		\node [style=none] (17) at (-2.75, 3) {};
		\node [style=none] (18) at (-3.25, 3.75) {};
		\node [style=none] (19) at (-1.25, 3) {};
		\node [style=none] (21) at (-0.25, 3) {};
		\node [style=none] (22) at (-0.75, 3.75) {};
		\node [style=none] (23) at (1.25, 5) {};
		\node [style=none] (24) at (-5.25, 5) {};
		\node [style=none] (25) at (2, 5) {};
		\node [style=none] (26) at (-6, 5) {};
		\node [style=none] (27) at (-6, 0) {};
		\node [style=none] (28) at (3, 2.5) {$=$};
		\node [style=none] (29) at (9.75, 2.25) {};
		\node [style=none] (30) at (10.5, 1.5) {};
		\node [style=none] (31) at (11.25, 2.25) {};
		\node [style=none] (32) at (12, 0) {};
		\node [style=none] (33) at (4.75, 2.25) {};
		\node [style=none] (34) at (5.5, 1.5) {};
		\node [style=none] (35) at (6.25, 2.25) {};
		\node [style=none] (37) at (8, 1.25) {};
		\node [style=none] (40) at (11.25, 3.25) {};
		\node [style=none] (41) at (6.25, 3.25) {};
		\node [style=none] (42) at (4.75, 3.25) {};
		\node [style=none] (44) at (7.25, 3.25) {};
		\node [style=none] (45) at (6.75, 4.25) {};
		\node [style=none] (46) at (8.75, 3.25) {};
		\node [style=none] (48) at (9.75, 3.25) {};
		\node [style=none] (49) at (9.25, 4.25) {};
		\node [style=none] (50) at (11.25, 5) {};
		\node [style=none] (51) at (4.75, 5) {};
		\node [style=none] (52) at (12, 5) {};
		\node [style=none] (53) at (4, 5) {};
		\node [style=none] (54) at (4, 0.5) {};
		\node [style=none] (55) at (4, 0) {};
		\node [style=none] (56) at (13, 2.5) {$=$};
		\node [style=none] (57) at (14, 0) {};
		\node [style=none] (58) at (14, 5) {};
		\node [style=none] (59) at (14.75, 2.25) {};
		\node [style=none] (62) at (14, 0.5) {};
		\node [style=none] (63) at (16.25, 2.25) {};
		\node [style=none] (64) at (15.5, 1.5) {};
		\node [style=none] (65) at (17, 0) {};
		\node [style=none] (66) at (17, 5) {};
		\node [style=none] (67) at (16.25, 5) {};
		\node [style=none] (68) at (14.75, 5) {};
	\end{pgfonlayer}
	\begin{pgfonlayer}{edgelayer}
		\draw [in=0, out=-90, looseness=1.25] (3.center) to (2.center);
		\draw [in=-180, out=-90, looseness=1.25] (1.center) to (2.center);
		\draw [in=0, out=-90, looseness=1.25] (8.center) to (7.center);
		\draw [in=-180, out=-90, looseness=1.25] (6.center) to (7.center);
		\draw [in=0, out=-90, looseness=1.25] (11.center) to (10.center);
		\draw [in=-180, out=-90, looseness=1.25] (9.center) to (10.center);
		\draw [style=bieski, in=0, out=-90, looseness=0.75] (10.center) to (8.center);
		\draw (14.center) to (8.center);
		\draw (15.center) to (6.center);
		\draw (12.center) to (1.center);
		\draw (13.center) to (3.center);
		\draw [in=0, out=90] (17.center) to (18.center);
		\draw [in=0, out=90] (21.center) to (22.center);
		\draw (15.center) to (24.center);
		\draw (13.center) to (23.center);
		\draw [style=thickstrand] (5.center) to (25.center);
		\draw (27.center) to (26.center);
		\draw [in=0, out=-90, looseness=1.25] (31.center) to (30.center);
		\draw [in=-180, out=-90, looseness=1.25] (29.center) to (30.center);
		\draw [in=0, out=-90, looseness=1.25] (35.center) to (34.center);
		\draw [in=-180, out=-90, looseness=1.25] (33.center) to (34.center);
		\draw (41.center) to (35.center);
		\draw (42.center) to (33.center);
		\draw (40.center) to (31.center);
		\draw [in=0, out=90] (44.center) to (45.center);
		\draw [in=0, out=90] (48.center) to (49.center);
		\draw (42.center) to (51.center);
		\draw (40.center) to (50.center);
		\draw [style=thickstrand] (32.center) to (52.center);
		\draw (54.center) to (53.center);
		\draw [in=180, out=90] (14.center) to (18.center);
		\draw [in=180, out=90] (19.center) to (22.center);
		\draw [in=180, out=90] (41.center) to (45.center);
		\draw [in=180, out=90] (46.center) to (49.center);
		\draw (54.center) to (55.center);
		\draw [style=bieski, in=0, out=-90, looseness=0.50] (37.center) to (54.center);
		\draw [in=180, out=-90] (44.center) to (37.center);
		\draw [in=0, out=-90] (46.center) to (37.center);
		\draw (48.center) to (29.center);
		\draw (57.center) to (58.center);
		\draw [in=0, out=-90, looseness=1.25] (63.center) to (64.center);
		\draw [in=-180, out=-90, looseness=1.25] (59.center) to (64.center);
		\draw [style=bieski, in=0, out=-90, looseness=1.25] (64.center) to (62.center);
		\draw [style=thickstrand] (66.center) to (65.center);
		\draw (68.center) to (59.center);
		\draw (67.center) to (63.center);
	\end{pgfonlayer}
\end{tikzpicture}
\]}
The first equality follows from naturality of $\widehat{\eta}$, and the second follows by straightening two zigzags, using the triangle identities for $(\mathbf{M}\mathrm{F},\mathbf{M}\mathrm{F}^{\blackdiamond},\widehat{\eta},\widehat{\varepsilon})$.
Equation~\ref{CoherentlyRespectful} for $\varepsilon$ follows similarly.
 We now show that $\eta, \varepsilon$ satisfy the triangle identities for an adjunction:
{\setlength{\abovedisplayskip}{3pt}
\setlength{\belowdisplayskip}{3pt}\[
 \scalebox{.787}{$
\begin{aligned}
 &
 \begin{tikzpicture}[scale=0.5]
	\begin{pgfonlayer}{nodelayer}
		\node [style=none] (32) at (-2.75, 9) {};
		\node [style=none] (35) at (-3.25, 10.5) {};
		\node [style=none] (36) at (-5.75, 9) {};
		\node [style=none] (37) at (-6.75, 9) {};
		\node [style=none] (38) at (-6.25, 10.5) {};
		\node [style=none] (41) at (-4.75, 8.5) {};
		\node [style=none] (43) at (-5.75, 7) {};
		\node [style=none] (44) at (-3.75, 9) {};
		\node [style=none] (46) at (-5.5, 6) {};
		\node [style=none] (47) at (-5.25, 7) {};
		\node [style=none] (48) at (-4.25, 7) {};
		\node [style=none] (49) at (-4, 6) {};
		\node [style=none] (50) at (-3.75, 7) {};
		\node [style=none] (53) at (1.25, 2.5) {};
		\node [style=none] (54) at (0.75, 1) {};
		\node [style=none] (55) at (-1.75, 2.5) {};
		\node [style=none] (56) at (-2.75, 2.5) {};
		\node [style=none] (57) at (-2.25, 1) {};
		\node [style=none] (58) at (-0.75, 3) {};
		\node [style=none] (59) at (-1.75, 4.5) {};
		\node [style=none] (60) at (0.25, 2.5) {};
		\node [style=none] (61) at (-1.5, 5.5) {};
		\node [style=none] (62) at (-1.25, 4.5) {};
		\node [style=none] (63) at (-0.25, 4.5) {};
		\node [style=none] (64) at (0, 5.5) {};
		\node [style=none] (65) at (0.25, 4.5) {};
		\node [style=none] (66) at (1.25, 7.5) {};
		\node [style=none] (68) at (1.25, 11) {};
		\node [style=none] (69) at (2, 0) {};
		\node [style=none] (70) at (2, 11) {};
		\node [style=none] (71) at (3, 6) {$\stackrel{(1)}{=}$};
		\node [style=none] (73) at (8, 9) {};
		\node [style=none] (74) at (7.5, 10.5) {};
		\node [style=none] (75) at (5, 9) {};
		\node [style=none] (76) at (4, 9) {};
		\node [style=none] (77) at (4.5, 10.5) {};
		\node [style=none] (78) at (6, 8.5) {};
		\node [style=none] (79) at (5, 7) {};
		\node [style=none] (80) at (7, 9) {};
		\node [style=none] (81) at (5.25, 6) {};
		\node [style=none] (82) at (5.5, 7) {};
		\node [style=none] (83) at (6.5, 7) {};
		\node [style=none] (84) at (6.75, 6) {};
		\node [style=none] (85) at (7, 7) {};
		\node [style=none] (87) at (12, 2.5) {};
		\node [style=none] (88) at (11.5, 1) {};
		\node [style=none] (89) at (9, 2.5) {};
		\node [style=none] (90) at (8, 2.5) {};
		\node [style=none] (91) at (8.5, 1) {};
		\node [style=none] (92) at (10, 3) {};
		\node [style=none] (93) at (9, 4.5) {};
		\node [style=none] (94) at (11, 2.5) {};
		\node [style=none] (95) at (9.25, 5.5) {};
		\node [style=none] (96) at (9.5, 4.5) {};
		\node [style=none] (97) at (10.5, 4.5) {};
		\node [style=none] (98) at (10.75, 5.5) {};
		\node [style=none] (99) at (11, 4.5) {};
		\node [style=none] (100) at (12, 7.5) {};
		\node [style=none] (102) at (12, 11) {};
		\node [style=none] (103) at (12.75, 0) {};
		\node [style=none] (104) at (12.75, 11) {};
		\node [style=none] (105) at (15, 0) {};
		\node [style=none] (106) at (16, 5) {};
		\node [style=none] (107) at (15, 5) {};
		\node [style=none] (108) at (15.5, 6.5) {};
		\node [style=none] (109) at (17, 4.5) {};
		\node [style=none] (110) at (16, 3) {};
		\node [style=none] (111) at (18, 5) {};
		\node [style=none] (112) at (16.25, 2) {};
		\node [style=none] (113) at (16.5, 3) {};
		\node [style=none] (114) at (17.5, 3) {};
		\node [style=none] (115) at (17.75, 2) {};
		\node [style=none] (116) at (18, 3) {};
		\node [style=none] (117) at (21, 5) {};
		\node [style=none] (118) at (20.5, 3.5) {};
		\node [style=none] (119) at (18, 5) {};
		\node [style=none] (120) at (19, 5.5) {};
		\node [style=none] (121) at (18, 7) {};
		\node [style=none] (122) at (20, 5) {};
		\node [style=none] (123) at (18.25, 8) {};
		\node [style=none] (124) at (18.5, 7) {};
		\node [style=none] (125) at (19.5, 7) {};
		\node [style=none] (126) at (19.75, 8) {};
		\node [style=none] (127) at (20, 7) {};
		\node [style=none] (130) at (21.75, 0.25) {};
		\node [style=none] (131) at (21.75, 11) {};
		\node [style=none] (132) at (21, 11) {};
		\node [style=none] (133) at (14, 6) {$\stackrel{(2)}{=}$};
		\node [style=none] (134) at (-6.75, 0) {};
		\node [style=none] (135) at (4, 0) {};
	\end{pgfonlayer}
	\begin{pgfonlayer}{edgelayer}
		\draw [in=0, out=90] (32.center) to (35.center);
		\draw [in=-180, out=90] (37.center) to (38.center);
		\draw [in=0, out=90] (36.center) to (38.center);
		\draw [style=bieski, in=0, out=90, looseness=1.25] (41.center) to (36.center);
		\draw (36.center) to (43.center);
		\draw [in=-180, out=90] (44.center) to (35.center);
		\draw [in=0, out=-90, looseness=0.75] (47.center) to (46.center);
		\draw [in=-180, out=90] (47.center) to (41.center);
		\draw [in=-180, out=-90, looseness=0.75] (43.center) to (46.center);
		\draw [in=0, out=-90, looseness=0.75] (50.center) to (49.center);
		\draw [in=-180, out=-90, looseness=0.75] (48.center) to (49.center);
		\draw [in=0, out=90] (48.center) to (41.center);
		\draw (44.center) to (50.center);
		\draw [in=0, out=-90] (53.center) to (54.center);
		\draw [in=180, out=-90] (56.center) to (57.center);
		\draw [in=0, out=-90] (55.center) to (57.center);
		\draw [style=bieski, in=0, out=-90, looseness=1.25] (58.center) to (55.center);
		\draw (55.center) to (59.center);
		\draw [in=180, out=-90] (60.center) to (54.center);
		\draw [in=0, out=90, looseness=0.75] (62.center) to (61.center);
		\draw [in=180, out=-90] (62.center) to (58.center);
		\draw [in=180, out=90, looseness=0.75] (59.center) to (61.center);
		\draw [in=0, out=90, looseness=0.75] (65.center) to (64.center);
		\draw [in=180, out=90, looseness=0.75] (63.center) to (64.center);
		\draw [in=0, out=-90] (63.center) to (58.center);
		\draw (60.center) to (65.center);
		\draw (53.center) to (66.center);
		\draw (68.center) to (66.center);
		\draw [style=thickstrand] (70.center) to (69.center);
		\draw [in=0, out=90] (73.center) to (74.center);
		\draw [in=-180, out=90] (76.center) to (77.center);
		\draw [in=0, out=90] (75.center) to (77.center);
		\draw [style=bieski, in=0, out=90, looseness=1.25] (78.center) to (75.center);
		\draw (75.center) to (79.center);
		\draw [in=-180, out=90] (80.center) to (74.center);
		\draw [in=0, out=-90, looseness=0.75] (82.center) to (81.center);
		\draw [in=-180, out=90] (82.center) to (78.center);
		\draw [in=-180, out=-90, looseness=0.75] (79.center) to (81.center);
		\draw [in=0, out=-90, looseness=0.75] (85.center) to (84.center);
		\draw [in=-180, out=-90, looseness=0.75] (83.center) to (84.center);
		\draw [in=0, out=90] (83.center) to (78.center);
		\draw (80.center) to (85.center);
		\draw [in=0, out=-90] (87.center) to (88.center);
		\draw [in=180, out=-90] (90.center) to (91.center);
		\draw [in=0, out=-90] (89.center) to (91.center);
		\draw (89.center) to (93.center);
		\draw [in=180, out=-90] (94.center) to (88.center);
		\draw [in=0, out=90, looseness=0.75] (96.center) to (95.center);
		\draw [in=180, out=-90] (96.center) to (92.center);
		\draw [in=180, out=90, looseness=0.75] (93.center) to (95.center);
		\draw [in=0, out=90, looseness=0.75] (99.center) to (98.center);
		\draw [in=180, out=90, looseness=0.75] (97.center) to (98.center);
		\draw [in=0, out=-90] (97.center) to (92.center);
		\draw (94.center) to (99.center);
		\draw (87.center) to (100.center);
		\draw (102.center) to (100.center);
		\draw [style=thickstrand] (104.center) to (103.center);
		\draw [style=bieski, in=-180, out=-90, looseness=1.25] (92.center) to (94.center);
		\draw [in=-180, out=90] (107.center) to (108.center);
		\draw [in=0, out=90] (106.center) to (108.center);
		\draw (105.center) to (107.center);
		\draw [style=bieski, in=0, out=90, looseness=1.25] (109.center) to (106.center);
		\draw (106.center) to (110.center);
		\draw [in=0, out=-90, looseness=0.75] (113.center) to (112.center);
		\draw [in=-180, out=90] (113.center) to (109.center);
		\draw [in=-180, out=-90, looseness=0.75] (110.center) to (112.center);
		\draw [in=0, out=-90, looseness=0.75] (116.center) to (115.center);
		\draw [in=-180, out=-90, looseness=0.75] (114.center) to (115.center);
		\draw [in=0, out=90] (114.center) to (109.center);
		\draw (111.center) to (116.center);
		\draw [in=0, out=-90] (117.center) to (118.center);
		\draw (119.center) to (121.center);
		\draw [in=180, out=-90] (122.center) to (118.center);
		\draw [in=0, out=90, looseness=0.75] (124.center) to (123.center);
		\draw [in=180, out=-90] (124.center) to (120.center);
		\draw [in=180, out=90, looseness=0.75] (121.center) to (123.center);
		\draw [in=0, out=90, looseness=0.75] (127.center) to (126.center);
		\draw [in=180, out=90, looseness=0.75] (125.center) to (126.center);
		\draw [in=0, out=-90] (125.center) to (120.center);
		\draw (122.center) to (127.center);
		\draw [style=thickstrand] (131.center) to (130.center);
		\draw [style=bieski, in=-180, out=-90, looseness=1.25] (120.center) to (122.center);
		\draw (117.center) to (132.center);
		\draw (56.center) to (32.center);
		\draw (37.center) to (134.center);
		\draw (76.center) to (135.center);
		\draw (73.center) to (90.center);
	\end{pgfonlayer}
\end{tikzpicture}
 \\
 &
\begin{tikzpicture}[scale=0.5]
	\begin{pgfonlayer}{nodelayer}
		\node [style=none] (106) at (1.75, 9) {};
		\node [style=none] (107) at (0.75, 9) {};
		\node [style=none] (108) at (1.25, 10.5) {};
		\node [style=none] (109) at (2.75, 8.5) {};
		\node [style=none] (110) at (1.75, 7) {};
		\node [style=none] (112) at (2, 6) {};
		\node [style=none] (113) at (2.25, 7) {};
		\node [style=none] (114) at (3.25, 7) {};
		\node [style=none] (117) at (5.75, 5) {};
		\node [style=none] (118) at (5.25, 3.5) {};
		\node [style=none] (120) at (3.75, 5.5) {};
		\node [style=none] (122) at (4.75, 5) {};
		\node [style=none] (124) at (3.25, 7) {};
		\node [style=none] (125) at (4.25, 7) {};
		\node [style=none] (126) at (4.5, 8) {};
		\node [style=none] (127) at (4.75, 7) {};
		\node [style=none] (130) at (6.5, 0) {};
		\node [style=none] (131) at (6.5, 11) {};
		\node [style=none] (132) at (5.75, 11) {};
		\node [style=none] (133) at (0, 6) {$\stackrel{(3)}{=}$};
		\node [style=none] (134) at (0.75, 0) {};
		\node [style=none] (135) at (9, 9) {};
		\node [style=none] (136) at (8, 9) {};
		\node [style=none] (137) at (8.5, 10.5) {};
		\node [style=none] (138) at (10, 8) {};
		\node [style=none] (139) at (9, 7) {};
		\node [style=none] (140) at (9.25, 4.75) {};
		\node [style=none] (141) at (9.5, 7) {};
		\node [style=none] (142) at (10.5, 7) {};
		\node [style=none] (143) at (13, 5) {};
		\node [style=none] (144) at (12.5, 3.5) {};
		\node [style=none] (145) at (11, 6) {};
		\node [style=none] (146) at (12, 5) {};
		\node [style=none] (147) at (10.5, 7) {};
		\node [style=none] (148) at (11.5, 7) {};
		\node [style=none] (150) at (12, 7) {};
		\node [style=none] (151) at (13.75, 0) {};
		\node [style=none] (152) at (13.75, 11) {};
		\node [style=none] (153) at (13, 6) {};
		\node [style=none] (154) at (7.25, 6) {$\stackrel{(4)}{=}$};
		\node [style=none] (155) at (8, 0) {};
		\node [style=none] (156) at (13, 11) {};
		\node [style=none] (157) at (9.5, 5.25) {};
		\node [style=none] (158) at (9, 5.25) {};
		\node [style=none] (159) at (12, 9) {};
		\node [style=none] (160) at (11.5, 9) {};
		\node [style=none] (161) at (11.75, 9.5) {};
		\node [style=none] (162) at (11.5, 9) {};
		\node [style=none] (163) at (16, 9) {};
		\node [style=none] (164) at (15, 9) {};
		\node [style=none] (165) at (15.5, 10.5) {};
		\node [style=none] (167) at (16, 7) {};
		\node [style=none] (168) at (16.25, 4.75) {};
		\node [style=none] (171) at (18, 5.25) {};
		\node [style=none] (172) at (17.5, 3.75) {};
		\node [style=none] (178) at (18.75, 0) {};
		\node [style=none] (179) at (18.75, 11) {};
		\node [style=none] (180) at (18, 6.25) {};
		\node [style=none] (181) at (14.25, 6) {$\stackrel{(5)}{=}$};
		\node [style=none] (182) at (15, 0) {};
		\node [style=none] (184) at (16.5, 5.25) {};
		\node [style=none] (185) at (16, 5.25) {};
		\node [style=none] (186) at (17, 5.25) {};
		\node [style=none] (187) at (16.5, 5.25) {};
		\node [style=none] (188) at (16.75, 5.75) {};
		\node [style=none] (189) at (16.5, 5.25) {};
		\node [style=none] (190) at (18, 11) {};
		\node [style=none] (192) at (20, 11) {};
		\node [style=none] (198) at (20.75, 0) {};
		\node [style=none] (199) at (20.75, 11) {};
		\node [style=none] (201) at (19.25, 6) {$\stackrel{(6)}{=}$};
		\node [style=none] (202) at (20, 0) {};
	\end{pgfonlayer}
	\begin{pgfonlayer}{edgelayer}
		\draw [in=-180, out=90] (107.center) to (108.center);
		\draw [in=0, out=90] (106.center) to (108.center);
		\draw [style=bieski, in=0, out=90, looseness=1.25] (109.center) to (106.center);
		\draw (106.center) to (110.center);
		\draw [in=0, out=-90, looseness=0.75] (113.center) to (112.center);
		\draw [in=-180, out=90] (113.center) to (109.center);
		\draw [in=-180, out=-90, looseness=0.75] (110.center) to (112.center);
		\draw [in=0, out=90] (114.center) to (109.center);
		\draw [in=0, out=-90] (117.center) to (118.center);
		\draw [in=180, out=-90] (122.center) to (118.center);
		\draw [in=180, out=-90] (124.center) to (120.center);
		\draw [in=0, out=90, looseness=0.75] (127.center) to (126.center);
		\draw [in=180, out=90, looseness=0.75] (125.center) to (126.center);
		\draw [in=0, out=-90] (125.center) to (120.center);
		\draw (122.center) to (127.center);
		\draw [style=thickstrand] (131.center) to (130.center);
		\draw [style=bieski, in=-180, out=-90, looseness=1.25] (120.center) to (122.center);
		\draw (117.center) to (132.center);
		\draw (134.center) to (107.center);
		\draw [in=-180, out=90] (136.center) to (137.center);
		\draw [in=0, out=90] (135.center) to (137.center);
		\draw (135.center) to (139.center);
		\draw [in=-180, out=90] (141.center) to (138.center);
		\draw [in=0, out=90] (142.center) to (138.center);
		\draw [in=0, out=-90] (143.center) to (144.center);
		\draw [in=180, out=-90] (146.center) to (144.center);
		\draw [in=180, out=-90] (147.center) to (145.center);
		\draw [in=0, out=-90] (148.center) to (145.center);
		\draw (146.center) to (150.center);
		\draw [style=thickstrand] (152.center) to (151.center);
		\draw (143.center) to (153.center);
		\draw (155.center) to (136.center);
		\draw (153.center) to (156.center);
		\draw (139.center) to (158.center);
		\draw (141.center) to (157.center);
		\draw [in=0, out=-90, looseness=1.25] (157.center) to (140.center);
		\draw [in=-90, out=-180, looseness=1.25] (140.center) to (158.center);
		\draw [style=bieski, in=0, out=-90, looseness=0.75] (145.center) to (157.center);
		\draw (148.center) to (160.center);
		\draw (150.center) to (159.center);
		\draw [in=180, out=90, looseness=1.25] (162.center) to (161.center);
		\draw [in=0, out=90, looseness=1.25] (159.center) to (161.center);
		\draw [style=bieski, in=180, out=90] (138.center) to (162.center);
		\draw [in=-180, out=90] (164.center) to (165.center);
		\draw [in=0, out=90] (163.center) to (165.center);
		\draw (163.center) to (167.center);
		\draw [in=0, out=-90] (171.center) to (172.center);
		\draw [style=thickstrand] (179.center) to (178.center);
		\draw (171.center) to (180.center);
		\draw (182.center) to (164.center);
		\draw (167.center) to (185.center);
		\draw [in=0, out=-90, looseness=1.25] (184.center) to (168.center);
		\draw [in=-90, out=-180, looseness=1.25] (168.center) to (185.center);
		\draw [in=180, out=90, looseness=1.25] (189.center) to (188.center);
		\draw [in=0, out=90, looseness=1.25] (186.center) to (188.center);
		\draw [in=180, out=-90] (186.center) to (172.center);
		\draw (180.center) to (190.center);
		\draw [style=thickstrand] (199.center) to (198.center);
		\draw (202.center) to (192.center);
	\end{pgfonlayer}
\end{tikzpicture}
\end{aligned}$}
\]}
The equalities above follow by:
\begin{enumerate}
 \item Axiom~\ref{Involution};
 \item Triangle identities for $\widehat{\eta}, \widehat{\varepsilon}$;
 \item Triangle identities for $\widehat{\eta}, \widehat{\varepsilon}$;
 \item Isotopy and Axiom~\ref{Involution};
 \item Axiom~\ref{ZigZag};
 \item Triangle identities for $\widehat{\eta}, \widehat{\varepsilon}$, twice.
\end{enumerate}

 \end{proof}

Finally, we need the latter condition of Theorem~\ref{LiberalUnitBar} to carry over to $\mathbf{M}$:
\begin{notation}
 For a semigroup $\csym{S}$-module category $\mathbf{M}$ over a semigroup category $\csym{S}$, we write $\csym{S} \star \mathbf{M} = \mathbf{M}$ if every object $X$ of $\mathbf{M}$ is a direct summand of a finite direct sum of objects of the form $\mathbf{M}\mathrm{F}(Y)$, for $Y \in \mathbf{M}$ and $\mathrm{F} \in \csym{S}$.
\end{notation}

The next theorem concerns module categories which, as a category, are of the form $[\mathcal{A}^{\on{op}},\mathbf{Vec}_{\Bbbk}]$, i.e. the category of presheaves over a small category $\mathcal{A}$. In Theorem~\ref{AbelianizationCorrespondence}, we refer to module categories of this form, where additionally the action functor is cocontinuous in both variables, as {\it module presheaf categories}.

\begin{theorem}\label{AbelianizationCorrespondence}
 Let $\csym{S}$ be a rigid semigroup category with an object $\mathrm{F}$ such that $\mathrm{F} \boxtimes -$ is liberal.
 If $\mathbf{M}$ is a semigroup $\csym{S}$-module category such that $\csym{S}\star \mathbf{M} = \mathbf{M}$, and such that the adjunction $(\mathrm{F},\mathrm{F}^{\blackdiamond}, \eta^{\mathrm{F},l},\eta^{\mathrm{F},r}, \varepsilon^{\mathrm{F},l},\varepsilon^{\mathrm{F},r})$ is respected in $\mathbf{M}$, then $[\mathbf{M}^{\on{op}},\mathbf{Vec}_{\Bbbk}]$ is a monoidal $[\csym{S}^{\on{op}},\mathbf{Vec}_{\Bbbk}]$-module category. Furthermore, there is a bijection
  \begin{equation}\label{Bijections}
  \begin{aligned}
 \left\{\rule{0cm}{1cm}\right.
  \begin{aligned}
  &\text{Cauchy complete semigroup }\csym{S}\text{-module categories} \\
  &\text{respecting }
  (\mathrm{F},\mathrm{F}^{\blackdiamond}, \eta^{\mathrm{F},l},\eta^{\mathrm{F},r}, \varepsilon^{\mathrm{F},l},\varepsilon^{\mathrm{F},r}) \text{ such that} \\
  &\csym{S}\star \mathbf{M} = \mathbf{M}
  \end{aligned}
\left.\rule{0cm}{1cm}\right\}/&\simeq
 \longleftrightarrow
  \Bigg\{
  \begin{aligned}
  &\text{Monoidal }[\csym{S}^{\on{op}},\mathbf{Vec}_{\Bbbk}]\text{-module presheaf} \\
  &\text{categories where }\Hom{-,\mathrm{F}^{\blackdiamond}} \text{ acts faithfully}
  \end{aligned}
\Bigg\}/\simeq
 \\
 &\mathbf{M} \longmapsto  [\mathbf{M}^{\on{op}},\mathbf{Vec}_{\Bbbk}] \\
 &\mathbf{N}\!\on{-f.g.proj} \longmapsfrom \mathbf{N}
 \end{aligned}
  \end{equation}
  Where the $[\csym{S}^{\on{op}},\mathbf{Vec}_{\Bbbk}]$-module category structure on $[\mathbf{M}^{\on{op}},\mathbf{Vec}_{\Bbbk}]$ is obtained via cocontinuous extension described following Terminology~\ref{dispel}, and the $\csym{S}$-module structure on $\mathbf{N}\!\on{-f.g.proj}$ is by restriction.
\end{theorem}

\begin{proof}
 If $\csym{S} \star \mathbf{M} = \mathbf{M}$, then $\mathbf{M}\mathrm{F}$ is liberal. Indeed, for any $X \in \mathbf{M}$, there are by assumption some $\mathrm{G} \in \csym{S}$ and $Y \in \mathbf{M}$ such that $X$ is a direct summand of $\mathbf{M}\mathrm{G}(Y)$. Since $\mathrm{F} \boxtimes -$ is liberal, the object $\mathrm{G}$ is a direct summand of $\mathrm{F}\boxtimes \mathrm{H}$, for some $\mathrm{H} \in \csym{S}$. Thus, $\mathbf{M}\mathrm{G}(Y)$, and by transitivity also $X$, is a direct summand of $\mathbf{M}(\mathrm{F} \boxtimes \mathrm{H})(Y) \simeq \mathbf{M}\mathrm{F}(\mathbf{M}(\mathrm{H})(Y))$.

 Since the adjunction $(\mathrm{F},\mathrm{F}^{\blackdiamond}, \eta^{\mathrm{F},l},\eta^{\mathrm{F},r}, \varepsilon^{\mathrm{F},l},\varepsilon^{\mathrm{F},r})$ is respected in $\mathbf{M}$, we have an adjoint pair $(\mathbf{M}\mathrm{F}, \mathbf{M}\mathrm{F}^{\blackdiamond})$ of endofunctors of $\mathbf{M}$. Since $\mathbf{M}\mathrm{F}$ is liberal, the right adjoint in the adjoint pair $([\mathbf{M}^{\on{op}},\mathbf{Vec}_{\Bbbk}]\Hom{-,\mathrm{F}},[\mathbf{M}^{\on{op}},\mathbf{Vec}_{\Bbbk}]\Hom{-,\mathrm{F}^{\blackdiamond}})$ is exact and faithful, and thus the pair is resolvent, so the coequalizer of
\begin{equation}\label{IsoToId}
\begin{tikzcd}
	{([\mathbf{M}^{\on{op}},\mathbf{Vec}_{\Bbbk}]\Hom{-,\mathrm{F}})\circ ([\mathbf{M}^{\on{op}},\mathbf{Vec}_{\Bbbk}]\Hom{-,\mathrm{F}^{\blackdiamond}})\circ ([\mathbf{M}^{\on{op}},\mathbf{Vec}_{\Bbbk}]\Hom{-,\mathrm{F}})\circ ([\mathbf{M}^{\on{op}},\mathbf{Vec}_{\Bbbk}]\Hom{-,\mathrm{F}^{\blackdiamond}})} \\
	{([\mathbf{M}^{\on{op}},\mathbf{Vec}_{\Bbbk}]\Hom{-,\mathrm{F}})\circ ([\mathbf{M}^{\on{op}},\mathbf{Vec}_{\Bbbk}]\Hom{-,\mathrm{F}^{\blackdiamond}})}
	\arrow["{(\varepsilon^{\mathrm{F},\mathbf{M}})_{!} \circ ([\mathbf{M}^{\on{op}},\mathbf{Vec}_{\Bbbk}]\Hom{-,\mathrm{F}})\circ ([\mathbf{M}^{\on{op}},\mathbf{Vec}_{\Bbbk}]\Hom{-,\mathrm{F}^{\blackdiamond}})}"', shift right=4, from=1-1, to=2-1]
	\arrow["{([\mathbf{M}^{\on{op}},\mathbf{Vec}_{\Bbbk}]\Hom{-,\mathrm{F}})\circ ([\mathbf{M}^{\on{op}},\mathbf{Vec}_{\Bbbk}]\Hom{-,\mathrm{F}^{\blackdiamond}}) \circ (\varepsilon^{\mathrm{F},\mathbf{M}})_{!}}", shift left=4, from=1-1, to=2-1]
\end{tikzcd}
\end{equation}
is isomorphic to the identity functor $\on{Id}_{[\mathbf{M}^{\on{op}},\mathbf{Vec}_{\Bbbk}]}$. Combining Definition~\ref{Respect} with the use of coherence isomorphisms
\[
[\mathbf{M}^{\on{op}},\mathbf{Vec}_{\Bbbk}](\mathrm{Q}\oast \mathrm{P}) \simeq [\mathbf{M}^{\on{op}},\mathbf{Vec}_{\Bbbk}](\mathrm{Q}) \circ [\mathbf{M}^{\on{op}},\mathbf{Vec}_{\Bbbk}](\mathrm{P}),
\]
one easily shows that Diagram~\eqref{IsoToId} is isomorphic to
\[\begin{tikzcd}
	{[\mathbf{M}^{\on{op}},\mathbf{Vec}_{\Bbbk}]\Hom{-,\mathrm{FF^{\blackdiamond}FF^{\blackdiamond}}}} &&& {[\mathbf{M}^{\on{op}},\mathbf{Vec}_{\Bbbk}]\Hom{-,\mathrm{FF^{\blackdiamond}}}}
	\arrow["{[\mathbf{M}^{\on{op}},\mathbf{Vec}_{\Bbbk}]{\Hom{-,\varepsilon^{\mathrm{F},r}_{\mathrm{F}}\mathrm{F}^{\blackdiamond}}}}", shift left=3, from=1-1, to=1-4]
	\arrow["{[\mathbf{M}^{\on{op}},\mathbf{Vec}_{\Bbbk}]{\Hom{-,\mathrm{F}\varepsilon^{\mathrm{F},l}_{\mathrm{F}^{\blackdiamond}}}}}"', shift right=3, from=1-1, to=1-4]
\end{tikzcd}\]
From the cocontinuity of the action, we obtain the second isomorphism in
\[
\begin{aligned}
 \on{Id}_{[\mathbf{M}^{\on{op}},\mathbf{Vec}_{\Bbbk}]} &\simeq \on{Coeq}
 \Bigg(
 \begin{tikzcd}
	{[\mathbf{M}^{\on{op}},\mathbf{Vec}_{\Bbbk}]\Hom{-,\mathrm{FF^{\blackdiamond}FF^{\blackdiamond}}}} &&& {[\mathbf{M}^{\on{op}},\mathbf{Vec}_{\Bbbk}]\Hom{-,\mathrm{FF^{\blackdiamond}}}}
	\arrow["{[\mathbf{M}^{\on{op}},\mathbf{Vec}_{\Bbbk}]{\Hom{-,\varepsilon^{\mathrm{F},r}_{\mathrm{F}}\mathrm{F}^{\blackdiamond}}}}", shift left=3, from=1-1, to=1-4]
	\arrow["{[\mathbf{M}^{\on{op}},\mathbf{Vec}_{\Bbbk}]{\Hom{-,\mathrm{F}\varepsilon^{\mathrm{F},l}_{\mathrm{F}^{\blackdiamond}}}}}"', shift right=3, from=1-1, to=1-4]
\end{tikzcd}
 \Bigg)\\
 &\simeq [\mathbf{M}^{\on{op}},\mathbf{Vec}_{\Bbbk}]\Bigg(\on{Coeq}\Bigg(
 \begin{tikzcd}
	{\Hom{-,\mathrm{FF^{\blackdiamond}FF^{\blackdiamond}}}} &&& {\Hom{-,\mathrm{FF^{\blackdiamond}}}}
	\arrow["{{\Hom{-,\varepsilon^{\mathrm{F},r}_{\mathrm{F}}\mathrm{F}^{\blackdiamond}}}}", shift left=3, from=1-1, to=1-4]
	\arrow["{{\Hom{-,\mathrm{F}\varepsilon^{\mathrm{F},l}_{\mathrm{F}^{\blackdiamond}}}}}"', shift right=3, from=1-1, to=1-4]
\end{tikzcd}
\Bigg)\Bigg)
\end{aligned}
\]
By Theorem~\ref{LiberalUnitBar}, the last term is isomorphic to $[\mathbf{M}^{\on{op}},\mathbf{Vec}_{\Bbbk}]\left(\mathbb{1}_{[\ccf{S}^{\on{op}},\mathbf{Vec}_{\Bbbk}]}\right)$. Thus, we have established that
\[
\on{Id}_{[\mathbf{M}^{\on{op}},\mathbf{Vec}_{\Bbbk}]} \simeq [\mathbf{M}^{\on{op}},\mathbf{Vec}_{\Bbbk}]\left(\mathbb{1}_{[\ccf{S}^{\on{op}},\mathbf{Vec}_{\Bbbk}]}\right)
\]
and, by Proposition~\ref{AnotherEGNOLemma}, $[\mathbf{M}^{\on{op}},\mathbf{Vec}_{\Bbbk}]$ is a monoidal $[\csym{S}^{\on{op}},\mathbf{Vec}_{\Bbbk}]$-module category.

To establish the bijective correspondence \eqref{Bijections}, denote the indicated map from the left-hand side to the right-hand side by $\Omega$, and that from the right-hand side to the left-hand side by $\Psi$.

To see that the map $\Psi$ is well-defined, recall that a monoidal module category respects all adjunctions, and so for $\mathbf{N}$ in the right-hand side, the functor $\mathbf{N}\mathrm{F}$ has a left and a right adjoint, and thus sends finitely generated projective objects to finitely generated projective objects. Thus, $\mathbf{N}\!\on{-f.g.proj}$ is a semigroup module $\csym{S}$-subcategory of $\mathbf{N}$. Further, this shows that the adjunction $(\mathrm{F},\mathrm{F}^{\blackdiamond}, \eta^{\mathrm{F},l},\eta^{\mathrm{F},r}, \varepsilon^{\mathrm{F},l},\varepsilon^{\mathrm{F},r})$ is respected in $\mathbf{N}\!\on{-f.g.proj}$.

Since $\mathbf{N}\Hom{-,\mathrm{F}^{\blackdiamond}}$ is faithful, the transformation $\varepsilon^{\mathrm{F},\mathbf{N}}$ is an epimorphism, so, for any $P \in \mathbf{N}\!\on{-f.g.proj}$, the morphism
\[
 (\mathbf{N}\Hom{-,\mathrm{FF^{\blackdiamond}}})(P) \xrightarrow{\varepsilon^{\mathrm{F},\mathbf{N}}} P
\]
is epi. Since $P$ is projective, this morphism is a split epimorphism, which shows that $\csym{S}\star \mathbf{N}\!\on{-f.g.proj} = \mathbf{N}\!\on{-f.g.proj}$, and so $\Psi$ is well-defined.

By the definition of the semigroup $[\csym{S}^{\on{op}},\mathbf{Vec}_{\Bbbk}]$-module structure on $[\mathbf{M}^{\on{op}},\mathbf{Vec}_{\Bbbk}]$, the equivalence
\[
[\mathbf{M}^{\on{op}},\mathbf{Vec}_{\Bbbk}]\!\on{-f.g.proj} \simeq \mathbf{M}
\]
is an equivalence of semigroup $\csym{S}$-module categories. This shows that $\Psi\circ \Omega = \on{Id}$. Further, since cocontinuous extensions of functors to categories of presheaves are essentially unique, and a cocontinuous functor from a presheaf category is determined on finitely generated projectives, the equivalence
\[
 [\mathbf{N}\!\on{-f.g.proj},\mathbf{Vec}_{\Bbbk}] \simeq \mathbf{N}
\]
is an equivalence of monoidal $[\csym{S}^{\on{op}},\mathbf{Vec}_{\Bbbk}]$-module categories, establishing the claim.
\end{proof}

\begin{definition}
 Given a rigid semigroup category $\csym{S}$, let $\widetilde{\csym{S}}$ be the monoidal subcategory of $[\csym{S}^{\on{op}},\mathbf{Vec}_{\Bbbk}]$, where $\on{Ob}\widetilde{\csym{S}} = \setj{\Hom{-,\mathrm{F}} \; | \; \mathrm{F} \in \csym{S}} \cup \setj{\mathbb{1}_{\widetilde{\ccf{S}}}}$, where $\mathbb{1}_{\widetilde{\ccf{S}}}$ is the coequalizer in the right-hand side of \eqref{UnitResolution}.
\end{definition}

\begin{remark}
 If $\csym{S}$ already is monoidal, then, since $\varepsilon^{\mathrm{F},l}$ is epi, we have $\mathbb{1}_{\widetilde{\ccf{S}}} \simeq \mathbb{1}_{\ccf{S}}$ - the operation of adjoining a unit object which sends $\csym{S}$ to $\widetilde{\csym{S}}$ is weakly idempotent.
\end{remark}

We may now reformulate Theorem~\ref{AbelianizationCorrespondence} in terms of a lifting problem, parallel to \cite[Theorem~3.16]{KMZ}:

\begin{corollary}
 There is a bijection
   \begin{equation}
  \begin{aligned}
 \left\{\rule{0cm}{1cm}\right.
  \begin{aligned}
  &\text{Cauchy complete semigroup }\csym{S}\text{-module categories} \\
  &\text{respecting }
  (\mathrm{F},\mathrm{F}^{\blackdiamond}, \eta^{\mathrm{F},l},\eta^{\mathrm{F},r}, \varepsilon^{\mathrm{F},l},\varepsilon^{\mathrm{F},r}) \text{ such that} \\
  &\csym{S}\star \mathbf{M} = \mathbf{M}
  \end{aligned}
\left.\rule{0cm}{1cm}\right\}/&\simeq
 \longrightarrow
  \Bigg\{
  \begin{aligned}
  &\text{Cauchy complete monoidal }\widetilde{\csym{S}}\text{-module} \\
  &\text{categories such that } \csym{S}\star \mathbf{M} = \mathbf{M}
  \end{aligned}
\Bigg\}/\simeq
 \\
 &\mathbf{M} \longmapsto  [\mathbf{M}^{\on{op}},\mathbf{Vec}_{\Bbbk}]\!\on{-f.g.proj} \\
 &\mathbf{N}_{|\ccf{S}} \longmapsfrom  \mathbf{N}
 \end{aligned}
  \end{equation}
\end{corollary}

\begin{remark}
 If the semigroup category $\csym{S}$ is {\it finitary}, i.e. there is a finite-dimensional algebra $B$ such that there is an equivalence of categories $B\!\on{-proj} \simeq \csym{S}$, then we may replace categories of presheaves in the above arguments with categories of finite-dimensional presheavs, i.e. instead of considering $[\csym{S}^{\on{op}},\mathbf{Vec}_{\Bbbk}]$, we consider $[\csym{S}^{\on{op}},\mathbf{vec}_{\Bbbk}]$. Indeed, in the finitary case, the latter becomes a monoidal subcategory of the former, and one can view it as free cocompletion under {\it finite} colimits. This cocompletion is equivalent to the projective abelianization of $\csym{S}$ in \cite{KMZ}. See \cite[Corollary~10.21]{St} and \cite[Corollary~10.23]{St} for details.
\end{remark}

\begin{remark}
 The statement of \cite[Theorem~3.16]{KMZ}, and, as a consequence, also of \cite[Corollary~3.17]{KMZ}, is missing the assumption that the $\widetilde{\csym{S}}$-module categories in the bijection above are also required to satisfy the assumption $\csym{S}\star \mathbf{M} = \mathbf{M}$. To see that the statement does not necessarily hold without this assumption, let $A$ be a finite-dimensional $\Bbbk$-algebra and consider the semigroup subcategory $A\!\on{-proj-}\!A = \on{add}\setj{A \kotimes A}$ of the monoidal category $A\!\on{-bimod-}\!A$.
 We have $[A\!\on{-proj-}A,\mathbf{vec}_{\Bbbk}] \simeq A\!\on{-bimod-}\!A$, and so $\widetilde{A\!\on{-proj-}\!A} \simeq \on{add}\setj{A\kotimes A, A}$ as monoidal categories. The monoidal category $\widetilde{A\!\on{-proj-}\!A}$ can in particular be identified with the $2$-category $\csym{C}_{\!A}$ studied in \cite[Section~5]{MM5}.

 Assume that $A$ is not separable, so that the regular bimodule ${}_{A}A_{A}$ is not a direct summand of $A \kotimes A$. The quotient of $\on{add}\setj{A\kotimes A, A}$ by the monoidal ideal generated by $\setj{\on{id}_{A \kotimes A}} \cup \on{Rad}\on{End}_{A\!\on{-bimod-}\!A}(A)$ is equivalent to $\mathbf{vec}_{\Bbbk}$. This gives a monoidal functor $\widetilde{A\!\on{-proj-}\!A} \rightarrow \mathbf{vec}_{\Bbbk}$, which endows $\mathbf{vec}_{\Bbbk}$ with the structure of a $\widetilde{A\!\on{-proj-}\!A}$-module category, where $A\kotimes A$ acts as zero, but $A$ acts isomorphically to identity. Restricting to a semigroup $(A\!\on{-proj-}\!A)$-module category, we get a zero action, and extending back to a $\widetilde{A\!\on{-proj-}\!A}$-module category yields a nonunital semigroup module category where also $A$ acts as zero. Hence the bijection would fail if we allowed this module category. We remark that this module category should be identified with the cell $2$-representation $\mathbf{C}_{\mathcal{L}_{0}}$ of $\widetilde{A\!\on{-proj-}\!A}$ associated to the left cell $\setj{A}$ - see \cite[Section~5]{MM5}.

 We now consider this problem using the setup of \cite{KMZ}. For simplicity, let us assume that $A$ is local. Then a bilax unit for $A\!\on{-proj-}\!A$ is $A \kotimes A$, together with the structure maps described in \cite[Section~4.1]{KMZ}. The epimorphism $\xi_{\mathtt{i}}$ introduced in \cite[Lemma~3.8]{KMZ} is given by the multiplication map $A \kotimes A \rightarrow A$. The argument that leads \cite{KMZ} to omit the condition $\csym{S} \star \mathbf{M} = \mathbf{M}$ is the assumption that this epimorphism should become non-zero in any $\widetilde{A\!\on{-proj-}\!A}$-module category. This is not true for the module category $\mathbf{C}_{\mathcal{L}_{0}}$ described above. Note that it would be true for any module category which extends to a $[A\!\on{-proj-}\!A,\mathbf{vec}_{\Bbbk}]$-module category whose action functor is right exact in both variables.
\end{remark}

\subsection{The non-liberal case}\label{s33}

In this section, we aim to show that the assumption of the existence of an object $\mathrm{F} \in \csym{S}$ such that $\mathrm{F} \boxtimes -$ is liberal can be removed, if we replace our use of resolvent pairs and comonad cohomology with a simple calculation.

\begin{theorem}\label{UnitGeneralCase}
 Let $\csym{S}$ be a rigid semigroup category. We have
 \[
  \mathbb{1}_{[\ccf{S}^{\on{op}},\mathbf{Vec}_{\Bbbk}]} \simeq
   \on{Coeq} \Bigg(
\begin{tikzcd}
	{\coprod_{\mathrm{F,G} \in \ccf{S}} \Hom{-,\mathrm{FF^{\blackdiamond}GG^{\blackdiamond}}}} &&& {\coprod_{\mathrm{F} \in \ccf{S}} \Hom{-,\mathrm{FF^{\blackdiamond}}}}
	\arrow["{(\Hom{-,\varepsilon^{\mathrm{F},r}_{\mathrm{G}}\mathrm{G}^{\blackdiamond}})_{\mathrm{F,G}}}", shift left=2, from=1-1, to=1-4]
	\arrow["{({\Hom{-,\mathrm{F}\varepsilon^{\mathrm{F},l}_{\mathrm{F}^{\blackdiamond}}}})_{\mathrm{F,G}}}"', shift right=2, from=1-1, to=1-4]
\end{tikzcd} \Bigg)
 \]
\end{theorem}

\begin{proof}
Consider the diagram
\begin{equation}\label{ToCoeq}
\begin{tikzcd}
	{\coprod_{\mathrm{F,G} \in \ccf{S}} \Hom{-,\mathrm{FF^{\blackdiamond}GG^{\blackdiamond}}} \oast \Hom{-,\mathrm{H}}} &&& {\coprod_{\mathrm{F} \in \ccf{S}} \Hom{-,\mathrm{FF^{\blackdiamond}}}\oast \Hom{-,\mathrm{H}}}
	\arrow["{(\Hom{-,\varepsilon^{\mathrm{F},r}_{\mathrm{G}}\mathrm{G}^{\blackdiamond}} \oast \Hom{-,\mathrm{H}})_{\mathrm{F,G}}}", shift left=2, from=1-1, to=1-4]
	\arrow["{({\Hom{-,\mathrm{F}\varepsilon^{\mathrm{F},l}_{\mathrm{F}^{\blackdiamond}}}}\oast \Hom{-,\mathrm{H}})_{\mathrm{F,G}}}"', shift right=2, from=1-1, to=1-4]
\end{tikzcd}
\end{equation}
and denote its coequalizer by $\euler{U}(\mathrm{H})$. Using Axiom~\ref{ZigZag}, we find that the composite morphism
\[\begin{tikzcd}
	{{\coprod_{\mathrm{F} \in \ccf{S}} \Hom{-,\mathrm{FF^{\blackdiamond}}}\oast \Hom{-,\mathrm{H}}}} & {\coprod_{\mathrm{F} \in \ccf{S}} \Hom{-,\mathrm{FF^{\blackdiamond}H}}} && {\Hom{-,\mathrm{H}}}
	\arrow["\simeq", from=1-1, to=1-2]
	\arrow["{{\left(\Hom{-,\varepsilon^{\mathrm{F},r}_{\mathrm{H}}}\right)_{\mathrm{F}}}}", shift left, from=1-2, to=1-4]
\end{tikzcd}\]
coequalizes Diagram~\eqref{ToCoeq}, and so we obtain the morphism $\theta: \euler{U}(\mathrm{H}) \rightarrow \mathrm{H}$, making the following diagram commute:
\[\begin{tikzcd}
	{{\coprod_{\mathrm{F} \in \ccf{S}} \Hom{-,\mathrm{FF^{\blackdiamond}}}\oast \Hom{-,\mathrm{H}}}} && {\euler{U}(\mathrm{H})} \\
	{\coprod_{\mathrm{F} \in \ccf{S}} \Hom{-,\mathrm{FF^{\blackdiamond}H}}} && {\Hom{-,\mathrm{H}}}
	\arrow[two heads, from=1-1, to=1-3]
	\arrow["{\exists! \theta_{\mathrm{H}}}", dashed, from=1-3, to=2-3]
	\arrow["{\Hom{-,(\varepsilon^{\mathrm{F},r}_{\mathrm{H}})_{\mathrm{F}}}}"', from=2-1, to=2-3]
	\arrow["\simeq"', from=1-1, to=2-1]
\end{tikzcd}\]
We now show that $\theta_{\mathrm{H}}$ is a split epi. Indeed, we have the following commutative diagram:
\[\begin{tikzcd}
	&& {{\coprod_{\mathrm{F} \in \ccf{S}} \Hom{-,\mathrm{FF^{\blackdiamond}}}\oast \Hom{-,\mathrm{H}}}} && {\euler{U}(\mathrm{H})} \\
	{\Hom{-,\mathrm{H}}} & {\Hom{-,\mathrm{HH^{\blackdiamond}H}}} & {\coprod_{\mathrm{F} \in \ccf{S}} \Hom{-,\mathrm{FF^{\blackdiamond}H}}} && {\Hom{-,\mathrm{H}}}
	\arrow[two heads, from=1-3, to=1-5]
	\arrow["{\exists! \theta_{\mathrm{H}}}", dashed, from=1-5, to=2-5]
	\arrow["{\Hom{-,(\varepsilon^{\mathrm{F},r}_{\mathrm{H}})_{\mathrm{F}}}}"', from=2-3, to=2-5]
	\arrow["\simeq", from=2-3, to=1-3]
	\arrow["{\Hom{-,\eta_{\mathrm{H}}^{\mathrm{H},l}}}", from=2-1, to=2-2]
	\arrow[hook, from=2-2, to=2-3]
\end{tikzcd}\]
where the shorter path equals $\on{id}_{\Hom{-,\mathrm{H}}}$, by Axiom~\ref{ZigZag}.

We now show that the section $\Hom{-,\mathrm{H}} \rightarrow \euler{U}(\mathrm{H})$ in the above diagram, which we will denote by $\sigma$, is an epimorphism.

Let $\mathrm{P} \in [\csym{S}^{\on{op}},\mathbf{Vec}_{\Bbbk}]$ and $\alpha: \euler{U}(\mathrm{H}) \rightarrow \mathrm{P}$ be such that $\alpha \circ \sigma = 0$. We want to show that also $\alpha = 0$. By definition, $\alpha$ corresponds to a map $\overline{\alpha}: \coprod_{\mathrm{F} \in \ccf{S}} \Hom{-,\mathrm{FF^{\blackdiamond}}}\oast \Hom{-,H} \rightarrow \mathrm{P}$, coequalizing the morphisms in Diagram~\eqref{ToCoeq}, and such that the composite
\begin{equation}\label{JustZero}
\begin{tikzcd}
	{\Hom{-,\mathrm{H}}} & {\Hom{-,\mathrm{HH^{\blackdiamond}H}}} & {\coprod_{\mathrm{F} \in \ccf{S}} \Hom{-,\mathrm{FF^{\blackdiamond}H}}} & {{\coprod_{\mathrm{F} \in \ccf{S}} \Hom{-,\mathrm{FF^{\blackdiamond}}}\oast \Hom{-,\mathrm{H}}}} & {\mathrm{P}}
	\arrow["{\overline{\alpha}}", two heads, from=1-4, to=1-5]
	\arrow["\simeq", from=1-3, to=1-4]
	\arrow["{\Hom{-,\eta_{\mathrm{H}}^{\mathrm{H},l}}}", from=1-1, to=1-2]
	\arrow[hook, from=1-2, to=1-3]
\end{tikzcd}
\end{equation}
is zero. It suffices to show that $\overline{\alpha} = 0$. Further, we claim that it suffices to show that $\overline{\alpha}_{\mathrm{H}}: \Hom{-,\mathrm{HH^{\blackdiamond}}} \oast \Hom{-,\mathrm{H}} \rightarrow \mathrm{P}$ is zero. In other words, we now show that $\overline{\alpha}_{\mathrm{H}} = 0$ implies $\overline{\alpha}_{\mathrm{G}} = 0$, for any $\mathrm{G}$. To see this, observe first that the coequalizing property of $\overline{\alpha}$ implies the commutativity of the following diagram:
\begin{equation}\label{ChangeGH}
\begin{tikzcd}
	& {\Hom{-,\mathrm{GG^{\blackdiamond}HH^{\blackdiamond}}}\oast \Hom{-,\mathrm{H}}} \\
	{\Hom{-,\mathrm{GG^{\blackdiamond}}}\oast \Hom{-,\mathrm{H}}} && {\Hom{-,\mathrm{HH^{\blackdiamond}}}\oast \Hom{-,\mathrm{H}}} \\
	& {\mathrm{P}}
	\arrow["{\overline{\alpha}_{\mathrm{G}}}", from=2-1, to=3-2]
	\arrow[""{name=0, anchor=center, inner sep=0}, "{\overline{\alpha}_{\mathrm{H}}}"', from=2-3, to=3-2]
	\arrow[""{name=1, anchor=center, inner sep=0}, "{\Hom{-,\mathrm{G}\varepsilon^{\mathrm{H},l}_{\mathrm{G}^{\blackdiamond}}}\oast \Hom{-,\mathrm{H}}}"', from=1-2, to=2-1]
	\arrow["{\Hom{-,\varepsilon^{\mathrm{G},r}_{\mathrm{H}}\mathrm{H}^{\blackdiamond}}}", from=1-2, to=2-3]
	\arrow["\circlearrowright"{description}, draw=none, from=0, to=1]
\end{tikzcd}
\end{equation}

So if $\overline{\alpha}_{\mathrm{H}} = 0$ then also $\overline{\alpha}_{\mathrm{G}} \circ {\Hom{-,\mathrm{G}\varepsilon^{\mathrm{H},l}_{\mathrm{G}^{\blackdiamond}}}\oast \Hom{-,\mathrm{H}}} = 0$. However, by Axiom~\ref{Involution}, the unique face of the following diagram commutes:
\[\begin{tikzcd}
	{\Hom{-,\mathrm{GG^{\blackdiamond}H}}} \\
	{\Hom{-,\mathrm{GG^{\blackdiamond}HH^{\blackdiamond}H}}} &&& {\Hom{-,\mathrm{GG^{\blackdiamond}HH^{\blackdiamond}H}}} \\
	{\Hom{-,\mathrm{GG^{\blackdiamond}HH^{\blackdiamond}}}\oast \Hom{-,\mathrm{H}}} &&& {\Hom{-,\mathrm{GG^{\blackdiamond}}}\oast \Hom{-,\mathrm{H}}} && {\mathrm{P}}
	\arrow["{\overline{\alpha}_{\mathrm{G}}}", from=3-4, to=3-6]
	\arrow["{\Hom{-,\mathrm{G}\varepsilon^{\mathrm{H},l}_{\mathrm{G}^{\blackdiamond}}}\oast \Hom{-,\mathrm{H}}}"', from=3-1, to=3-4]
	\arrow["\simeq"', from=2-1, to=3-1]
	\arrow["{\Hom{-,\mathrm{GG^{\blackdiamond}\eta^{\mathrm{H},l}_{\mathrm{H}}}}}"', from=1-1, to=2-1]
	\arrow["\simeq", from=2-4, to=3-4]
	\arrow["{\Hom{-,\mathrm{G}\mathrm{G}^{\blackdiamond}\varepsilon^{\mathrm{H},r}_{\mathrm{H}}}}", from=2-1, to=2-4]
\end{tikzcd}\]
The morphism defined by traversing the diagram is zero, since so is the bottom horizontal path. On the other hand,
by Axiom~\ref{ZigZag}, precomposing the above diagram with the isomorphism $\Hom{-,\mathrm{GG^{\blackdiamond}H}} \simeq \Hom{-,\mathrm{GG^{\blackdiamond}}} \oast \Hom{-,\mathrm{H}}$ yields $\overline{\alpha}_{\mathrm{G}}$, showing that indeed $\overline{\alpha}_{\mathrm{G}} = 0$.

We now show that $\overline{\alpha}_{H}$ is zero. Consider the diagram
\begin{equation}\label{HZero}
\begin{tikzcd}
	&&& {\Hom{-,\mathrm{H}}} \\
	{\Hom{-,\mathrm{HH^{\blackdiamond}H}}} && {\Hom{-,\mathrm{HH^{\blackdiamond}HH^{\blackdiamond}H}}} & {\Hom{-,\mathrm{HH^{\blackdiamond}H}}} & {\Hom{-,\mathrm{HH^{\blackdiamond}}}\oast \Hom{-,\mathrm{H}}} \\
	&& {\Hom{-,\mathrm{HH^{\blackdiamond}H}}} & {\Hom{-,\mathrm{HH^{\blackdiamond}}}\oast \Hom{-,\mathrm{H}}} & {\mathrm{P}}
	\arrow["{\Hom{-,\mathrm{HH^{\blackdiamond}}\eta^{\mathrm{H},l}_{\mathrm{H}}}}", from=2-1, to=2-3]
	\arrow[""{name=0, anchor=center, inner sep=0}, "{\Hom{-,\mathrm{H}\varepsilon^{\mathrm{H},l}_{\mathrm{H}^{\blackdiamond}}\mathrm{H}}}", from=2-3, to=3-3]
	\arrow[""{name=1, anchor=center, inner sep=0}, "{\Hom{-,\varepsilon^{\mathrm{H},r}_{\mathrm{H}}\mathrm{H^{\blackdiamond}H}}}", from=2-3, to=2-4]
	\arrow["{\overline{\alpha}_{\mathrm{H}}}", from=2-5, to=3-5]
	\arrow["{\overline{\alpha}_{\mathrm{H}}}"{description}, from=3-4, to=3-5]
	\arrow[""{name=2, anchor=center, inner sep=0}, "\simeq"{description}, from=2-4, to=2-5]
	\arrow[""{name=3, anchor=center, inner sep=0}, "\simeq", from=3-3, to=3-4]
	\arrow[""{name=4, anchor=center, inner sep=0}, "{=}"', from=2-1, to=3-3]
	\arrow[""{name=5, anchor=center, inner sep=0}, "{\Hom{-,\varepsilon^{\mathrm{H},r}_{\mathrm{H}}}}"{description}, curve={height=-18pt}, from=2-1, to=1-4]
	\arrow["{\Hom{-,\eta^{\mathrm{H},l}}_{\mathrm{H}}}", from=1-4, to=2-4]
	\arrow["{(1)}"{description}, draw=none, from=4, to=0]
	\arrow["{(2)}"{description}, draw=none, from=3, to=2]
	\arrow["{(3)}"{description}, draw=none, from=5, to=1]
\end{tikzcd}
\end{equation}

Face (1) commutes by Axiom~\ref{ZigZag}. Face (2) commutes since the morphisms in Diagram~\eqref{ToCoeq} are coequalized by $\overline{\alpha}$. Face (3) commutes by the naturality of $\varepsilon^{\mathrm{H},r}$; alternatively, by the naturality of $\eta^{\mathrm{H},l}$. Traversing the above diagram along the lower path yields $\overline{\alpha}_{\mathrm{H}}$; traversing the upper path gives the zero morphism, since the morphism in Diagram~\eqref{JustZero} is zero.

Thus, $\theta_{\mathrm{H}}$ is a split epi such that the section $\sigma_{\mathrm{H}}$ a split mono and an epi; hence, $\theta_{\mathrm{H}}$ and $\sigma_{\mathrm{H}}$ are mutually inverse isomorphisms. Further, $\theta_{\mathrm{H}}$ clearly is natural in $\mathrm{H}$, and so we conclude that
\[
 \Bigg( \on{Coeq} \Bigg(
\begin{tikzcd}
	{\coprod_{\mathrm{F,G} \in \ccf{S}} \Hom{-,\mathrm{FF^{\blackdiamond}GG^{\blackdiamond}}}} &&& {\coprod_{\mathrm{F} \in \ccf{S}} \Hom{-,\mathrm{FF^{\blackdiamond}}}}
	\arrow["{(\Hom{-,\varepsilon^{\mathrm{F},r}_{\mathrm{G}}\mathrm{G}^{\blackdiamond}})_{\mathrm{F,G}}}", shift left=2, from=1-1, to=1-4]
	\arrow["{({\Hom{-,\mathrm{F}\varepsilon^{\mathrm{F},l}_{\mathrm{F}^{\blackdiamond}}}})_{\mathrm{F,G}}}"', shift right=2, from=1-1, to=1-4]
\end{tikzcd} \Bigg) \Bigg) \oast \Hom{-,\mathrm{H}} \simeq \Hom{-,\mathrm{H}}
\]
for all $\mathrm{H}$, naturally in $\mathrm{H}$. Since
\[
 \Bigg( \on{Coeq} \Bigg(
\begin{tikzcd}
	{\coprod_{\mathrm{F,G} \in \ccf{S}} \Hom{-,\mathrm{FF^{\blackdiamond}GG^{\blackdiamond}}}} &&& {\coprod_{\mathrm{F} \in \ccf{S}} \Hom{-,\mathrm{FF^{\blackdiamond}}}}
	\arrow["{(\Hom{-,\varepsilon^{\mathrm{F},r}_{\mathrm{G}}\mathrm{G}^{\blackdiamond}})_{\mathrm{F,G}}}", shift left=2, from=1-1, to=1-4]
	\arrow["{({\Hom{-,\mathrm{F}\varepsilon^{\mathrm{F},l}_{\mathrm{F}^{\blackdiamond}}}})_{\mathrm{F,G}}}"', shift right=2, from=1-1, to=1-4]
\end{tikzcd} \Bigg) \Bigg) \oast -
\]
is cocontinuous, we conclude that
\[
\on{Id}_{[\csym{S}^{\on{op}},\mathbf{Vec}_{\Bbbk}]} \simeq
 \Bigg( \on{Coeq} \Bigg(
\begin{tikzcd}
	{\coprod_{\mathrm{F,G} \in \ccf{S}} \Hom{-,\mathrm{FF^{\blackdiamond}GG^{\blackdiamond}}}} &&& {\coprod_{\mathrm{F} \in \ccf{S}} \Hom{-,\mathrm{FF^{\blackdiamond}}}}
	\arrow["{(\Hom{-,\varepsilon^{\mathrm{F},r}_{\mathrm{G}}\mathrm{G}^{\blackdiamond}})_{\mathrm{F,G}}}", shift left=2, from=1-1, to=1-4]
	\arrow["{({\Hom{-,\mathrm{F}\varepsilon^{\mathrm{F},l}_{\mathrm{F}^{\blackdiamond}}}})_{\mathrm{F,G}}}"', shift right=2, from=1-1, to=1-4]
\end{tikzcd} \Bigg) \Bigg) \oast - .
\]
Similarly one shows that also
\[
 \on{Id}_{[\csym{S}^{\on{op}},\mathbf{Vec}_{\Bbbk}]} \simeq
 -\oast \Bigg( \on{Coeq} \Bigg(
\begin{tikzcd}
	{\coprod_{\mathrm{F,G} \in \ccf{S}} \Hom{-,\mathrm{FF^{\blackdiamond}GG^{\blackdiamond}}}} &&& {\coprod_{\mathrm{F} \in \ccf{S}} \Hom{-,\mathrm{FF^{\blackdiamond}}}}
	\arrow["{(\Hom{-,\varepsilon^{\mathrm{F},r}_{\mathrm{G}}\mathrm{G}^{\blackdiamond}})_{\mathrm{F,G}}}", shift left=2, from=1-1, to=1-4]
	\arrow["{({\Hom{-,\mathrm{F}\varepsilon^{\mathrm{F},l}_{\mathrm{F}^{\blackdiamond}}}})_{\mathrm{F,G}}}"', shift right=2, from=1-1, to=1-4]
\end{tikzcd} \Bigg) \Bigg).
\]
By Proposition~\ref{EGNOUnit}, we find
\[
\mathbb{1}_{[\ccf{S}^{\on{op}},\mathbf{Vec}_{\Bbbk}]} \simeq
\on{Coeq} \Bigg(
\begin{tikzcd}
	{\coprod_{\mathrm{F,G} \in \ccf{S}} \Hom{-,\mathrm{FF^{\blackdiamond}GG^{\blackdiamond}}}} &&& {\coprod_{\mathrm{F} \in \ccf{S}} \Hom{-,\mathrm{FF^{\blackdiamond}}}}
	\arrow["{(\Hom{-,\varepsilon^{\mathrm{F},r}_{\mathrm{G}}\mathrm{G}^{\blackdiamond}})_{\mathrm{F,G}}}", shift left=2, from=1-1, to=1-4]
	\arrow["{({\Hom{-,\mathrm{F}\varepsilon^{\mathrm{F},l}_{\mathrm{F}^{\blackdiamond}}}})_{\mathrm{F,G}}}"', shift right=2, from=1-1, to=1-4]
\end{tikzcd} \Bigg).
\]
\end{proof}

\begin{theorem}\label{Extension}
 Let $\csym{S}$ be a rigid semigroup category, and let $\mathbf{M}$ be a semigroup $\csym{S}$-module category such that $\csym{S}\star \mathbf{M} = \mathbf{M}$ and such that it respects all adjunctions in $\csym{S}$. Then $[\mathbf{M}^{\on{op}},\mathbf{Vec}_{\Bbbk}]$ is a monoidal $[\csym{S}^{\on{op}},\mathbf{Vec}_{\Bbbk}]$-module. Moreover, we obtain a bijection
  \begin{equation}
  \begin{aligned}
 \Bigg\{
  \begin{aligned}
  &\text{Cauchy complete semigroup }\csym{S}\text{-module categories}\\
  &\text{respecting adjunctions in }\csym{S},
  \text{ such that }
  \csym{S}\star \mathbf{M} = \mathbf{M}
  \end{aligned}
  \Bigg\}/&\simeq
 \longleftrightarrow
  \Bigg\{
  \begin{aligned}
  &\text{Monoidal }[\csym{S}^{\on{op}},\mathbf{Vec}_{\Bbbk}]\text{-module presheaf } \\
  &\text{categories such that }(\overline{\varepsilon}^{\mathrm{F},\mathbf{M}}_{!})_{\mathrm{F}} \text{ is epi}
  \end{aligned}
\Bigg\}/\simeq
 \\
 &\mathbf{M} \longmapsto  [\mathbf{M}^{\on{op}},\mathbf{Vec}_{\Bbbk}] \\
 &\mathbf{N}\!\on{-f.g.proj} \longmapsfrom \mathbf{N}
 \end{aligned}
  \end{equation}
\end{theorem}

\begin{proof}
 First, we show that, given a semigroup $\csym{S}$-module category satisfying the assumptions of the theorem, the $[\csym{S}^{\on{op}},\mathbf{Vec}_{\Bbbk}]$-module category $[\mathbf{M}^{\on{op}},\mathbf{Vec}_{\Bbbk}]$, as described in in Definition~\ref{MV}, is a monoidal module category.

 The proof is analogous to that of Theorem~\ref{UnitGeneralCase}; we will describe the necessary minor modifications in detail.
 Given $X \in \mathbf{M}$, the morphism
\[\begin{tikzcd}
	{\coprod_{\mathrm{F}}\Hom{-,\mathrm{FF^{\blackdiamond}}} \ostar \Hom{-,X}} & {\coprod_{\mathrm{F}}\Hom{-,\mathrm{FF^{\blackdiamond}}\star X}} && {\Hom{-,X}}
	\arrow["{(\varepsilon^{\mathrm{F},\mathbf{M}}_{X})_{\mathrm{F}}}", from=1-2, to=1-4]
	\arrow["\simeq", from=1-1, to=1-2]
\end{tikzcd}\]
which we denote by $(\overline{\varepsilon}^{\mathrm{F},\mathbf{M}}_{X})_{\mathrm{F}}$ below, coequalizes the diagram
\begin{equation}\label{ModuleCoeq}
\begin{tikzcd}[ampersand replacement=\&]
	{\coprod_{\mathrm{F,G}}\Hom{-,\mathrm{FF^{\blackdiamond}GG^{\blackdiamond}}} \ostar \Hom{-,X}} \&\& {\coprod_{\mathrm{F}}\Hom{-,\mathrm{FF^{\blackdiamond}}} \ostar \Hom{-,X}}
	\arrow["{(\Hom{-,\varepsilon^{\mathrm{F},r}_{\mathrm{G}}\mathrm{G}^{\blackdiamond}})_{\mathrm{F,G}} \ostar \Hom{-,X}}", shift left=2, from=1-1, to=1-3]
	\arrow["{({\Hom{-,\mathrm{F}\varepsilon^{\mathrm{F},l}_{\mathrm{F}^{\blackdiamond}}}})_{\mathrm{F,G}} \ostar \Hom{-,X}}"', shift right=2, from=1-1, to=1-3]
\end{tikzcd}.
\end{equation}
Using the description of $\mathbb{1}_{[\ccf{S}^{\on{op}},\mathbf{Vec}_{\Bbbk}]}$ in Theorem~\ref{UnitGeneralCase}, we find a morphism $\theta^{\mathbf{M}}_{X}$ from $[\mathbf{M}^{\on{op}},\mathbf{Vec}_{\Bbbk}](\mathbb{1}_{[\ccf{S}^{\on{op}},\mathbf{Vec}_{\Bbbk}]})(\Hom{-,X})$ to $\Hom{-,X}$. Further, $\theta^{\mathbf{M}}_{X}$ is clearly natural in $X$. We now show that $\theta^{\mathbf{M}}_{X}$ is a split epimorphism, for any $X$. Using the assumption that $\csym{S}\star \mathbf{M} = \mathbf{M}$, we choose $\mathrm{H} \in \csym{S}$ and $Z \in \mathbf{M}$, a split mono $\iota: X \hookrightarrow \mathrm{H} \star Z$ and a split epi $\pi: \mathrm{H} \star Z \twoheadrightarrow X$. Consider the following diagram:
\begin{equation}\label{ModuleSection}
\begin{tikzcd}
	{\Hom{-,X}} & {\Hom{-,\mathrm{H} \star Z}} & {\Hom{-,\mathrm{HH^{\blackdiamond}H}\star Z}} & {\Hom{-,\mathrm{HH^{\blackdiamond}}\star X}} & {\coprod_{\mathrm{F}}\Hom{-,\mathrm{FF^{\blackdiamond}}\star X}} \\
	&& {\Hom{-,\mathrm{H}\star Z}} & {\Hom{-,X}} & {\coprod_{\mathrm{F}}\Hom{-,\mathrm{FF^{\blackdiamond}}} \ostar \Hom{-,X}}
	\arrow[""{name=0, anchor=center, inner sep=0}, "{\Hom{-,\varepsilon^{\mathrm{H},\mathbf{M}}_{X}}}"', from=1-4, to=2-4]
	\arrow["{\Hom{-,\iota}}", from=1-1, to=1-2]
	\arrow["{\Hom{-,\eta^{\mathrm{H},l}_{\mathrm{H}}\star Z}}", from=1-2, to=1-3]
	\arrow["{\Hom{-,\varepsilon^{\mathrm{H},r}_{\mathrm{H}}\star Z}}", from=1-3, to=2-3]
	\arrow[""{name=1, anchor=center, inner sep=0}, "{\on{id}}"', from=1-2, to=2-3]
	\arrow[""{name=2, anchor=center, inner sep=0}, "{\Hom{-,\mathrm{HH^{\blackdiamond}}\star \pi}}", from=1-3, to=1-4]
	\arrow[""{name=3, anchor=center, inner sep=0}, "{\Hom{-,\pi}}"', from=2-3, to=2-4]
	\arrow[hook, from=1-4, to=1-5]
	\arrow["\simeq", from=1-5, to=2-5]
	\arrow["{(\varepsilon^{\mathrm{F},\mathbf{M}}_{X})_{\mathrm{F}}}"{description}, from=1-5, to=2-4]
	\arrow[""{name=4, anchor=center, inner sep=0}, "{(\overline{\varepsilon}^{\mathrm{F},\mathbf{M}}_{X})_{\mathrm{F}}}", from=2-5, to=2-4]
	\arrow["{(1)}"{description}, draw=none, from=1, to=1-3]
	\arrow["{(3)}"{description}, shift left, draw=none, from=0, to=1-5]
	\arrow["{(4)}"{description, pos=0.6}, shift right, draw=none, from=4, to=1-5]
	\arrow["{(2)}"{description}, draw=none, from=3, to=2]
\end{tikzcd}
\end{equation}

Face (1) commutes by Axiom~\ref{ZigZag}. The adjunction $(\mathrm{H},\mathrm{H}^{\blackdiamond}, \eta^{\mathrm{H},l},\eta^{\mathrm{H},r}, \varepsilon^{\mathrm{H},l},\varepsilon^{\mathrm{H},r})$ is respected in $\mathbf{M}$, which yields $\Hom{-,\varepsilon^{\mathrm{H},r}_{\mathrm{H}} \star Z} = \Hom{-,\varepsilon^{\mathrm{H},\mathbf{M}}_{\mathrm{H}\star Z}}$. Using this equality, one sees that Face (2) commutes by the naturality of $\varepsilon^{\mathrm{H},\mathbf{M}}$. Face (3) commutes by the elementary properties of coproducts, and Face (4) commutes by the definition of ${(\overline{\varepsilon}^{\mathrm{F},\mathbf{M}}_{X})_{\mathrm{F}}}$. Traversing the diagram along its outer paths shows that ${(\overline{\varepsilon}^{\mathrm{F},\mathbf{M}}_{X})_{\mathrm{F}}}$, and hence also $\theta^{\mathbf{M}}$, is a split epi. Similarly to the proof of Theorem~\ref{UnitGeneralCase}, we now show that the section $\sigma$ of $\theta^{\mathbf{M}}$ obtained in Diagram~\eqref{ModuleSection} also is an epimorphism. We let $\alpha$ be a morphism from the coequalizer of Diagram~\eqref{ModuleCoeq} to an object $\mathrm{Q} \in [\mathbf{M}^{\on{op}},\mathbf{Vec}_{\Bbbk}]$, such that $\alpha \circ \sigma = 0$. Equivalently, we consider a morphism $\overline{\alpha}: \coprod_{\mathrm{F}} \Hom{-,\mathrm{FF^{\blackdiamond}}}\ostar \Hom{-,X} \rightarrow \mathrm{Q}$ coequalizing Diagram~\eqref{ModuleCoeq}, such that
\begin{equation}\label{Stonoga}
\begin{tikzcd}[row sep = small]
	{\Hom{-,\mathrm{H} \star Z}} & {\Hom{-,\mathrm{HH^{\blackdiamond}H}\star Z}} & {\Hom{-,\mathrm{HH^{\blackdiamond}}\star X}} & {\coprod_{\mathrm{F}}\Hom{-,\mathrm{FF^{\blackdiamond}}\star X}} & {\coprod_{\mathrm{F}}\Hom{-,\mathrm{FF^{\blackdiamond}}} \ostar \Hom{-,X}} \\
	{\Hom{-,X}} &&&& {\mathrm{Q}}
	\arrow["{\Hom{-,\iota}}", from=2-1, to=1-1]
	\arrow["{\Hom{-,\eta^{\mathrm{H},l}_{\mathrm{H}}\star Z}}", from=1-1, to=1-2]
	\arrow["{\Hom{-,\mathrm{HH^{\blackdiamond}}\star \pi}}", from=1-2, to=1-3]
	\arrow[hook, from=1-3, to=1-4]
	\arrow["\simeq", from=1-4, to=1-5]
	\arrow["{\overline{\alpha}}", from=1-5, to=2-5]
\end{tikzcd}
\end{equation}
is zero. We want to show that $\overline{\alpha}$ itself is zero. Similarly to the proof of Theorem~\ref{UnitGeneralCase} it suffices to show that $\alpha_{\mathrm{H}}$ is zero. Consider the following diagram:
\[\begin{tikzcd}[ampersand replacement=\&]
	{\Hom{-,\mathrm{GG^{\blackdiamond}}\star X}} \&\& {\Hom{-,\mathrm{GG^{\blackdiamond}}} \ostar \Hom{-,X}} \& {\mathrm{Q}} \\
	{\Hom{-,\mathrm{GG^{\blackdiamond}}\star\mathrm{H}Y}} \& {\Hom{-,\mathrm{GG^{\blackdiamond}}\star \mathrm{H}Y}} \& {\Hom{-,\mathrm{GG^{\blackdiamond}}} \ostar \Hom{-,\mathrm{H}Y}} \& {\Hom{-,\mathrm{GG^{\blackdiamond}}}\ostar \Hom{-,X}} \\
	\\
	{\Hom{-,\mathrm{GG^{\blackdiamond}}\star \mathrm{HH^{\blackdiamond}H}Y}} \&\& {\Hom{-,\mathrm{GG^{\blackdiamond}}\star \mathrm{HH^{\blackdiamond}}X}} \& {\Hom{-,\mathrm{GG^{\blackdiamond}HH^{\blackdiamond}}}\ostar \Hom{-,X}}
	\arrow[""{name=0, anchor=center, inner sep=0}, "\simeq"{description}, from=1-3, to=1-1]
	\arrow["{\Hom{-,\mathrm{GG^{\blackdiamond}\star \iota}}}"', from=1-1, to=2-1]
	\arrow["{\Hom{-,\mathrm{GG^{\blackdiamond}}\star \eta^{\mathrm{H},\mathbf{M}}_{\mathrm{H}\star Y}}}"', from=2-1, to=4-1]
	\arrow["{\Hom{-,\mathrm{GG^{\blackdiamond}}\star \mathrm{HH^{\blackdiamond}}\pi}}"', shift right, from=4-1, to=4-3]
	\arrow["{\on{id}}"{description}, from=2-1, to=2-2]
	\arrow[""{name=1, anchor=center, inner sep=0}, "{\Hom{-,\mathrm{GG^{\blackdiamond}}\star \varepsilon^{\mathrm{H},\mathbf{M}}_{\mathrm{H}\star Y}}}"{description}, from=4-1, to=2-2]
	\arrow["\simeq"', from=4-3, to=4-4]
	\arrow[""{name=2, anchor=center, inner sep=0}, "{\on{id}}"{description}, from=1-3, to=2-4]
	\arrow["{\Hom{-,\mathrm{GG^{\blackdiamond}}\varepsilon^{\mathrm{H},l}_{\mathrm{G}^{\blackdiamond}}}\ostar \Hom{-,X}}"', from=4-4, to=2-4]
	\arrow["{\overline{\alpha}_{\mathrm{G}}}"{description}, from=1-3, to=1-4]
	\arrow[""{name=3, anchor=center, inner sep=0}, "{\overline{\alpha}_{\mathrm{G}}}", from=2-4, to=1-4]
	\arrow[""{name=4, anchor=center, inner sep=0}, "\simeq", from=2-2, to=2-3]
	\arrow["{\Hom{-,\mathrm{GG^{\blackdiamond}}} \ostar \Hom{-,\pi}}"', shift right, from=2-3, to=2-4]
	\arrow["{\Hom{-,\mathrm{GG^{\blackdiamond}}} \ostar \Hom{-,\iota}}"', from=1-3, to=2-3]
	\arrow["{(1)}"{description}, draw=none, from=2-2, to=0]
	\arrow["{(2)}"{description}, draw=none, from=2-1, to=1]
	\arrow["{(3)}"{description}, draw=none, from=4-3, to=4]
	\arrow["{(4)}"{description, pos=0.3}, shift left, draw=none, from=2-3, to=2]
	\arrow["{(5)}"{description}, draw=none, from=2, to=3]
\end{tikzcd}\]
Here, Face (1) commutes due to the naturality of the coherence isomorphisms. Face (2) commutes due to the zigzag equations for $\eta^{\mathrm{H},\mathbf{M}}$ and $\varepsilon^{\mathrm{H},\mathbf{M}}$. Face (3) commutes due to the interchange law combined with the naturality of coherence isomorphisms and the adjunction $(\mathrm{H},\mathrm{H}^{\blackdiamond}, \eta^{\mathrm{H},l},\eta^{\mathrm{H},r}, \varepsilon^{\mathrm{H},l},\varepsilon^{\mathrm{H},r})$ being respected in $\mathbf{M}$. Face (4) commutes due to $\pi \circ \iota = \on{id}_{X}$ and Face (5) commutes trivially. The short outer path around the diagram yields $\overline{\alpha}_{\mathrm{G}}$; traversing the long outer path in the diagram yields the zero morphism if $\overline{\alpha}_{\mathrm{H}} = 0$. This follows similarly to the analogous claim in the proof of Theorem~\ref{UnitGeneralCase}, using a diagram analogous to Diagram~\eqref{ChangeGH}.

The analogue of Diagram~\eqref{HZero}, which shows that $\overline{\alpha}_{\mathrm{H}}$ is zero, is the following:
\[\begin{tikzcd}[ampersand replacement=\&,column sep=scriptsize]
	\& X \&\& {\mathrm{H}Z} \&\& {\mathrm{HH^{\blackdiamond}H}Z} \&\& {\mathrm{HH^{\blackdiamond}}X} \\
	{\mathrm{HH^{\blackdiamond}}X} \&\& {\mathrm{HH^{\blackdiamond}H}Z} \&\& {\mathrm{HH^{\blackdiamond}HH^{\blackdiamond}H}Z} \&\& {\mathrm{HH^{\blackdiamond}HH^{\blackdiamond}}X} \\
	\&\&\&\& {\mathrm{HH^{\blackdiamond}H}Z} \&\& {\mathrm{HH^{\blackdiamond}}X} \& {\mathrm{Q}}
	\arrow["{\mathrm{HH^{\blackdiamond}}\iota}", from=2-1, to=2-3]
	\arrow[""{name=0, anchor=center, inner sep=0}, "{\varepsilon^{\mathrm{H},\mathbf{M}}_{X}}", from=2-1, to=1-2]
	\arrow["\iota", from=1-2, to=1-4]
	\arrow[""{name=1, anchor=center, inner sep=0}, "{\varepsilon^{\mathrm{H},\mathbf{M}}_{\mathrm{H}Z}}"{pos=0.4}, from=2-3, to=1-4]
	\arrow["{\mathrm{H}\eta^{\mathrm{H},\mathbf{M}}_{Z}}", from=1-4, to=1-6]
	\arrow[""{name=2, anchor=center, inner sep=0}, "{\varepsilon^{\mathrm{H},\mathbf{M}}_{\mathrm{HH^{\blackdiamond}H}Z}}"{pos=0.4}, from=2-5, to=1-6]
	\arrow["{\mathrm{HH^{\blackdiamond}H}\eta^{\mathrm{H},\mathbf{M}}_{Z}}", from=2-3, to=2-5]
	\arrow["{\mathrm{HH^{\blackdiamond}HH^{\blackdiamond}}\pi}"', from=2-5, to=2-7]
	\arrow[""{name=3, anchor=center, inner sep=0}, "{\mathrm{HH^{\blackdiamond}}\varepsilon^{\mathrm{H},\mathbf{M}}_{X}}"{description}, from=2-7, to=1-8]
	\arrow["{\mathrm{HH^{\blackdiamond}}\pi}"{pos=0.6}, from=1-6, to=1-8]
	\arrow["{\overline{\alpha}_{\mathrm{H}}}", from=1-8, to=3-8]
	\arrow["{\mathrm{HH^{\blackdiamond}}\varepsilon^{\mathrm{H},\mathbf{M}}_{X}}", from=2-7, to=3-7]
	\arrow["{\mathrm{HH^{\blackdiamond}}\varepsilon^{\mathrm{H},\mathbf{M}}_{\mathrm{H}Z}}", from=2-5, to=3-5]
	\arrow[""{name=4, anchor=center, inner sep=0}, "{\mathrm{HH^{\blackdiamond}}\pi}"{description}, from=3-5, to=3-7]
	\arrow["{\overline{\alpha}_{\mathrm{H}}}"{description}, from=3-7, to=3-8]
	\arrow[""{name=5, anchor=center, inner sep=0}, "{\on{id}}"{description}, from=2-3, to=3-5]
	\arrow["{(1)}"{description}, draw=none, from=0, to=1]
	\arrow["{(2)}"{description}, draw=none, from=1, to=2]
	\arrow["{(3)}"{description}, draw=none, from=2, to=3]
	\arrow["{(4)}"{description, pos=0.6}, draw=none, from=5, to=2-5]
	\arrow["{(5)}"{description}, draw=none, from=4, to=2-7]
	\arrow["{(6)}"{description}, draw=none, from=3, to=3-8]
\end{tikzcd}\]
Here, Faces (1), (2), (3) and (5) commute by the naturality of $\varepsilon^{\mathrm{H},\mathbf{M}}$, Face (4) commutes by the zigzag equations for $\eta^{\mathrm{H},\mathbf{M}}$ and $\varepsilon^{\mathrm{H},\mathbf{M}}$. Face (6) commutes by definition. The upper outer path around the diagram is zero since the morphism given by Diagram~\eqref{Stonoga} is zero. The lower outer path is $\overline{\alpha}_{\mathrm{H}}$. We conclude that $\overline{\alpha}$ is zero, and so $\sigma$ is epi, which yields that $\theta^{\mathbf{M}}$ is an isomorphism, showing that
\[
\theta^{\mathbf{M}}_{X}: [\mathbf{M}^{\on{op}},\mathbf{Vec}_{\Bbbk}](\mathbb{1}_{[\ccf{S}^{\on{op}},\mathbf{Vec}_{\Bbbk}]})(\Hom{-,X}) \xiso \Hom{-,X}.
\]
Thus, $[\mathbf{M}^{\on{op}},\mathbf{Vec}_{\Bbbk}](\mathbb{1}_{[\ccf{S}^{\on{op}},\mathbf{Vec}_{\Bbbk}]})$ is cocontinuous and isomorphic to the identity on representables, so
\begin{equation}\label{ThisSuffices}
 [\mathbf{M}^{\on{op}},\mathbf{Vec}_{\Bbbk}](\mathbb{1}_{[\ccf{S}^{\on{op}},\mathbf{Vec}_{\Bbbk}]}) \simeq \on{Id}_{[\mathbf{M}^{\on{op}},\mathbf{Vec}_{\Bbbk}]}.
\end{equation}
Using Proposition~\ref{AnotherEGNOLemma}, formula~\eqref{ThisSuffices} is sufficient to conclude that $[\mathbf{M}^{\on{op}},\mathbf{Vec}_{\Bbbk}]$ is a monoidal $[\csym{S}^{\on{op}},\mathbf{Vec}_{\Bbbk}]$-module category.

The claimed bijective correspondence between equivalence classes of module categories can be shown using essentially the same proof as that of Theorem~\ref{AbelianizationCorrespondence}, with a few minor modifications.
We have shown that the components of $(\overline{\varepsilon}^{\mathrm{F},\mathbf{M}}_{!})_{\mathrm{F}}$ of the form $((\overline{\varepsilon}^{\mathrm{F},\mathbf{M}}_{!})_{\mathrm{F}})_{\Hom{-,X}}$ are epimorphisms for all $X \in \mathbf{M}$.
Since it is a transformation of cocontinuous functors, this is enough to conclude that $(\overline{\varepsilon}^{\mathrm{F},\mathbf{M}}_{!})_{\mathrm{F}}$ is an epimorphism, and so, given $\mathbf{M}$ on the left-hand side, the module category $[\mathbf{M}^{\on{op}},\mathbf{Vec}_{\Bbbk}]$ satisfies the conditions stated on the right-hand side.
The fact that $\mathbf{N}\!\on{-f.g.proj}$ is an $\csym{S}$-module subcategory, for any $\mathbf{N}$ in right-hand side, follows verbatim as in Theorem~\ref{AbelianizationCorrespondence}.
Since $((\overline{\varepsilon}^{\mathrm{F},\mathbf{N}}_{!})_{\mathrm{F}})_{\mathrm{Q}}$ is epi, there is a finite collection $\euler{G}(\mathrm{Q})$ of objects of $\csym{S}$ such that $((\overline{\varepsilon}^{\mathrm{F},\mathbf{N}}_{!})_{\mathrm{F}\in \euler{G}(\mathrm{Q})})_{\mathrm{Q}}$ is epi. Since an epimorphism onto a projective is split, we find that $((\overline{\varepsilon}^{\mathrm{F},\mathbf{N}}_{!})_{\mathrm{F}\in \euler{G}(\Hom{-,X})})_{\Hom{-,X}}$ is a split epi for any $X \in \mathbf{M}$, showing that $\csym{S} \star \mathbf{N}\!\on{-f.g.proj} = \mathbf{N}\!\on{-f.g.proj}$. The remainder of the proof is verbatim that of Theorem~\ref{AbelianizationCorrespondence}.
\end{proof}

\subsection{The symmetric case and enriched traces}\label{s34}

We now briefly recall the notion of a {\it trace} for a $\Bbbk$-linear category. A more extensive account, with a view towards categorification, can be found in \cite{BGHL}.
\begin{definition}
 Let $\mathcal{A}$ be a $\Bbbk$-linear category. The {\it trace} of $\mathcal{A}$ is the space
 \[
  \on{Tr}_{\Bbbk}(\mathcal{A}) = \int^{X \in \mathcal{A}} \on{Hom}_{\mathcal{A}}(X,X).
 \]
\end{definition}

The trace of a monoidal $\Bbbk$-linear category $\csym{C}$ is naturally endowed with the structure of a $\Bbbk$-algebra: writing
\[
\on{Tr}_{\Bbbk}(\csym{C}) = \Big(\bigoplus_{X \in \ccf{C}} \on{Hom}_{\ccf{C}}(X,X)\Big)/\left\langle f \circ g - g \circ f \; | \; f \in \on{Hom}_{\ccf{C}}(Y,Z), g \in \on{Hom}_{\ccf{C}}(Z,Y)  \right\rangle,
\]
one sees that the relation imposed can be graphically interpreted as identifying the top and the bottom boundary for the graphical calculus of $\csym{C}$ - thus, the trace can be thought of as the span of string diagrams drawn on an annulus, with the multiplication being given by ``stacking circles''.

Formally speaking, we have
\[
 \on{Tr}_{\Bbbk}(\csym{C}) =
\on{Coeq}\Bigg(
\begin{tikzcd}
	{\coprod_{\mathrm{F,G \in \ccf{C}}} \on{Hom}_{\ccf{C}}(\mathrm{F,G}) \kotimes \on{Hom}_{\ccf{C}}(\mathrm{G,F})} & {\coprod_{\mathrm{H \in \ccf{C}}} \on{Hom}_{\ccf{C}}(\mathrm{H,H}) }
	\arrow["{f \otimes g \mapsto f \circ g}", shift left=2, from=1-1, to=1-2]
	\arrow["{f\otimes g \mapsto g \circ f}"', shift right=2, from=1-1, to=1-2]
\end{tikzcd}
\Bigg)
\]
If $\csym{C}$ is closed symmetric, we could instead consider the $\csym{C}$-enriched category $\Cint$, defined by $\on{Ob}\Cint = \on{Ob}\csym{C}$ and $\Cint(\mathrm{F,G}) = [\mathrm{F,G}]$. We may then consider the $\csym{C}$-enriched variant of the coend defining the trace:
\[
 \on{Tr}_{\ccf{C}}(\Cint) =
 \on{Coeq}\Bigg(
\begin{tikzcd}
	{\coprod_{\mathrm{F,G} \in \Cintjr} \on{Hom}_{\Cintjr }(\mathrm{F,G}) \cotimes \on{Hom}_{\Cintjr }(\mathrm{G,F})} & {\coprod_{\mathrm{H} \in \Cintjr } \on{Hom}_{\Cintjr }(\mathrm{H,H}) }
	\arrow["{\mathsf{k}_{\mathrm{G,F,G}}}", shift left=2, from=1-1, to=1-2]
	\arrow["{\overline{\mathsf{k}}_{\mathrm{G,F,G}}}"', shift right=2, from=1-1, to=1-2]
\end{tikzcd}
\Bigg),
\]
where $\mathsf{k}$ is the composition transformation of $\Cint$, and $\overline{\mathsf{k}}$ is $\mathsf{k}$ precomposed with the symmetry of $\csym{C}$. Observe that this colimit does not necessarily exist in $\csym{C}$ - similarly to the approach of the preceding sections, we may want to consider the colimit of the image of the above diagram in a suitable cocompletion of $\csym{C}$.
By definition,
\[
 \on{Tr}_{\ccf{C}}(\Cint) =
 \on{Coeq}\Bigg(
 \begin{tikzcd}
	{\coprod_{\mathrm{F,G} \in \Cintjr} [\mathrm{F,G}] \cotimes [\mathrm{G,F}]} & {\coprod_{\mathrm{H} \in \Cintjr } [\mathrm{H,H}]}
	\arrow["{\mathsf{k}_{\mathrm{G,F,G}}}", shift left=2, from=1-1, to=1-2]
	\arrow["{\overline{\mathsf{k}}_{\mathrm{G,F,G}}}"', shift right=2, from=1-1, to=1-2]
 \end{tikzcd}
\Bigg)
\]
A particular case is that when $\csym{C}$ is rigid symmetric: then we find a closed structure on $\csym{C}$ by setting $[\mathrm{F,G}] :=\mathrm{F}^{\blackdiamond}\otimes \mathrm{G}$, and the composition map is defined via
\[
 \mathsf{k}_{\mathrm{F,G,H}}: [\mathrm{G,H}]\otimes [\mathrm{F,G}] \xiso [\mathrm{F,G}]\otimes [\mathrm{G,H}]
 = \mathrm{F}^{\blackdiamond} \otimes \mathrm{G} \otimes \mathrm{G}^{\blackdiamond} \otimes \mathrm{H} \xrightarrow{\mathrm{F}^{\blackdiamond} \otimes \varepsilon^{\mathrm{G}} \otimes \mathrm{H}} \mathrm{F}^{\blackdiamond} \otimes \mathbb{1} \otimes \mathrm{H} \xiso \mathrm{F}^{\blackdiamond} \otimes \mathrm{H} = [\mathrm{F,H}].
\]
Further, evaluation maps of this closed structure are defined by
\[
 \on{ev}_{[\mathrm{F,G}]}: \mathrm{F}\otimes [\mathrm{F,G}] = \mathrm{F} \otimes \mathrm{F}^{\blackdiamond}\otimes \mathrm{G} \xrightarrow{\varepsilon^{\mathrm{F}}} \mathrm{G}.
\]

It is easy to verify that the following diagram commutes:
\[\begin{tikzcd}
	{\mathrm{G\otimes G^{\blackdiamond}}\otimes \mathrm{F\otimes F}^{\blackdiamond}} && {\mathrm{F\otimes F^{\blackdiamond}}} \\
	{\mathrm{G}\otimes [\mathrm{G,F}]\otimes \mathrm{F}^{\blackdiamond}} \\
	{\mathrm{F}^{\blackdiamond}\otimes \mathrm{G} \otimes [\mathrm{G,F}]} & {[\mathrm{F,G}]\otimes [\mathrm{G,F}]} & {[\mathrm{G,F}]\otimes [\mathrm{F,G}],}
	\arrow["{=}"', from=1-1, to=2-1]
	\arrow["{\varepsilon^{\mathrm{G}}\otimes \mathrm{F}\otimes\mathrm{F}^{\blackdiamond}}", from=1-1, to=1-3]
	\arrow["{\on{ev}_{[\mathrm{G,F}}\otimes \mathrm{F}^{\blackdiamond}}"', from=2-1, to=1-3]
	\arrow["\simeq"', from=2-1, to=3-1]
	\arrow["{=}"', from=3-1, to=3-2]
	\arrow["{\mathsf{k}_{\mathrm{F,G,F}}}"', from=3-3, to=1-3]
	\arrow["\simeq"', from=3-2, to=3-3]
\end{tikzcd}\]
where the unlabelled isomorphisms are obtained by braiding. Thus, Theorem~\ref{UnitGeneralCase} can be reformulated, in the symmetric case, as follows:
\begin{theorem}
 If $\csym{S}$ is a rigid symmetric semigroup category, then $[\csym{S}^{\on{op}},\mathbf{Vec}_{\Bbbk}]$ is symmetric monoidal, and
 \[
  \mathbb{1}_{[\ccf{S}^{\on{op}},\mathbf{Vec}_{\Bbbk}]} \simeq \int^{\mathrm{F} \in \ccf{S}} \Hom{-,[\mathrm{F,F}]} = \on{Tr}_{\ccf{S}}(\Sint).
 \]
\end{theorem}

We conclude Section~\ref{s3} with the following conjecture:
\begin{conjecture}\label{ClosedConjecture}
 Let $\csym{S}$ be a closed symmetric semigroup cateogry. Then $[\csym{S}^{\on{op}},\mathbf{Vec}_{\Bbbk}]$ is symmetric monoidal, and
 \[
  \mathbb{1}_{[\ccf{S}^{\on{op}},\mathbf{Vec}_{\Bbbk}]} \simeq \on{Tr}_{\ccf{S}}(\Sint).
 \]
\end{conjecture}
Observe that Conjecture~\ref{ClosedConjecture} is satisfied for the category $\csym{S} = A\!\on{-proj}$ of projective modules over a finite-dimensional commutative algebra $A$, where $\on{Tr}_{\ccf{S}}(\Sint) = \on{Hom}_{A}(A,A)/\left\langle ab - ba\right\rangle = A$.

\section{Finite tensor categories and simple rigid semigroup categories}\label{s4}

In this section, we assume the ground field $\Bbbk$ to be algebraically closed.

\begin{definition}
 Let $\mathcal{A}$ be a $\Bbbk$-linear category. We say that $\mathcal{A}$ is {\it finitary} if there is a finite-dimensional $\Bbbk$-algebra such that $\mathcal{A} \simeq A\!\on{-proj}$. We say that $\mathcal{A}$ is {\it finite abelian}, if there is a finite-dimensional $\Bbbk$-algebra $A$ such that $\mathcal{A} \simeq A\!\on{-mod}$.

 In particular, the category of projective objects of a finite abelian category is finitary.
\end{definition}

Following the terminology schema of \cite{MMMTZ1}, we define the following:
\begin{definition}
 We say that a $\Bbbk$-linear semigroup category is {\it quasi-fias} if it is finitary and rigid.
\end{definition}

\begin{definition}
 A {\it finite tensor category} is a rigid, monoidal, finite abelian $\Bbbk$-linear category $\mathcal{A}$ such that $\on{End}_{\mathcal{A}}(\mathbb{1}) \simeq \Bbbk$.
\end{definition}

\begin{definition}
 Let $\csym{S}$ be a semigroup category and let $\mathbf{M}$ be an $\csym{S}$-module category. An {\it $\csym{S}$-stable ideal} in $\mathbf{M}$ is an ideal $\mathbf{I}$ in the category $\mathbf{M}$ such that for all $\mathrm{F} \in \csym{S}$ and $g \in \mathbf{I}$, we have $\mathrm{F} \star g \in \mathbf{I}$.

 An $\csym{S}$-stable ideal $\mathbf{I}$ is said to be {\it non-trivial} if $\mathbf{I} \neq 0$ and $\mathbf{I} \neq \mathbf{M}$.
 An $\csym{S}$-stable ideal $\mathbf{I}$ is said to be {\it thick} if it is of the form $\mathbf{I} = \left\langle \on{id}_{X} \; | \; X \in \mathtt{I} \right\rangle$ for some collection $\mathtt{I} \subseteq \on{Ob}\mathbf{M}$.

 An $\csym{S}$-module category $\mathbf{M}$ is said to be {\it transitive} if it admits no non-trivial thick ideals, and it is said to be {\it simple transitive} if it admits no non-trivial ideals.

\end{definition}

\begin{definition}[{\cite[Definition~7.5.1]{EGNO}}]
  If $\csym{S}$ is an abelian semigroup category with enough projectives, and $\mathbf{M}$ is an abelian $\csym{S}$-module category with enough projectives, we say that $\mathbf{M}$ is an {\it exact $\csym{S}$-module category} if for any projective object $\mathrm{F} \in \csym{S}$ and any object $X \in \mathbf{M}$, the object $\mathrm{F} \star X$ is projective.
\end{definition}

Since finitary categories often feature as categories of projectives, we modify the above notion accordingly:
\begin{definition}
 If $\csym{S}$ is a finitary semigroup category, and $\mathbf{M}$ is an abelian $\csym{S}$-module category with enough projectives, we say that $\mathbf{M}$ is a {\it projectivizing $\csym{S}$-module category} if for any object $\mathrm{F} \in \csym{S}$ and any object $X \in \mathbf{M}$, the object $\mathrm{F} \star X$ is projective.
\end{definition}

\subsection{Auxiliary lemmata}\footnote[1]{The allusion is to \cite[Section~5.3]{MM5}} \label{41}

\begin{lemma}[{\cite[Lemma~13]{MM5}}]
 Let $\mathcal{A}$ be a finite abelian category and let $\euler{F}$ be an exact endofunctor of $\mathcal{A}$ such that, for every object $X \in \mathcal{A}$, the object $\euler{F}(X)$ is projective. Let $A$ be a finite-dimensional $\Bbbk$-algebra such that there is an equivalence $\mathcal{A} \xrightarrow[\simeq]{\Phi} A\!\on{-mod}$. There is a projective $A$-$A$-bimodule $V \in \on{add}\setj{A \kotimes A}$ such that the following diagram commutes up to natural isomorphism:
\[\begin{tikzcd}
	{\mathcal{A}} && {\mathcal{A}} \\
	{A\!\on{-mod}} && {A\!\on{-mod}}
	\arrow["{\euler{F}}", from=1-1, to=1-3]
	\arrow["\Phi"', from=1-3, to=2-3]
	\arrow["\simeq", from=1-3, to=2-3]
	\arrow["\Phi", from=1-1, to=2-1]
	\arrow[""{name=0, anchor=center, inner sep=0}, "\simeq"', from=1-1, to=2-1]
	\arrow["{{}_{A}V_{A}\otimes_{A}-}"', from=2-1, to=2-3]
	\arrow["\simeq", shift left=3, shorten <=13pt, Rightarrow, from=0, to=2-3]
\end{tikzcd}\]
\end{lemma}

Our next aim is to give an account of \cite[Lemma~12]{MM5}, or, rather, its generalization \cite[Theorem~11]{KMMZ}. Although we do not give a generalization thereof beyond what can be concluded from the proof of \cite[Theorem~11]{KMMZ}, we hope to state it in a manner as open for further generalization as possible, since, as the reader will see in the remainder of this section, this is the main reason for the finiteness conditions required in Theorem~\ref{FinitaryCharacterization}.

 Given an additive category $\mathcal{C}$, we denote its split Grothendieck group by $\got{\mathcal{C}}$. Given a ring $R$, we denote the base change $R \zotimes \got{\mathcal{C}}$ by $\got{\mathcal{C}}_{R}$. We denote by $\on{Indec}	\mathcal{C}$ the set of isomorphism classes of indecomposable objects in $\mathcal{C}$, and by $\got{\on{Indec}\mathcal{C}}$ its image in $\got{\mathcal{C}}$; if $\mathcal{C}$ is Krull-Schmidt, then $\got{\mathcal{C}}$ is free and $\got{\on{Indec}\mathcal{C}}$ is a basis therein.

 Setting $\mathcal{A} \overline{\otimes}\mathcal{B} = (\mathcal{A} \zotimes \mathcal{B})^{\oplus}$, where $(-)^{\oplus}$ denotes the additive envelope, endows the category $\mathbf{Cat}_{\oplus}$ of additive categories and preadditive functors with a monoidal structure, and taking split Grothendieck groups defines a monoidal functor $\got{-}: (\mathbf{Cat}_{\oplus},\overline{\otimes}) \rightarrow (\mathbf{Ab},\zotimes)$.
 Thus, if $\csym{C}\in \mathbf{Cat}_{\oplus}$ is monoidal, then $\got{\csym{C}}$ is a ring; if $\csym{S}$ is an additive semigroup category, then $\got{\csym{S}}$ is a nonunital ring; if $\mathbf{M}$ is an $\csym{S}$- respectively $\csym{C}$-module category, then $\got{\mathbf{M}}$ is an $\got{\csym{S}}$-module respectively a unital $\got{\csym{C}}$-module. If $\euler{F}: \mathcal{A} \rightarrow \mathcal{B}$ is a preadditive functor, then $\got{\euler{F}}$ is a group homomorphism.

 \begin{lemma}
  Let $\mathcal{C}$ be an abelian Krull-Schmidt category all of whose objects admit projective covers. Then taking projective covers defines a group homomorphism $\mathfrak{c}: \got{\mathcal{C}} \rightarrow \got{\mathcal{C}\!\on{-proj}}$.
 \end{lemma}

 \begin{proof}
  This is an immediate consequence of projective covers being additive: the projective cover of a direct sum is the direct sum of respective projective covers.
 \end{proof}

 We denote the cone in $\got{\mathcal{C}}_{\mathbb{R}}$ given by the non-negative span of $\got{\on{Indec}\mathcal{C}}$ by $\got{\mathcal{C}}_{\geq 0}$. The cone induces a partial ordering on $\got{\mathcal{C}}_{\mathbb{R}}$ by setting $u \geq v$ if $u - v \in \got{\mathcal{C}}_{\geq 0}$. A further partial ordering is induced on $\on{Hom}_{\mathbb{R}}(\gotr{\mathbf{C}},\gotr{\mathbf{C}})$, by $\mathbf{f} \leq \mathbf{g}$ if $\mathbf{f}(u) - \mathbf{g}(u) \in \got{\mathcal{C}}_{\geq 0}$ for all $u \in \got{\mathcal{C}}_{\geq 0}$. Assume that $\mathcal{C}$ is Krull-Schmidt, and let $v = \sum_{\mathrm{X} \in \on{Indec}\mathcal{C}} \lambda_{X}\got{X} \in \got{\mathcal{C}}_{\geq 0}$, where $\lambda_{X} \in \mathbb{R}_{\geq 0}$ is zero for all but finitely many $X \in \on{Indec}\mathcal{C}$. We denote by $\on{Supp}(v)$ the set $\setj{X \in \on{Indec}\mathcal{C} \; | \; \lambda_{X} \neq 0}$.

 \begin{lemma}\label{Inequalities}
  Let $\csym{S}$ be a rigid semigroup category and let $\mathbf{M}$ be an abelian Krull-Schmidt $\csym{S}$-module category, respecting adjunctions in $\csym{S}$, and such that all of the objects of $\mathbf{M}$ admit projective covers. The map $\mathfrak{c}_{\mathbb{R}} \in \on{End}_{\mathbb{R}}(\gotr{\mathbf{M}})$ is then $\csym{S}$-subhomogeneous, i.e. for any $\mathrm{F} \in \csym{S}$, we have $\mathfrak{c}_{\mathbb{R}} \circ \gotr{\mathbf{M}\mathrm{F}} \leq \gotr{\mathbf{M}\mathrm{F}}\circ \mathfrak{c}$.
 \end{lemma}

 \begin{proof}
  Since $\mathbf{M}$ respects adjuctions in $\csym{S}$, its full subcategory $\mathbf{M}\!\on{-proj}$ is an $\csym{S}$-module category, as an exact functor with a right adjoint sends projective objects to projective objects. Further, if $P \xtwoheadrightarrow{p} X$ is a projective cover in $\mathbf{M}$, then, for any $\mathrm{F} \in \csym{S}$, the morphism $\mathrm{F}\star P \xtwoheadrightarrow{\mathrm{F}\star p} \mathrm{F}\star X$ is an epimorphism from a projective. But then the projective cover of $\mathrm{F}\star X$ is a direct summand of $\mathrm{F}\star P$, by \cite[Lemma~5.6]{ASS}.

  Hence $\mathfrak{c}(\got{\mathbf{M}\mathrm{F}}(\got{X})) = \mathfrak{c}(\got{\mathrm{F}\star X}) \leq \got{\mathrm{F}\star P} = \got{\mathbf{M}\mathrm{F}}(\mathfrak{c}(\got{X}))$, which proves the claim.
 \end{proof}

 \begin{definition}\label{KMIdempotents}
  Let $\csym{S}$ be a Krull-Schmidt semigroup category and let $\mathbf{M}$ be a Krull-Schmidt $\csym{S}$-module category.
  Denote by $\got{\mathbf{e}\star \mathbf{M}}_{\geq 0}$ the non-negative span of $\on{Indec}(\on{add}\setj{\mathrm{F} \star X \; | \; \mathrm{F} \in \on{Supp}(\mathbf{e}), X \in \mathbf{M}})$.

  We say that an element $\mathbf{e}$ of $\got{\csym{S}}_{\mathbb{R}}$ is a {\it KM-idempotent for $\mathbf{M}$} if
  \begin{enumerate}
   \item $\mathbf{e} \in \got{\csym{S}}_{\geq 0}$;
   \item $\mathbf{M}_{\mathbb{R}}(\mathbf{e})$ is idempotent;
   \item \label{SelfTranspose} $\mathbf{M}_{\mathbb{R}}(\mathbf{e})$ is injective on $\got{\mathbf{e}\star \mathbf{M}}_{\geq 0}$.
  \end{enumerate}
 \end{definition}

 \begin{lemma}\label{CoverPreservation}
  Let $\csym{S}$ be a Krull-Schmidt rigid semigroup category and let $\mathbf{M}$ be a Krull-Schmidt $\csym{S}$-module category.
  If $\mathbf{e}$ is a KM-idempotent for $\mathbf{M}$, then, for any $\mathrm{F},\mathrm{G} \in \on{Supp}(\mathbf{e})$ and any $X \in \mathbf{M}$, the projective cover of $\mathrm{GF}\star X$ can be obtained as $\mathrm{G} \star P$, where $P$ is the projective cover of $\mathrm{F}\star X$.
 \end{lemma}

 \begin{proof}
  By Lemma~\ref{Inequalities}, it suffices to show that $\gotr{\mathbf{M}\mathrm{G}} \circ \mathfrak{c}\circ \gotr{\mathbf{M}\mathrm{F}}\leq \mathfrak{c}\circ \gotr{\mathbf{M}\mathrm{GF}}$.
 We have
 \begin{equation}\label{MoreInequalities}
 \mathfrak{c} \circ \mathbf{M}_{\mathbb{R}}(\mathbf{e}) = \mathfrak{c}\circ \mathbf{M}_{\mathbb{R}}(\mathbf{e}) \circ \mathbf{M}_{\mathbb{R}}(\mathbf{e}) \leq \mathbf{M}_{\mathbb{R}}(\mathbf{e})\circ \mathfrak{c} \circ \mathbf{M}_{\mathbb{R}}(\mathbf{e}),
 \end{equation}
  where the first equality is by the idempotency of $\mathbf{M}_{\mathbb{R}}(\mathbf{e})$, and the latter inequality follows by Lemma~\ref{Inequalities}.

  Clearly, $\on{Im}(\mathbf{M}_{\mathbb{R}}(\mathbf{e})\circ \mathfrak{c} \circ \mathbf{M}_{\mathbb{R}}(\mathbf{e})) \subseteq \got{\mathbf{e}\star \mathbf{M}}_{\geq 0}$. Since $\got{\mathbf{e}\star \mathbf{M}}_{\geq 0}$ is a subcone of $\got{\mathbf{M}}_{\geq 0}$, the inequality~\eqref{MoreInequalities} implies that $\on{Im}\big(\mathbf{M}_{\mathbb{R}}(\mathbf{e})\circ \mathfrak{c} \circ \mathbf{M}_{\mathbb{R}}(\mathbf{e})-\mathfrak{c} \circ \mathbf{M}_{\mathbb{R}}(\mathbf{e})\big) \subseteq \got{\mathbf{e}\star \mathbf{M}}_{\geq 0}$. Using Condition~\eqref{SelfTranspose} in Definition~\ref{KMIdempotents}, we find that if $\mathbf{M}_{\mathbb{R}}(\mathbf{e})\circ \mathfrak{c} \circ \mathbf{M}_{\mathbb{R}}(\mathbf{e})-\mathfrak{c} \circ \mathbf{M}_{\mathbb{R}}(\mathbf{e}) \neq 0$, then also
  \[
   \mathbf{M}_{\mathbb{R}}(\mathbf{e})\circ \big(\mathbf{M}_{\mathbb{R}}(\mathbf{e})\circ \mathfrak{c} \circ \mathbf{M}_{\mathbb{R}}(\mathbf{e})-\mathfrak{c} \circ \mathbf{M}_{\mathbb{R}}(\mathbf{e})\big) \neq 0,
  \]
  which is a contradiction, showing that
  \begin{equation}\label{BeforeBasis}
   \mathbf{M}_{\mathbb{R}}(\mathbf{e}) \circ \mathfrak{c} \circ \mathbf{M}_{\mathbb{R}}(\mathbf{e}) = \mathfrak{c} \circ \mathbf{M}_{\mathbb{R}}(\mathbf{e}) = \mathfrak{c} \circ \mathbf{M}_{\mathbb{R}}(\mathbf{e}) \circ \mathbf{M}_{\mathbb{R}}(\mathbf{e}).
  \end{equation}
 If we now write $\mathbf{e} = \sum_{i = 1}^{n} \lambda_{i} \gotr{\mathrm{F}_{i}}$, for $\mathrm{F}_{i} \in \csym{S}$ and $\lambda_{i} \in \mathbb{R}_{\geq 0}$, Equation~\eqref{BeforeBasis} becomes
 \begin{equation}\label{FinalSteps}
  \sum_{i,j} \lambda_{i}\lambda_{j} \gotr{\mathbf{M}\mathrm{F}_{i}}\circ \mathfrak{c} \circ \gotr{\mathbf{M}\mathrm{F}_{j}} = \sum_{i,j} \lambda_{i}\lambda_{j} \mathfrak{c} \circ \gotr{\mathbf{M}\mathrm{F}_{i}}\circ \gotr{\mathbf{M}\mathrm{F}_{j}}.
 \end{equation}
 From Lemma~\ref{Inequalities}, we know that, for all $i,j$, we have
 \begin{equation}\label{Ininequalities}
 \mathfrak{c} \circ \got{\mathbf{M}\mathrm{F}_{i}}\circ  \got{\mathbf{M}\mathrm{F}_{j}} \leq \got{\mathbf{M}\mathrm{F}_{i}} \circ \mathfrak{c} \circ  \got{\mathbf{M}\mathrm{F}_{j}}.
 \end{equation}
 But since Equation~\eqref{FinalSteps} holds, all of the inequalities in Equation~\eqref{Ininequalities} must be equalities.
 \end{proof}

 Given a ring $R$, we denote by $\mathfrak{J}$ the Jacobson radical of $R$. We now recall some of the facts about the radical of an additive category, the proofs of which can be found in \cite{Kr}:

 \begin{lemma}
  Let $\mathcal{C}$ be an additive category. There is a two-sided ideal $\on{Rad}_{\mathcal{C}}$ in $\mathcal{C}$, defined by setting
  \[
\on{Rad}_{\mathcal{C}}(X,Y) = \setj{f:X\rightarrow Y \mid gf \in \mathfrak{J}(\on{End}_{\mathcal{C}}(Y)) \text{ for all } g \in \on{Hom}_{\mathcal{C}}(Y,X)}
  \]
  and $\on{Rad}_{\mathcal{C}}$ is the unique two-sided ideal in $\mathcal{C}$ satisfying $\mathcal{I}(X,X) = \mathfrak{J}(\on{End}_{\mathcal{C}}(X))$ for all $X \in \mathcal{C}$.
 \end{lemma}

 We have not found the proof of the following elementary observation in the literature, and thus spell it out below:

 \begin{lemma}\label{RadicalCover}
  Let $\mathcal{C}$ be a Krull-Schmidt abelian category, and assume that $P_{1} \xrightarrow{p_{1}} P_{0} \xtwoheadrightarrow{p_{0}} X \rightarrow 0$ is an exact sequence, where $p_{0}$ is a projective cover. Then $p_{1} \in \on{Rad}_{\mathcal{C}}(P_{1},P_{0})$.
 \end{lemma}

 \begin{proof}
  Assume $\alpha \not\in \on{Rad}_{\mathcal{C}}(P_{1},P_{0})$. Then there is $\varphi \in \on{Hom}_{\mathcal{C}}(P_{0},P_{1})$ such that $\alpha \circ \varphi \not\in \mathfrak{J}(\on{End}_{\mathcal{C}}(P_{0})) = \on{Rad}_{\mathcal{C}}(P_{0},P_{0})$.
  Using the Krull-Schmidt property, write $P_{0} = \bigoplus_{i=1}^{n} Q_{i}$, where $Q_{i}$ is indecomposable for $i=1,\ldots,n$. We then find $\on{id}_{P_{0}} = \sum_{i=1}^{n} \iota_{i} \circ \pi_{i}$, where $\iota_{i}$ and $\pi_{i}$ are the split monomorphisms and epimorphisms corresponding to the above decomposition. Thus, $\alpha \circ \varphi = \sum_{i=1}^{n} \iota_{i} \circ \pi_{i} \circ \alpha \circ \varphi \not\in \on{Rad}_{\mathcal{C}}(P_{0},P_{0})$.
  Hence, there is some $j \in \setj{1,\ldots,n}$ such that $\iota_{j} \circ \pi_{j} \circ \alpha \circ \varphi \not\in \on{Rad}_{\mathcal{C}}(P_{0},P_{0})$. If all the terms were radical, then the sum would also be radical. It follows that $\pi_{j} \circ \alpha \circ \varphi \not\in \on{Rad}_{\mathcal{C}}(P_{0},Q_{j})$ - otherwise, postcomposing with $\iota_{j}$ would render $\iota_{j} \circ \pi_{j} \circ \alpha \circ \varphi \not\in \on{Rad}_{\mathcal{C}}(P_{0},P_{0})$ radical.

  We now claim that $\pi_{j} \circ \alpha \circ \varphi$ is an epimorphism. If it is not, then neither is the endomorphism $\pi_{j} \circ \alpha \circ \varphi \circ \iota_{j}$ of the indecomposable object $Q_{i}$. But, since $\on{End}_{\mathcal{C}}(Q_{i})$ is local, its Jacobson radical consists of the non-units, and so $\pi_{j} \circ \alpha \circ \varphi \circ \iota_{j} \in \on{Rad}_{\mathcal{C}}(Q_{i},Q_{i})$, which is a contradiction.

  We further infer that $\pi_{j} \circ \alpha$ is an epimorphism, and so there is a natural embedding $Q_{j} \hookrightarrow \on{Im}p_{1} = \on{Ker}p_{0}$. This shows that $p_{0} \circ \iota_{j} = 0$. But then, the composite morphism $\bigoplus_{i\neq j} Q_{i} \xrightarrow{(\iota_{i})_{i\neq j}} P_{0} \xtwoheadrightarrow{p_{0}} X$ is an epimorphism, while $(\iota_{i})_{i\neq j}$ is not, which shows that $p_{0}$ is not a projective cover, yielding a contradiction.
 \end{proof}

 \begin{proposition}\label{MiniMM5}
 Let $\csym{S}$ be a rigid semigroup category and let $\mathbf{M}$ be an abelian Krull-Schmidt $\csym{S}$-module category with enough projectives, and such that, for all $\mathrm{F} \in \csym{S}$ and $X \in \mathbf{M}$, the object $\mathrm{F}\star X$ admits a projective cover. Assume that
 \begin{itemize}
  \item $\mathbf{M}\!\on{-proj}$ is a simple transitive $\csym{S}$-module category;
  \item there is a KM-idempotent $\mathbf{e}$ for $\mathbf{M}$ such that for any $\mathrm{F} \in \csym{S}$, there is $\mathrm{F}' \in \on{Supp}(\mathbf{e})$ such that $\mathbf{M}\mathrm{F} \simeq \mathbf{M}\mathrm{F}'$.
 \end{itemize}
Then $\mathbf{M}$ is a projectivizing module category.
 \end{proposition}

 \begin{proof}
 Let $X$ and let $\mathrm{F} \in \csym{S}$. We want to show that $\mathbf{M}\mathrm{F}(X)$ is projective. Let $\mathrm{F}' \in \on{Supp}(\mathbf{e})$ be such that $\mathbf{M}\mathrm{F} \simeq \mathbf{M}\mathrm{F}'$. Using the assumptions on $\mathbf{M}$, there is an exact sequence $P_{1} \xrightarrow{\alpha} P_{0} \xtwoheadrightarrow{\pi} \mathbf{M}\mathrm{F}'(X) \rightarrow 0$ such that $P_{1},P_{0}$ are projective and $\pi$ is a projective cover. By Lemma~\ref{CoverPreservation}, for any $\mathrm{G} \in \csym{S}$, the sequence
\[\begin{tikzcd}
	{\mathbf{M}\mathrm{G}(P_{1})} & {\mathbf{M}\mathrm{G}(P_{0})} & {\mathbf{M}(\mathrm{GF'})(X)} & 0
	\arrow["{\mathbf{M}\mathrm{G}(\alpha)}", from=1-1, to=1-2]
	\arrow["{\mathbf{M}\mathrm{G}(\pi)}", two heads, from=1-2, to=1-3]
	\arrow[from=1-3, to=1-4]
\end{tikzcd}\]
has the same properties.
 By Lemma~\ref{RadicalCover}, both $\alpha$ and $\mathbf{M}\mathrm{G}(\alpha)$ are radical. Thus, $\setj{\alpha}\sqcup\setj{\mathbf{M}\mathrm{G}(\alpha) \mid \mathrm{G} \in \csym{S}} \subseteq \on{Rad}_{\mathbf{M}\!\on{-proj}}$ and thus the $\csym{S}$-stable ideal $\left\langle \alpha \right\rangle$ generated by $\alpha$ is contained in $\on{Rad}_{\mathbf{M}\!\on{-proj}}$. Thus, since $\mathbf{M}\!\on{-proj}$ is simple transitive, $\left\langle\alpha \right\rangle = 0$, which implies that $\alpha = 0$. Thus, $\mathbf{M}\mathrm{F}'(X) \simeq \mathbf{M}\mathrm{F}(X)$ is projective, as we wanted to show.
 \end{proof}

\begin{definition}
 We say that a semigroup category $\csym{S}$ is {\it $\mathcal{J}$-cell-trivial} if, for any $\mathrm{F,G} \in \csym{S}$, there are $\mathrm{H,K} \in \csym{S}$ such that $\mathrm{F}$ is a direct summand of $\mathrm{H}\otimes \mathrm{G} \otimes \mathrm{K}$.
\end{definition}

In particular, if $\csym{S}$ is a $\mathcal{J}$-cell trivial semigroup category, then both the left regular $\csym{S}$-module category ${}_{\ccf{S}}\csym{S}$ and the right regular $\csym{S}$-module category $\csym{S}_{\ccf{S}}$ are transitive.

\begin{lemma}[{\cite[Proposition~18(i)]{KM}}]\label{PerronFrobenius}
 Let $\csym{S}$ be a $\mathcal{J}$-cell-trivial quasi-fias semigroup category. There is an idempotent $\mathbf{e} \in \got{\csym{S}}_{\geq 0}$ such that $\on{Supp}(\mathbf{e}) = \on{Indec}\csym{S}$. It is given by $\lim\limits_{n \rightarrow \infty} \frac{\got{\bigoplus_{\mathrm{F} \in \on{Indec}\cccsym{S}} \mathrm{F}}^{n}}{\lambda^{n}}$, where $\lambda$ is the Perron-Frobenius eigenvalue of $\got{\bigoplus_{\mathrm{F} \in \on{Indec}\cccsym{S}} \mathrm{F}}$.

 The idempotent $\mathbf{e}$ is a KM-idempotent for any transitive Krull-Schmidt $\csym{S}$-module category $\mathbf{M}$.
\end{lemma}

\begin{proof}
 The first statement is \cite[Proposition~18(i)]{KM}. For the latter, observe that condition (i) of Definition~\ref{KMIdempotents} is satisfied by assumption. Condition (ii) holds since $\got{\mathbf{M}(-)}: \got{\csym{S}} \rightarrow \on{End}_{\mathbf{Ab}}(\got{\mathbf{M}})$ is a ring homomorphism, and thus sends idempotents to idempotents. Finally, condition (iii) holds since for any $X \in \on{Indec}\mathbf{M}$, the object $\bigoplus_{\mathrm{F} \in \on{Indec}\ccf{S}} \mathrm{F} \star X$ is not zero.
\end{proof}

\begin{theorem}[{\cite[Proposition~18(i)]{KM}}]\label{MazorchukMiemietzLemma}
 Let $\csym{S}$ be a $\mathcal{J}$-cell-trivial quasi-fias semigroup category. Let $\mathbf{M}$ be a finitary simple transitive $\csym{S}$-module category. The $\csym{S}$-module category $[\mathbf{M}^{\on{op}},\mathbf{vec}_{\Bbbk}]$, endowed with $\csym{S}$-module category structure described in Definition~\ref{MV}, is projectivizing.
\end{theorem}

\begin{proof}
 The category $[\mathbf{M}^{\on{op}},\mathbf{vec}_{\Bbbk}]$ is finite abelian, and thus it has enough projectives, and all of its objects admit projective covers.
 Since $[\mathbf{M}^{\on{op}},\mathbf{vec}_{\Bbbk}]\!\on{-proj} \simeq \mathbf{M}$, the first condition in Proposition~\ref{MiniMM5} is satisfied by $[\mathbf{M}^{\on{op}},\mathbf{vec}_{\Bbbk}]$. The latter condition is satisfied by Lemma~\ref{PerronFrobenius}, and so the result follows from Proposition~\ref{MiniMM5}.
\end{proof}

We now state some analogues of some of the results of \cite{BEO} on splitting ideals:
\begin{definition}[{\cite[Proposition~2.12]{BEO}}]
 Given a quasi-fias semigroup category $\csym{S}$, the bounded above homotopy category $\mathcal{K}^{-}(\csym{S})$ has a natural structure of a semigroup category, defined analogously to the monoidal case. Transporting the semigroup category structure along the equivalence $\mathcal{K}^{-}(\csym{S}) \simeq \mathcal{D}^{-}([\csym{S}^{\on{op}},\mathbf{vec}_{\Bbbk}])$ of the bounded below homotopy category of $\csym{S}$ with the bounded below derived category of $[\csym{S}^{\on{op}},\mathbf{vec}_{\Bbbk}]$ endows the latter with a semigroup category structure.
\end{definition}

Similarly to \cite{BEO}, given $\mathrm{X,Y} \in [\csym{S}^{\on{op}},\mathbf{vec}_{\Bbbk}]$, we write $X \boxtimes^{\mathbf{L}} Y$ for the object of $\mathcal{D}^{-}([\csym{S}^{\on{op}},\mathbf{vec}_{\Bbbk}])$ given by the tensor product $\mathrm{P}_{\bullet} \boxtimes \mathrm{Q}_{\bullet}$ of the projective resolutions of $\mathrm{X}$ and $\mathrm{Y}$ respectively. For $i \in \mathbb{Z}_{\geq 0}$, we write $\mathrm{X} \boxtimes^{\mathbf{L}_{i}} Y$ for the object of $[\csym{S}^{\on{op}},\mathbf{vec}_{\Bbbk}]$ given by $\mathtt{H}^{i}(\mathrm{X} \boxtimes \mathrm{Q}_{\bullet})$. The following is our analogue of \cite[Lemma~2.15(ii), Corollary~2.16]{BEO}:
\begin{lemma}\label{BEO215}
 Let $\mathrm{P} \in \csym{S}$ and $\mathrm{X,Y} \in [\csym{S}^{\on{op}},\mathbf{vec}_{\Bbbk}]$.
 \begin{enumerate}
  \item $- \boxtimes -: [\csym{S}^{\on{op}},\mathbf{vec}_{\Bbbk}] \kotimes [\csym{S}^{\on{op}},\mathbf{vec}_{\Bbbk}] \rightarrow [\csym{S}^{\on{op}},\mathbf{vec}_{\Bbbk}]$ is right exact in both variables;
  \item if $\mathrm{X}$ is projective, then the functors $\mathrm{X} \boxtimes -$ and $- \boxtimes \mathrm{X}$ are exact;
  \item $\euler{H}^{i}(\mathrm{P} \boxtimes \mathrm{Y}) = 0 = \euler{H}^{i}(\mathrm{Y} \boxtimes \mathrm{P})$ for $i > 0$;
  \item $\mathrm{X} \boxtimes^{\mathbf{L}_{i}} \mathrm{Y} \simeq \euler{H}^{i}(\mathrm{X} \boxtimes^{\euler{L}} \mathrm{Y})$ for $i \geq 0$.
 \end{enumerate}
\end{lemma}

\begin{proof}
 The first statement follows from the fact that $\boxtimes$ is defined via Day convolution, and $[\csym{S}^{\on{op}},\mathbf{vec}_{\Bbbk}]$ is the free finite cocompletion of $\mathbf{M}$. Thus, this is a slight modification of Theorem~\ref{UniversalPropertyDay}, as described in \cite[Section~10]{St}.
 The second statement follows from the monoidal equivalence $[\csym{S}^{\on{op}},\mathbf{vec}_{\Bbbk}]\!\on{-proj} \simeq \csym{S}$, combined with the rigidity of $\csym{S}$. This shows that if $\mathrm{X}$ is projective, then $\mathrm{X} \boxtimes -$ and $- \boxtimes \mathrm{X}$ admit left and right adjoints, yielding the exactness. The latter two statements follow directly from the first two.
\end{proof}

\subsection{Finite tensor categories are abelianizations of disimple quasi-fias categories}\label{42}

\begin{definition}
 A quasi-fias category $\csym{S}$ is {\it disimple} if both the left $\csym{S}$-module category ${}_{\ccf{S}}\csym{S}$ and the right $\csym{S}$-module category $\csym{S}_{\ccf{S}}$ are simple transitive.
\end{definition}

In particular, a disimple quasi-fias category $\csym{S}$ is $\mathcal{J}$-cell trivial.

\begin{proposition}[{\cite[Proposition~4.2.12]{EGNO}}]\label{EGNOHelpsAgain}
 Let $\csym{C}$ be a finite tensor category. Let $\mathrm{P} \in \csym{C}\!\on{-proj}$, and let $\mathrm{X} \in \csym{C}$. Then $\mathrm{P} \otimes \mathrm{X}, \mathrm{X} \otimes \mathrm{P} \in \csym{C}\!\on{-proj}$.
 Similarly, if $\mathrm{I} \in \csym{C}\!\on{-inj}$, then $\mathrm{I} \otimes \mathrm{X},\, \mathrm{X} \otimes \mathrm{I} \in \csym{C}\!\on{-inj}$.
\end{proposition}

\begin{corollary}
 Let $\csym{C}$ be a finite tensor category. Then $\csym{C}\!\on{-proj}$ is a semigroup category, and $\csym{C}$ is a projectivizing semigroup module category over $\csym{C}\!\on{-proj}$.
\end{corollary}

\begin{proposition}\label{FaithfulUnits}
 Let $\csym{C}$ be a finite tensor category and let $\mathrm{P} \in \csym{C}\!\on{-proj}$ be a non-zero projective object of $\csym{C}$. Then the unit morphism $\eta^{\mathrm{P}}: \mathbb{1}_{\ccf{C}} \rightarrow \mathrm{P}^{\blackdiamond} \otimes \mathrm{P}$ is a monomorphism. Thus, the functors $\mathrm{P} \otimes -: \csym{C} \rightarrow \csym{C}$ and $- \otimes \mathrm{P}: \csym{C} \rightarrow \csym{C}$ are faithful, and the functors $\mathrm{P} \otimes -: \csym{C}\!\on{-proj} \rightarrow \csym{C}\!\on{-proj}$ and $-\otimes \mathrm{P}: \csym{C}\!\on{-proj} \rightarrow \csym{C}\!\on{-proj}$ are liberal.
\end{proposition}

\begin{proof}
 That $\eta^{\mathrm{P}}$ is a monomorphism follows from the fact that $\mathbb{1}_{\ccf{C}}$ is a simple object of $\csym{C}$ (\cite[Theorem~4.3.8]{EGNO}), and that $\mathrm{P} \neq 0$, so $\eta^{\mathrm{P}} \neq 0$.

 By exactness of the tensor product in $\csym{C}$, the transformation $\eta \otimes -: \mathbb{1}_{\ccf{C}} \otimes - \rightarrow (\mathrm{P}^{\blackdiamond} \otimes \mathrm{P}) \otimes -$ is a monomorphism, which is equivalent to $\mathrm{P} \otimes -: \csym{C} \rightarrow \csym{C}$ being faithful. By Proposition~\ref{SufficientCondition}, this implies that $\mathrm{P} \otimes -: \csym{C}\!\on{-proj} \rightarrow \csym{C}\!\on{-proj}$ is liberal. The result for $- \otimes \mathrm{P}$ follows similarly.

\end{proof}

\begin{corollary}[{\cite[Proposition~6.1.3]{EGNO}}] \label{SelfInjective}
 A finite tensor category is self-injective, that is, $\csym{C}\!\on{-proj} = \csym{C}\!\on{-inj}$.
\end{corollary}

\begin{proof}
 Let $\mathrm{P} \in \csym{C}\!\on{-proj}$. Then, following Remark~\ref{SplitMono}, $\mathrm{P}$ is a direct summand of $\mathrm{P} \otimes \mathrm{P}^{\blackdiamond} \otimes \mathrm{P}$. First, we have $\on{Hom}_{\ccf{C}}(-,\mathrm{P}^{\blackdiamond}) \simeq \on{Hom}_{\ccf{C}}(\mathrm{P},{}^{\blackdiamond}\!(-))$, so $\mathrm{P}^{\blackdiamond} \in \csym{C}\!\on{-inj}$. From Lemma~\ref{EGNOHelpsAgain} it follows that $\mathrm{P} \otimes \mathrm{P}^{\blackdiamond} \otimes \mathrm{P} \in \csym{C}\!\on{-inj}$. But, following Remark~\ref{SplitMono}, $\mathrm{P}$ is a direct summand of $\mathrm{P} \otimes \mathrm{P}^{\blackdiamond} \otimes \mathrm{P}$, so $\mathrm{P} \in \csym{C}\!\on{-inj}$. The proof that every injective is projective is similar.
\end{proof}

\begin{proposition}\label{ProjDis}
 Let $\csym{C}$ be a finite tensor category. The category $\csym{C}\!\on{-proj}$ of projective objects in $\csym{C}$ is a disimple quasi-fias category.
\end{proposition}

\begin{proof}
 Since $\csym{C}$ is a finite abelian category, $\csym{C}\!\on{-proj}$ is finitary.

 A non-trivial right $(\csym{C}\!\on{-proj})$-ideal $\mathbf{I}$ in $\csym{C}\!\on{-proj}$ cannot contain the identity morphism of a non-zero object: indeed, if $\on{id}_{\mathrm{P}} \in \mathbf{I}$, then $\on{id}_{\mathrm{P'}} \in \mathbf{I}$, for any direct summand $\mathrm{P}'$ of an object of the form $\mathrm{P} \otimes \mathrm{Q}$, for $\mathrm{Q} \in \csym{C}\!\on{-proj}$. By Proposition~\ref{FaithfulUnits}, this implies that $\on{id}_{\mathrm{P'}} \in \mathbf{I}$ for all $\mathrm{P}' \in \csym{C}\!\on{-proj}$.

 Now let $\mathrm{f} \in \mathbf{I}(\mathrm{P,Q})$ be an arbitrary non-zero morphism in $\mathbf{I}$.  Let $\mathrm{P} \xrightarrow{\mathrm{p}} \on{Im}(\mathrm{f}) \xrightarrow{\iota}\mathrm{Q}$ be the image decomposition of $\mathrm{f}$. Let $\mathrm{P}' \in \csym{C}\!\on{-proj}$.

 By separate exactness of $- \otimes -$, the morphism $\mathrm{p} \otimes \mathrm{P}'$ is epi, and $\iota \otimes \mathrm{P}'$ is mono. Since $\mathrm{P}'$ is projective, by \ref{EGNOHelpsAgain}, so is $\mathrm{Q} \otimes \mathrm{P}$, and thus $\mathrm{p} \otimes \mathrm{P}'$ is a split epimorphism.
 By Corollary~\ref{SelfInjective}, $\mathrm{P}'$ is also injective, and thus, following \ref{EGNOHelpsAgain}, so is $\mathrm{P} \otimes \mathrm{P}'$, so $\iota \otimes \mathrm{P}'$ is a split monomorphism.
 Choosing a section $\mathrm{i}: \on{Im}(\mathrm{f}) \otimes \mathrm{P}' \hookrightarrow \mathrm{P}\otimes \mathrm{P}'$ for $\mathrm{p} \otimes \mathrm{P}'$, and a retraction $\pi: \mathrm{Q} \otimes \mathrm{P}' \twoheadrightarrow \on{Im}(\mathrm{f}) \otimes \mathrm{P}'$ to $\iota \otimes \mathrm{P}'$, we find that
 \[
  \pi \circ (\mathrm{f} \otimes \mathrm{P}') \circ \mathrm{i} = \pi \circ (\iota \otimes \mathrm{P}') \circ (\mathrm{p} \otimes \mathrm{P}') \circ \mathrm{i} = \on{id}_{\on{Im}(\mathrm{f}) \otimes \mathrm{P}'}
 \]
 Finally, $\mathrm{f} \otimes \mathrm{P}' \neq 0$, since $\mathrm{f} \neq 0$ and, by Proposition~\ref{FaithfulUnits}, $- \otimes \mathrm{P}'$ is faithful. Thus, $\on{Im}(\mathrm{f}) \otimes \mathrm{P}' \simeq \on{Im}(\mathrm{f} \otimes \mathrm{P}') \neq 0$, and we have found the identity morphism of a non-zero object in $\mathbf{I}$, which, using the first paragraph, yields $\mathbf{I} = \csym{C}\!\on{-proj}$.

 Thus, $\csym{C}\!\on{-proj}$ admits no non-trivial right $(\csym{C}\!\on{-proj})$-ideals. The proof that no non-trivial left ideals exist is similar.
\end{proof}

\begin{proposition}\label{FinTensor}
 Let $\csym{S}$ be a disimple quasi-fias category. Then $[\csym{S}^{\on{op}},\mathbf{vec}_{\Bbbk}]$ is a finite tensor category.
\end{proposition}

\begin{proof}
 Since $\csym{S}$ is disimple, both the $\csym{S}$-module categories ${}_{\ccf{S}}[\csym{S}^{\on{op}},\mathbf{vec}_{\Bbbk}]$ and $[\csym{S}^{\on{op}},\mathbf{vec}_{\Bbbk}]_{\ccf{S}}$ are simple transitive, and thus projectivizing, by Theorem~\ref{MazorchukMiemietzLemma}.

 Having established this fact, one can repeat verbatim the proofs of \cite[Proposition~2.34(ii)]{BEO} and \cite[Proposition~2.35]{BEO}, and show that $[\csym{S}^{\on{op}},\mathbf{vec}_{\Bbbk}]$ is a multiring category. Since every object of $[\csym{S}^{\on{op}},\mathbf{vec}_{\Bbbk}]$ is the cokernel of a morphism in $\csym{S} \simeq [\csym{S}^{\on{op}},\mathbf{vec}_{\Bbbk}]\!\on{-proj}$, we find that $[\csym{S}^{\on{op}},\mathbf{vec}_{\Bbbk}]$ is rigid, by \cite[Proposition~2.36]{BEO}.
 Thus, it is a finite multitensor category. By \cite[Remark~4.3.4]{EGNO}, there are finite tensor categories $\widehat{\csym{S}}_{ij}$, for $i,j=1,\ldots,n$, for some positive integer $n$, such that $[\csym{S}^{\on{op}},\mathbf{vec}_{\Bbbk}] \simeq \bigoplus_{i,j} \widehat{\csym{S}}_{ij}$.
 Further, for a fixed $j$, the category $\bigoplus_{i} \widehat{\csym{S}}_{ij}$ defines a thick left $[\csym{S}^{\on{op}},\mathbf{vec}_{\Bbbk}]$-ideal in ${}_{[\ccf{S}^{\on{op}},\mathbf{vec}_{\Bbbk}]}[\csym{S}^{\on{op}},\mathbf{vec}_{\Bbbk}]$. Thus, $\big(\bigoplus_{i} \widehat{\csym{S}}_{ij}\big)\!\on{-proj} \simeq \bigoplus_{i} \widehat{\csym{S}}_{ij}\big)\!\on{-proj}$ defines a thick left $\csym{S}$-ideal in ${}_{\ccf{S}}\csym{S}$. Thus, since $\csym{S}$ is disimple, we must have $n=1$, and therefore $[\csym{S}^{\on{op}},\mathbf{vec}_{\Bbbk}]$ is a finite tensor category.
\end{proof}

We have thus established that  a finite abelian monoidal category $\csym{C}$ is a tensor category if and only if $\csym{C}\!\on{-proj}$ is disimple quasi-fias. Conversely, the abelianization a finitary semigroup $\csym{S}$ is a finite tensor category if and only if $\csym{S}$ is disimple quasi-fias. Indeed, we obtain the following result:
\begin{theorem}\label{FinitaryCharacterization}
 There is a bijection
   \begin{equation}\label{FinalBijection}
  \begin{aligned}
 \Big\{
  \begin{aligned}
  &\text{Disimple quasi-fias semigroup categories} 
  \end{aligned}\Big\}
/&\simeq
 \longleftrightarrow
  \Big\{
  \begin{aligned}
  \text{Finite tensor categories}
  \end{aligned}
\Big\}/\simeq
 \\
 &\csym{S} \longmapsto  [\csym{S}^{\on{op}},\mathbf{vec}_{\Bbbk}] \\
 &\csym{C}\!\on{-proj} \longmapsfrom \csym{C}
 \end{aligned}
  \end{equation}

\end{theorem}

\begin{proof}
 Both maps are well-defined as a consequence of Proposition~\ref{ProjDis} and Proposition~\ref{FinTensor}. That they are mutually inverse bijections follows similarly to Theorem~\ref{AbelianizationCorrespondence}.
\end{proof}

\vspace{5mm}

\noindent
Department of Mathematics, Uppsala University, Box. 480, SE-75106, Uppsala, SWEDEN, \\
email: {\tt mateusz.stroinski\symbol{64}math.uu.se}

\end{document}